%% file: zfp_stability_paper.tex
\documentclass[review,onefignum,onetabnum]{siamart171218}

\input{packages.tex}

\input{definitions.tex}

\headers{Error Analysis of ZFP Compression for Floating-Point Data}{J. Diffenderfer, A. Fox, J. Hittinger, G. Sanders, P. Lindstrom}

\title{Error Analysis of ZFP Compression for
  Floating-Point Data\thanks{Submitted to the editors 25 January
    2018.\funding{This work was performed under the auspices 
    of the U.S. Department of Energy by Lawrence Livermore National
    Laboratory under Contract DE-AC52-07NA27344 and was supported by the
    LLNL-LDRD Program under Project No. 17-SI-004, LLNL-JRNL-744818-DRAFT.}}}
\author{James Diffenderfer\thanks{The University of Florida,
    Gainesville, FL (\email{jdiffen1@ufl.edu})}
  \and Alyson Fox\thanks{Lawrence Livermore National Laboratory,
    Livermore, CA (\email{fox33@llnl.gov})}
\and Jeffrey Hittinger\thanks{Lawrence Livermore National Laboratory,
    Livermore, CA (\email{hittinger1@llnl.gov})}
\and Geoffrey Sanders\thanks{Lawrence Livermore National Laboratory,
    Livermore, CA (\email{sanders29@llnl.gov})} 
\and Peter Lindstrom\thanks{Lawrence Livermore National Laboratory,
    Livermore, CA (\email{lindstrom2@llnl.gov})}}

\ifpdf
\hypersetup{
  pdftitle={Error Analysis of ZFP Compression for Floating-Point Data},
  pdfauthor={J. Diffenderfer, A. Fox, J. Hittinger, G. Sanders, P. Lindstrom}
}
\fi

\ifpdf
\hypersetup{
	pdftitle={Error Analysis of ZFP Compression for Floating-Point Data},
	pdfauthor={J. Diffenderfer, A. Fox, J. Hittinger, G. Sanders, P. Lindstrom}
}
\fi
\RequirePackage[normalem]{ulem} %DIF PREAMBLE
\RequirePackage{color}\definecolor{RED}{rgb}{1,0,0}\definecolor{BLUE}{rgb}{0,0,1} %DIF PREAMBLE
\providecommand{\DIFadd}[1]{{\protect\color{blue}\uwave{#1}}} %DIF PREAMBLE
\providecommand{\DIFaddbegin}{} %DIF PREAMBLE
\providecommand{\DIFaddend}{} %DIF PREAMBLE
\providecommand{\DIFdelbegin}{} %DIF PREAMBLE
\providecommand{\DIFdelend}{} %DIF PREAMBLE
\providecommand{\DIFaddFL}[1]{\DIFadd{#1}} %DIF PREAMBLE
\begin{document}

\maketitle

% REQUIRED
\begin{abstract}
Compression of floating-point data will play an important role in
high-performance computing as data bandwidth and storage become dominant
costs.  Lossy compression of floating-point data is powerful, but
theoretical results are needed to bound its errors when used to store
look-up tables, simulation results, or even the solution state during
the computation.  \black{In this paper, we analyze the round-off error
introduced by ZFP, a %state-of-the-art 
lossy compression algorithm.} The stopping criteria for ZFP depends on the compression mode specified
by the user; either fixed rate, fixed accuracy, or fixed precision
\cite{zfp-doc}. While most of our discussion is focused on the fixed precision mode
of ZFP, we establish a bound on the error introduced by all three compression 
modes. In order to tightly
capture the error, we first introduce a vector space that allows us to
work with binary 
representations of components. Under this vector space, we define
operators that implement each step of the ZFP compression and
decompression to establish a bound on the error caused by
ZFP.
%Furthermore, ZFP
%compression is capable of inline compression and decompression. Only the
%data needed for the update of a time-evolving iterative method, at a
%particular value, needs to be decompressed, while 
%the remaining data can remain in a compressed state. However, in a
%numerical simulation the solution state already contains traditional
%errors (floating-point round-off, truncation error, and discretization
%error.)  The information that is lost during ZFP could represent the
%traditional errors, however, any additional error caused by ZFP will
%contaminate the current iterate.
%Under some basic assumptions, we show
%that the error caused by ZFP, used inline on the solution state in a
%simulation, is bounded as it propagates through the simulation by a
%discrete advancement operator. 
To conclude, numerical tests are provided
to demonstrate the accuracy of the established bounds.  
\end{abstract}

% REQUIRED

% REQUIRED

% Body
\input{introduction}

\input{compression}
\input{notation}
\input{analysis}
\input{bounds}
\input{results}

\input{conclusions}

\input{appendix}

\bibliographystyle{siamplain}
\bibliography{compression}

\end{document}

%% file: packages.tex
\usepackage{enumerate}
\usepackage{amsmath}
\usepackage{mathrsfs}
\usepackage{xcolor}	
\usepackage{amssymb}
\usepackage{bm}
\usepackage{graphicx}
\usepackage{subcaption}
\usepackage{adjustbox}
\usepackage{soul}
\usepackage{listings}
\usepackage{array}  
\usepackage{mathtools}

%% file: definitions.tex
\usepackage{mdframed}
\newmdenv[linecolor=green!50!black, fontcolor=green!50!black, backgroundcolor=green!20, linewidth=2pt, roundcorner=10pt]{gnote}

\definecolor{amber}{rgb}{1.0, 0.49, 0.0}
\definecolor{darkpastelgreen}{rgb}{0.01, 0.75, 0.24}
\definecolor{darkred}{rgb}{0.64, 0.0, 0.0}
\definecolor{amberhl}{rgb}{1.0, 0.75, 0.0}

\newcommand{\red}[1]{\textcolor{black}{#1}}
\newcommand{\black}[1]{\textcolor{black}{#1}}

\newmdenv[linecolor=blue!50!black, fontcolor=blue!50!black, backgroundcolor=blue!20, linewidth=2pt, roundcorner=10pt]{anote}

\newtheorem{prop}[theorem]{Proposition}
\newcolumntype{L}{>{$}l<{$}} 

\newcommand{\ba}{{\bm a}}
\newcommand{\bb}{{\bm b}}
\newcommand{\bc}{{\bm c}}
\newcommand{\bd}{{\bm d}}

\newcommand{\bw}{{\bm w}}
\newcommand{\bx}{{\bm x}}
\newcommand{\by}{{\bm y}}
\newcommand{\bz}{{\bm z}}

\newcommand{\bfz}{{\bm 0}}

\newcommand{\cB}{\mathcal{B}}

\newcommand{\cI}{\mathcal{I}}

\newcommand{\cL}{L}

\newcommand{\cN}{\mathcal{N}}

\newcommand{\cP}{\mathcal{P}}

\newcommand{\cR}{\mathcal{R}}
\newcommand{\cS}{\mathcal{S}}

\newcommand{\secref}[1]{Section~\ref{#1}}

%% file: introduction.tex
\section {Introduction}
\label{sec:introduction}
For several reasons, the trade-offs to obtain high performance in computing
have shifted.  Traditionally, the emphasis on algorithmic complexity in numerical
computation has focused on operation counts, which was justifiable when
processor clock rates were increasing and memory was cheap and
plentiful.  With the end of Dennard
scaling~\cite{752522}, clock speeds have 
frozen (or even reduced), so more capability, in terms of
FLOPs, is now being obtained by adding more processing units~\cite{BorkarChien2011}.
Simultaneously, the ubiquity of hand-held devices and the power requirements of
extreme-scale supercomputers is encouraging a shift to lower-power
processors and co-processors.

Unfortunately, advances in memory and memory bandwidth are not
increasing apace with the advances in processors.  Thus, the memory per core
and the bandwidth per core are decreasing as the number of processing
units increases~\cite{Xcutting2010,XDMAV2011}.  The on-node cost of data motion (cache and
main memory accesses), both in time and
power, is increasingly the limiting factor in many
calculations~\cite{Xcutting2010,Williams:2009:RIV:1498765.1498785}.
Since off-node and I/O data motion have 
historically been orders of magnitude slower than on-node data motion,
the movement of data anywhere on a computer system must now be seriously
considered as the leading-order cost. 

An obvious approach to address this challenge would be to consider data
compression techniques.  Indeed, lossless data compression is routinely used in
network communications. However, for floating-point data typical of the
scientific calculations done on high performance computers, standard
lossless compression techniques such as Lempel--Ziv~\cite{1055714,1055934},
DEFLATE~\cite{deflate,katz91}, Lempel–-Ziv–-Welch~\cite{1659158}, fpzip~\cite{fpzip} and other variants,
which reproduce the original data with 
no degradation, struggle to produce significant compression
rates~\cite{Ratanaworabhan:2006:FLC:1126009.1126035,fpzip}. 
Lossy compression algorithms for floating-point data, e.g.,
SZ~\cite{sz} and ZFP~\cite{zfp}, allow an inexact approximation of
the original data to be reconstructed from the compressed data.  Lossy 
data compression typically produces a much higher rate of data reduction
than lossless compression at the cost of introducing additional
approximation error into the data.

Lossy floating-point compression may be a useful tool in reducing data
motion costs, particularly if there is a schema that allows for
progressive (cf. global) decompression of data on demand.  Certainly,
for storage (e.g., tabular data) and I/O operations (data and restart
files), there may be much to gain by using lossy compression provided
that the data retain sufficient accuracy for the intended purposes.  We
propose that, in addition, solution state data in a simulation could be
stored in a compressed state and be decompressed, operated on, and
recompressed in a lossy way inline during each time step or iteration of a numerical
algorithm.  Numerical simulation is fundamentally about approximation,
and the solution state already contains truncation, iteration, and other
roundoff errors. However, the repeated application of compression and
decompression does generate an additional error, and it must be shown
that these lossy compression errors can be bounded to prove that such a
process is stable.

As a first step towards this goal, we consider the ZFP lossy compression
algorithm and develop an approach to analyze and bound the error resulting from lossy
compression and decompression. \black{While recent works have provided empirical studies of ZFP  and other lossy compression algorithms on real-world 
	data sets \mbox{\cite{gmd-9-4381-2016, Laney:2013:AED:2503210.2503283, Lindstrom2017}}, 
	this paper establishes the first closed form expression for bounds on the error introduced by ZFP.
%ZFP is the state-of-the-art in floating-point compression, so it is a natural algorithm to analyze. 
It is expected that our} approach can be generalized to other algorithms involving the manipulation of components represented using bits.  ZFP, which operates on blocks of $4^d$ values, can encode and truncate data using one of
three modes:  fixed rate, fixed accuracy, or fixed precision. The
fixed rate mode compresses a block to a fixed number of bits, the
fixed precision compresses to a variable number of bits while retaining
a fixed number of bit planes, and fixed accuracy mode compresses a block
with relation to the tolerated maximum error.  The goal of this paper is
to provide an error analysis for the fixed precision mode, as it
is the simplest to represent algebraically. However, as a result of the analysis of
the fixed precision mode we are able to develop bounds on the error
introduced by the fixed accuracy and fixed rate compression modes.

The remainder of this paper is structured as follows.  In the next
section, we describe the ZFP compression algorithm.
In~\secref{sec:notation}, we introduce the notation, definitions, and
lemmas that we will use in~\secref{sec:analysis} to prove bounds on
the error introduced at each stage of the ZFP compression algorithm.
In~\secref{sec:bounds}, we derive the error bounds for  the fixed precision mode for the composite
compression and decompression action, as well as an error bound for both fixed accuracy and fixed rate modes. Finally, we demonstrate
the validity of these results numerically in~\secref{sec:results}.

%% file: compression.tex
\section{ZFP Data Compression}
\label{sec:zfp}
We first provide a brief overview of the current ZFP compression algorithm. Further details of the implementation of ZFP can be found in \cite{zfp} with modifications described in the software documentation \cite{zfp-doc}. \black{For clarity purposes, a small example of ZFP is provided in Appendix \ref{sec:appendixa}. }
 \begin{itemize}
\item[\it Step 1:]The $d$-dimensional array is partitioned into arrays of dimension $4^d$, called \emph{blocks}. A 2-d example is depicted in Figure \ref{fig:cut4dblock}. If the $d$-dimensional array cannot be partitioned exactly into blocks, then the boundary of the $d$-dimensional array is padded until an exact partition is possible.
\begin{figure}
    \centering
        \includegraphics[width=.45\textwidth]{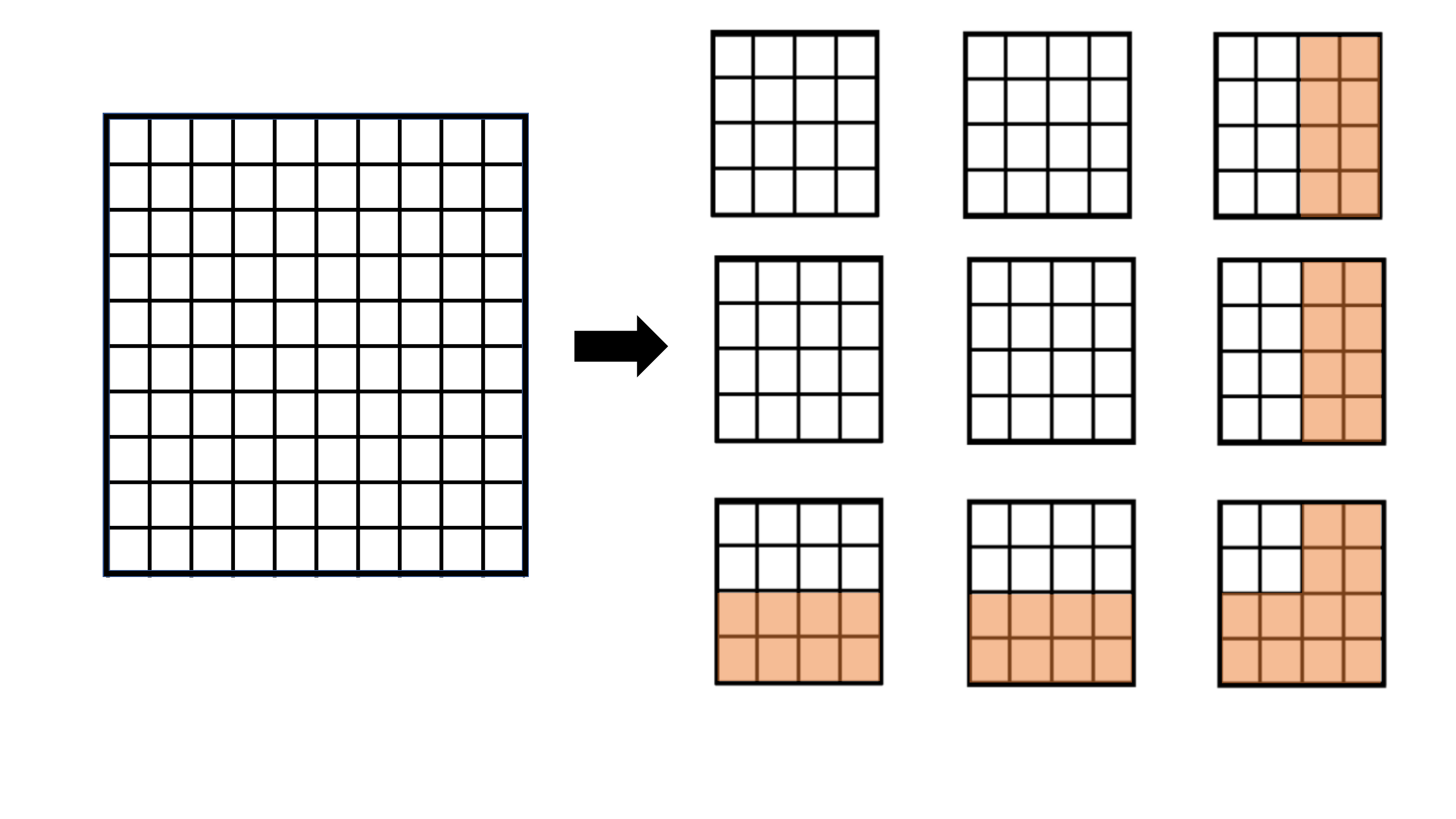}
    \caption{Deconstruction of a $10 \times 10$ 2-dimensional array into independent $4\times4$ blocks. If the data is not divisible by 4 the data at the boundaries is padded (shown in orange).}
    \label{fig:cut4dblock}
\end{figure}
\begin{figure}
    \centering
        \includegraphics[width=.8\textwidth]{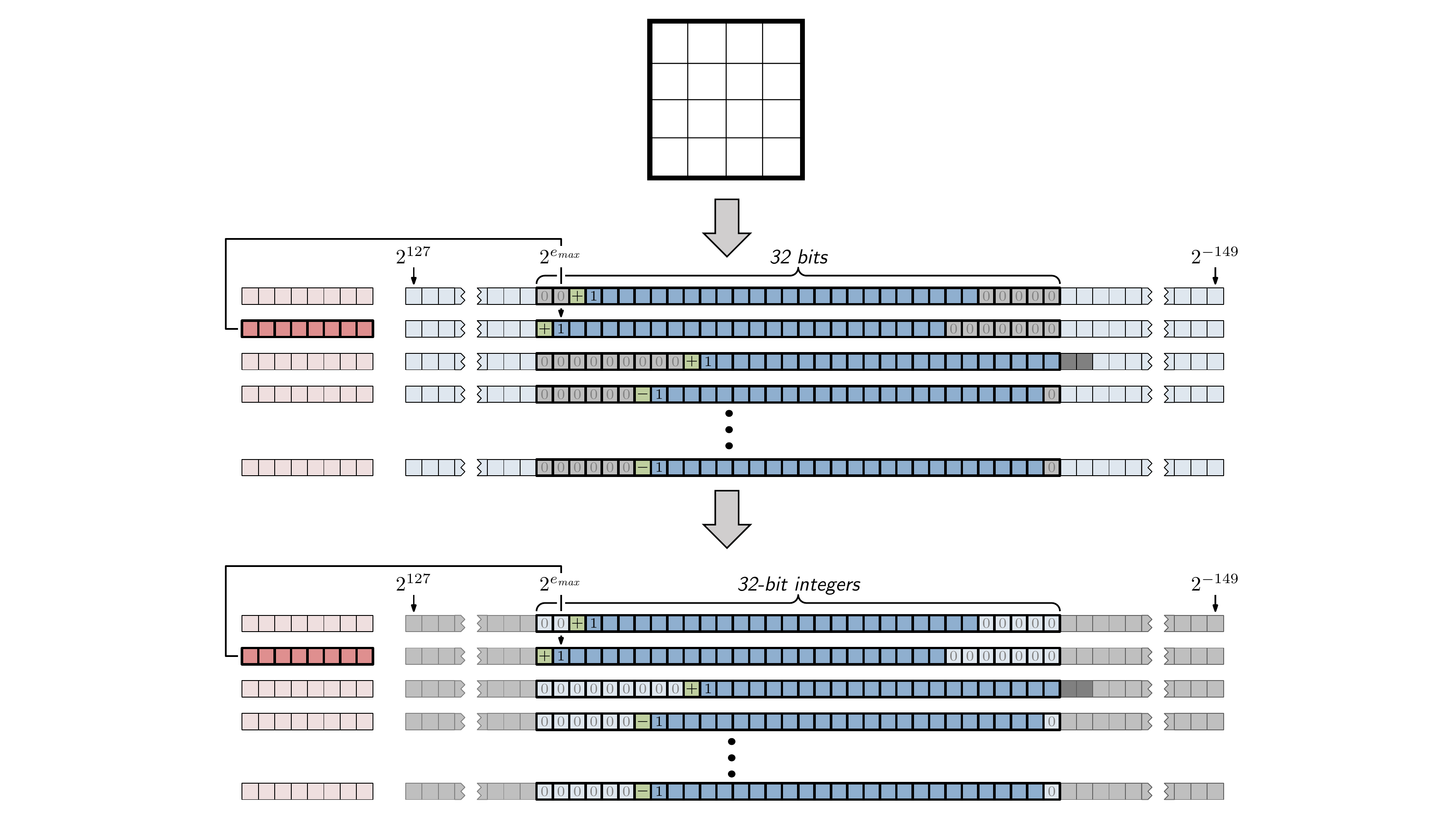}
    \caption{Floating-point bit representation in single precision converted to a block-floating-point representation and its corresponding signed integers. Note that, depending on the relative disparity of the 16 numbers, some truncation may occur for the numbers of the smallest magnitude. }
   \label{fig:bfprepresentation}
\end{figure}
\item[\it Step 2:] The floating-point values in each block are converted
  to a block-floating-point representation using a common exponent for
  each block \cite{MitrablockRoundingError} and integers in two's complement
  format. The block-floating point representation is then shifted and
  rounded to $4^d$ signed integers as seen in Figure
  \ref{fig:bfprepresentation}. 
\item[\it Step 3:]The integers are decorrelated using a custom, high-speed, near orthogonal transform that is similar to the discrete cosine transform. \black{The idea is that continuous fields tend to exhibit autocorrelation which can be viewed as redundant information.  The decorrelating transform removes these redundancies via a change in basis resulting in a ``sparser" representation with smaller magnitude coefficients.  The many leading zeros in the small coefficients offer an opportunity for compression (see Section \ref{Step3Sec} for details.)}
\item[\it Step 4:] \black{Coefficient magnitude tends to correlate (inversely) with sequency. A 2-d example of total sequency can be seen in Figure \ref{fig:totalsequency}. Sequency ordering is done to place the coefficients roughly in order of decreasing magnitude, which tends to group ones together and zeros together in each bit plane.  This facilitates compression as often small coefficients tend to share leading zeros.} 
\begin{figure}
    \centering
        \includegraphics[width=.5\textwidth]{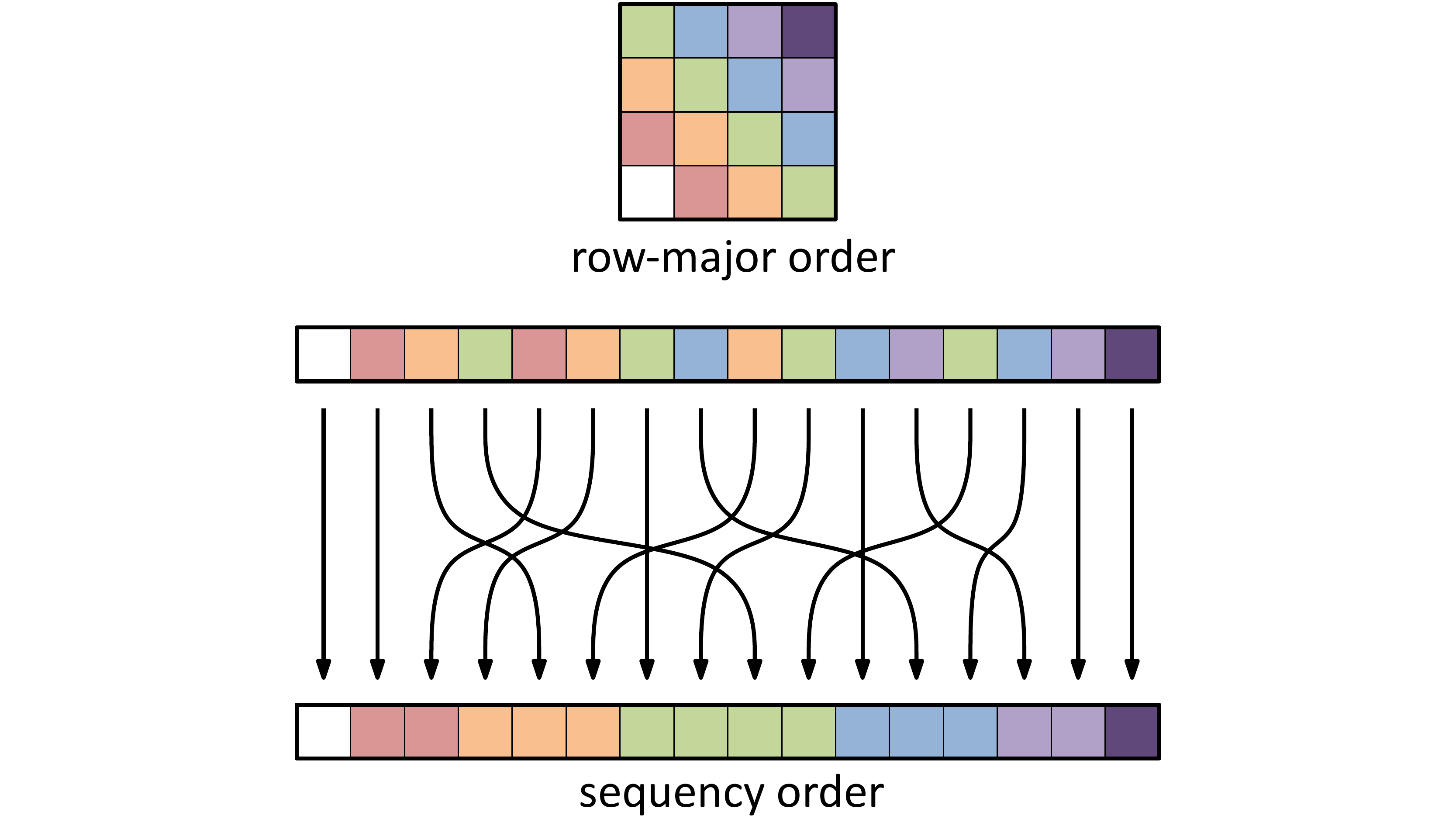}
    \caption{Total sequency ordering for a 2-dimensional array, which groups the diagonal elements together. }
    \label{fig:totalsequency}
\end{figure}
\item[\it Step 5:] \black{The sign bit is typically the left most bit in any traditional binary representation, which does not provide any useful information until the leading one-bit is encountered, \black{i.e., the transition from 0 to 1 (for positive values) or from 1 to 0 (for negative values)}. However, the first nonzero bit encountered in a negabinary representation immediately informs of the sign and magnitude. For example, if the leftmost one-bit in negabinary is at position $e$, then the magnitude of the number is in $2^e\cdot[1/3, 4/3]$.  If $e$ is even, then the number is positive; otherwise it is negative. Thus, the two's complement signed integers (the standard integer representation) are converted to their negabinary representation \cite{Knuth}. The negabinary representation also ensures that the error caused by the remaining steps is mostly centered around zero with a slight bias depending on the index of truncated bit-plane. } 
\item[\it Step 6:] The bits that represent the list of $4^d$ integers
  are transposed so that they are ordered by bit plane, from most to
  least significant bit, instead of by coefficient.  
\item[\it Step 7:] Each bit plane is compressed losslessly using
  embedded coding, which exploits the property that the transform
  coefficients tend to have many leading zeros. \black{The idea is to encode groups of zero-bits together using a single bit to indicate that the whole group consists of zeros.} As this step is
  lossless, the encoding details are omitted.  
\item[\it Step 8:]The embedded coder emits one bit at a time until
  stopping \black{criterion} are satisfied. The exact stopping criteria is dependent on the mode of ZFP compression: either fixed rate, fixed precision, or
  fixed accuracy.   
\end{itemize}

Most of our discussion and error analysis in Section~\ref{sec:analysis}
will focus on Steps 2, 3, and 8 as these steps are likely to introduce
additional round-off error during compression and, in some cases, during
decompression. Our approach for bounding the round off error introduced
by ZFP compression and decomporession is to define an operator for each
step of the algorithm, compose these operators to form a ZFP compression
and a ZFP decompression operator, compute the value returned after
compressing and decompressing an arbitrary input, and finally compare
this value to the value of the original input. The next section will
focus on introducing notation and definitions that will be useful in
defining mathematical operators for each step of the ZFP compression
algorithm.   
%For reference, Table \ref{table:notation} briefly describes the notation used and the section and/or equation in which it is defined. 
\renewcommand\arraystretch{1.2}
\begin{table}[htp!]
\caption{Notation Table} 
\begin{center}
\begin{adjustbox}{width=1\textwidth}
\begin{tabular}{|c l c|}
\hline 
Symbol & Description & Location \\
\hline 
\hline

$d$ & dimension of the input data & \S \ref{Step1Sec} \\
$k$ & the number of IEEE mantissa bits, including the leading one-bit& \S \ref{Step2Sec}\\
$q$ & the number of consecutive bits used to represent an element in the block-floating point transform &  \S \ref{Step2Sec} \\
$\beta$ & number of bit planes kept in Step 8 &  \S  \ref{Step8Sec}  \\
\hline
$\cI$ & active bit set & \S \ref{sec:notation}  \\
%$\cB$ & infinite bit space & \S \ref{sec:bitvectorSpace} \\
$\cB^n$ & infinite binary vector space & \S \ref{sec:bitvectorSpace} \\
$\cB^n_k$ & subset of $\cB^n$ with finite active bit set & \S \ref{sec:bitvectorSpace} \\
$\cN^n$ & infinite negabinary vector space & \S  \ref{sec:bitvectorSpace}\\
$\cN_k^n$ & subset of $\cN^n$ with finite active bit set & \S \ref{sec:bitvectorSpace}\\
$\bf{0}_\cB$,$\bf{0}_\cN$  & additive identity in $\cB^n$ and $\cN^n$, respectively & \S \ref{sec:bitvectorSpace}\\ 
$\bf{1}_\cB $, $\bf{1}_\cN$  & multiplicative identity in $\cB^n$ and $\cN^n$, respectively & \S \ref{sec:bitvectorSpace}\\
$\| \cdot \|_{\cB, p}$, $\| \cdot \|_{\cN, p}$  & $p$-norm with respect to $\cB^n$ and $\cN^n$, respectively & \S \ref{NormsSubsection} \\

\hline
$f_\cB$, $f_\cB^{-1}$, $f_\cN$, $f_\cN^{-1}$ & bijective maps from $\cB \rightarrow \mathbb{R}$, $\mathbb{R} \rightarrow \cB$, $\cN \rightarrow \mathbb{R}$ and $\mathbb{R} \rightarrow \cN$, respectively & Eqn. (\ref{eqn:f2})\\
\black{$F_\cB$,$F_\cB^{-1}$, $F_{\cN}$,$F_\cN^{-1}$}  & \black{bijective maps from $\cB^n \rightarrow \mathbb{R}^n$, $\mathbb{R}^n \rightarrow \cB^n$, $\cN^n \rightarrow \mathbb{R}^n$ and $\mathbb{R}^n \rightarrow \cN^n$, respectively} & \S \ref{sec:bitvectorSpace}\\
$\oplus $ &bit vector addition  & Lma. \ref{FieldLemma}\\
$\odot$ &bit vector multiplication & Lma. \ref{FieldLemma}\\
\hline

$s_l$, $S_l$ & shift operator on $\cB$ and $\cB^n$, respectively & \S \ref{ShiftSubsection} \\
$t_{\cS_k},T_{\cS_k} $ & truncation operator with respect to the set $\cS_k$ & \ref{DefCOperators} \\
{$r $} & rounding operator for two's complement representation & \S \ref{Step3Sec} \\
$e_{min}$, $e_{max}$ & min and max exponent of the floating-point representation of the block &  Def.  \ref{def:emaxemin}\\
$\epsilon_m$ & constant $\epsilon_m := 2^{1-m}$, with $m \in \mathbb{N}$ &  \S \ref{sec:bitvectorSpace} \\
\hline
$L$, $L_d$ & one and $d$-dimension forward decorrelating linear transform &  Eqn.   \ref{eqn:T} \\
$\tilde{L}$, $\tilde{L}_d$ & floating-point arithmetic approximation of $L$ and $L_d$ &  \S \ref{Step3Sec}   \\
$L_d^{-1} $,  $\tilde{L}^{-1} _d$ &  $d$-dimension backward decorrelating linear transform and the floating-point arithmetic approximation of $L^{-1}$&  \S \ref{Step3Sec} \\
\hline
$C_k$, $\tilde{C_k}$ & lossless/lossy operator for Step $k$ of ZFP compression & \S \ref{sec:analysis} \\
%$C_2$, $\tilde{C_2}$ & lossless/lossy compression operator of Step 2 (block-floating-point) &\S \ref{Step2Sec}\\
%$C_3$, $\tilde{C_3}$ & lossless/lossy compression operator of Step 3 (decorrelating transform)&  \S \ref{Step3Sec} \\
%$C_8$, $\tilde{C_8}$ &lossless/lossy compression operator of Step 8 (bit-plane approximation)&  \S  \ref{Step78Sec}  \\
%\hline
$D_k$, $\tilde{D_k}$ & lossless/lossy operator for Step $k$ of ZFP decompression & \S \ref{sec:analysis} \\
%$D_3$, $\tilde{D_3}$ & lossless/lossy decompression operator of Step 3 (decorrelating transform)&   \S \ref{Step3Sec} \\
\hline
\end{tabular}
\end{adjustbox}
\end{center}

\label{table:notation}
\end{table}

%%%%%%%%%%%%%%%%%%%%%%%%%%%%%%%%%%%%%%%%%%%%%%%%%%%%%%%%%%%%%%%%%%%%

%% file: notation.tex
\section{Preliminary Notations, Definitions, and Lemmas}
\label{sec:notation}
%\subsection{Definition of the Two's Complement Bit-Vector Space}
\label{IntroSec:seqspace} 
\black{As noted in Section \ref{sec:zfp},} many of the steps of ZFP are defined by direct
manipulation of the bits used to represent each component of the
input. While we could attempt to define operators that imitate the steps
of ZFP over the real vector space, it would be more straightforward to
work with the bitwise representation that is manipulated at each step of
the ZFP compression algorithm. Hence, in order to define operators for
each step of ZFP as actions on the bitwise representation of each
component, we construct vector spaces under which the components
correspond to binary or negabinary representations of the real
numbers. Accordingly, let $\mathbb{B} = \{ 0, 1 \}$ and define  
\begin{align}
\mathcal{C} := \left\{ \{ c_i \}_{i = -\infty}^{\infty} : c_i \in \mathbb{B} \ \text{for all} \ i \in \mathbb{Z} \right\} \label{eqn:infinitebitvector}.
\end{align}
For \black{$c \in \mathcal{C}$,} we define the \emph{active bit set of
  $c$} by $\mathcal{I} (c) := \{ i \in \mathbb{Z} : c_i = 1 \}.$
Additionally, we define the following operators on $\mathcal{C}$ that
will be used as building blocks for defining each step of ZFP. 
\begin{definition}\label{DefCOperators}
Let $\mathcal{S} \subseteq \mathbb{Z}$. The \emph{truncation operator},
$t_{\mathcal{S}} : \mathcal{C} \to \mathcal{C}$,  is defined by  
\begin{align*}
t_{\mathcal{S}} (c)_i = \left\{ \begin{array}{ccc} 
                c_i &: & i \in \mathcal{S} \\
                0 &: & i \not\in \mathcal{S} \\
                \end{array} \right., \ \ \ \text{for all} \ c \in \mathcal{C} \ \text{and all} \ i \in \mathbb{Z}.
\end{align*}
Let $\ell \in \mathbb{Z}$. The \emph{shift operator}, $s_{\ell} :
\mathcal{C} \to \mathcal{C}$, is defined by 
\begin{align*}
s_{\ell} (c)_i = c_{i + \ell}, \ \ \ \text{for all} \ c \in \mathcal{C} \ \text{and all} \ i \in \mathbb{Z}.
\end{align*}
\end{definition}
\noindent From these definitions, it follows that $t_{\mathcal{S}}$ is a
nonlinear operator and \black{$s_{\ell}$ is a linear operator.} \black{These operators
can be extended to operators on $\mathcal{C}^n$ 
by defining
$T_{\mathcal{S}} : \mathcal{C}^n \to \mathcal{C}^n$ and $S_{\ell} : \mathcal{C}^n \to \mathcal{C}^n$ by 
\begin{align*}
T_{\mathcal{S}} (\bm{c}) = \begin{bmatrix}
\ t_{\mathcal{S}} ({\bc}_1) \ \\
\ \vdots \ \\
\ t_{\mathcal{S}} ({\bc}_n) \ \
\end{bmatrix} \ \ \ \ \ \ \text{and} \ \ \ \ \ \ S_{\ell} (\bc) = \begin{bmatrix}
\ s_{\ell} ({\bc}_1) \ \\
\ \vdots \ \\
\ s_{\ell} ({\bc}_n) \ \
\end{bmatrix}\text{, for all } \bc\in \mathcal{C}^n.
 \end{align*}}
\noindent\black{For clarity, note that the components of $\bm{c}$ are each binary sequences as $\bc_i \in \mathcal{C}$, for $1 \leq i \leq n$.} Additionally, note that $S_{\ell}$ is invertible with $S_{\ell}^{-1} : \mathcal{C}^n \to \mathcal{C}^n$ given by $S_{\ell}^{-1} = S_{-\ell}$. 

\subsection{Defining Signed Binary and Negabinary Bit-Vector Spaces}
\label{sec:bitvectorSpace}
Let $x \in \mathbb{R}$  be given. Then there exist $c, d \in \mathcal{C}$ and $p \in \mathbb{B}$ such that $x$ can be represented in signed binary and negabinary as 
\begin{align}
{\text{Signed Binary: }} x = (-1)^p \sum_{i = - \infty}^{\infty} c_i 2^i \ \ \ \text{ and}  \ \ \ {\text{ Negabinary: }} x = \sum_{i = - \infty}^{\infty} d_i (-2)^i. \label{CrepX}
\end{align}
As such, there exist subsets $\mathcal{A}$ and $\mathcal{N}$ of
$\mathcal{C}$ such that, for each $x \in \mathbb{R}$, there exist unique
elements $c \in \mathcal{A}$, $p \in \mathbb{B}$, and $d \in
\mathcal{N}$ such that $x$ can be represented in the binary and
negabinary form in (\ref{CrepX}) using $c$, $p$, and $d$,
respectively. In particular, we choose $\mathcal{A}$ and $\mathcal{N}$
such that $\mathcal{I} (c)$ and $\mathcal{I} (d)$ are finite whenever
possible. \black{This choice is made so that elements can be represented using finitely many nonzero bits.} Now define
$0_{\mathcal{C}}, 1_{\mathcal{C}} \in \mathcal{C}$ to be the elements
satisfying $\mathcal{I} \left( 0_{\mathcal{C}} \right) = \emptyset$ and
$\mathcal{I} \left( 1_{\mathcal{C}} \right) = \{ 0 \}$. It follows from
our choice of $\mathcal{A}$ and $\mathcal{N}$ that $\left\{
0_{\mathcal{C}}, 1_{\mathcal{C}} \right\} \subset \mathcal{A} \cap
\mathcal{N}$.  

Defining $\mathcal{B} := \{ (p, a) \in  \mathbb{B} \times \mathcal{A} :
(p, a) \neq ( 1, 0_{\mathcal{C}} ) \}$, we have that, for each $x \in
\mathbb{R}$, there exists a unique $b = (p, a) \in \mathcal{B}$ such that
$x = (-1)^p \sum_{i = - \infty}^{\infty} a_i 2^i$. Additionally, it is
clear from our choice of $\mathcal{N}$ that, for each $x \in \mathbb{R}$,
there exists a unique $d \in \mathcal{N}$ such that $x = \sum_{i = -
  \infty}^{\infty} d_i (-2)^i$. We now define $f_\cB : \mathcal{B} \to
\mathbb{R}$ by 
\begin{equation}
\label{eqn:f2}
f_\cB (b) = (-1)^{{p}} \sum_{i = -\infty}^{\infty} a_i 2^i, \ \ \ \ \ \text{for all} \ b = ({p}, a) \in \mathcal{B},
\end{equation}
and $f_\cN : \mathcal{N} \to \mathbb{R}$ by 
\begin{equation}
\label{eqn:fN}
f_\cN (d) = \sum_{i = -\infty}^{\infty} d_i (-2)^i, \ \ \ \ \ \text{for all} \ d \in \mathcal{N}.
\end{equation}
\black{By our choice of $\mathcal{B}$ and $\mathcal{N}$, $f_\cB$ and $f_\cN$ are bijections and with inverses denoted by $f_\cB^{-1} : \mathbb{R} \to \mathcal{B}$ and $f_\cN^{-1} : \mathbb{R} \to \mathcal{N}$, respectively. We now define binary operators $\oplus_{\mathcal{B}} : \mathcal{B} \times \mathcal{B} \to \mathcal{B}$ and $\odot_{\mathcal{B}} : \mathcal{B} \times \mathcal{B} \to \mathcal{B}$ by}
\begin{align}
\alpha \oplus_{\mathcal{B}} \beta = f_\cB^{-1} \left( f_\cB( \alpha ) + f_\cB(\beta) \right) \ \ \ \ \ \ \text{and} \ \ \ \ \ \ \alpha \odot_{\mathcal{B}} \beta = f_\cB^{-1} \left( f_\cB(\alpha) \cdot f_\cB(\beta) \right) \label{oplustimesdef}
\end{align}
for all $\alpha, \beta \in \mathcal{B}$, where $+$ and $\cdot$ represent standard addition and multiplication in $\mathbb{R}$. Similarly, we can define $\oplus_{\mathcal{N}} : \mathcal{N} \times \mathcal{N} \to \mathcal{N}$ and $\odot_{\mathcal{N}} : \mathcal{N} \times \mathcal{N} \to \mathcal{N}$ by replacing all $\mathcal{B}$ with $\mathcal{N}$ in (\ref{oplustimesdef}). With these definitions in place, we now have the following result.
\begin{lemma}\label{FieldLemma}
\black{$(\mathcal{B}, \oplus_\mathcal{B}, \odot_\mathcal{B})$ and $(\mathcal{N}, \oplus_\mathcal{N}, \odot_\mathcal{N})$ are fields with additive and multiplicative identities $0_{\mathcal{B}} := (0,0_{\mathcal{C}})$ and $1_{\mathcal{B}} := (0, 1_{\mathcal{C}})$ and $0_{\mathcal{N}} := 0_{\mathcal{C}}$ and $1_{\mathcal{N}} := 1_{\mathcal{C}}$, respectively.}
\end{lemma} 

For the remainder of the discussion, the sign bit will be omitted from elements of $\mathcal{B}$ by letting $a$ represent the element $(0,a) \in \mathcal{B}$ and $- a$ represent $(1,a) \in \mathcal{B}$. Additionally, to simplify the notation in the following sections, we will write $+$ instead of $\oplus_{\mathcal{B}}$ or $\oplus_{\mathcal{N}}$ where the operation should be clear from the context in which it is used.

Note that $f_\cB$ and $f_\cN$ can be generalized to vector-valued functions by defining $F_\cB : \mathcal{B}^n \to \mathbb{R}^n$ and $F_\cN : \mathcal{N}^n \to \mathbb{R}^n$ as \black{ $F_\cB(\bm{a}) =[\ f_\cB({\ba}_1), \cdots, f_\cB({\ba}_n)]^t$ and $F_\cN(\bm{a}) =[\ f_\cN({\bd}_1), \cdots, f_\cN({\bd}_n)]^t$, where $\ba \in \cB^n$ and $\bd \in \cN^n$, respectively.}
%\red{\begin{equation}\DIFdelbegin %DIFDELCMD < \label{eqn:FN}
% \label{eqn:F2FN}
% F_\cB(\bm{a}) = \begin{bmatrix}
%\ f_\cB({\ba}_1) \ \ \\
%\ \vdots \ \ \\
%\ f_\cB({\ba}_n) \ \
%\end{bmatrix} \ \ \ \ \ \ \text{and} \ \ \ \ \ \ F_\cN ({\bd}) = \begin{bmatrix}
%\ f_\cN({\bd}_1) \ \ \\
%\ \vdots \ \ \\
%\ f_\cN({\bd}_n) \ \
%\end{bmatrix}, 
%\end{equation} 
%$\bm{a} = \left[ \ba_1, \ \cdots, \ \ba_n \right]^T \in \mathcal{B}^n$ and all $\bm{d} = \left[ \bd_1, \ \cdots, \ \bd_n \right]^T \in \mathcal{N}^n$.} 
By definition, $F_\cB$ and $F_\cN$ are invertible with inverses $F_\cB^{-1}$ and $F_\cN^{-1}$ defined by applying $f_\cB^{-1}$ and $f_\cN^{-1}$ componentwise, respectively. We will let $\bm{0}_{\mathcal{B}}$ and $\bm{0}_{\mathcal{N}}$ denote the the additive identity in $\mathcal{B}^n$ and $\mathcal{N}^n$, respectively, and $I_{\mathcal{B}} : \mathcal{B}^n \to \mathcal{B}^n$ and $I_{\mathcal{N}} : \mathcal{N}^n \to \mathcal{N}^n$ denote the identity map on $\mathcal{B}^n$ and $\mathcal{N}^n$, respectively.

\black{To imitate floating-point representations} %DIF > used in a computer, 
we define the following subsets of $\cB$ and $\cN$. Given $k \in \mathbb{N}$, 
\begin{align*}
\mathcal{B}_{k}  = \left\{ (p,a) \in \mathcal{B} : \mathcal{I} (a) \subseteq \{ i, i + 1, \ldots, {i + k -1} \} \ \text{for some} \ i \in \mathbb{Z} \right\}
\end{align*}
and
\begin{align*}
\mathcal{N}_{k} = \left\{ d \in \mathcal{N} : \mathcal{I} (d) \subseteq \{ i, i + 1, \ldots,  i + k -1 \} \ \text{for some} \ i \in \mathbb{Z} \right\}. 
\end{align*}
Here, $k$ represents the maximum number of nonzero bits allotted for each representation. For example, in $\cB_k$, $k$ represents the number of bits allotted for the \black{mantissa and $i$ indicates} the exponent in IEEE. It should be noted that $\mathcal{B}_{k}$ and $\mathcal{N}_k$ are subsets but \emph{not} subspaces of $\mathcal{B}$ and $\mathcal{N}$, respectively, as they do not satisfy the property of closure under $\oplus$ and $\odot$. As such, the analysis will take place in $\cB$, $\cN$, or $\mathbb{R}$ with the use of the truncation operator, $T_\cS$, to imitate working with finite precision elements. 

\subsection{Meaningful Norms on $\mathcal{B}^n$ and $\mathcal{N}^n$}\label{NormsSubsection} 
From Lemma \ref{FieldLemma}, it follows that $\mathcal{B}^n$ is a vector space under $+$. \black{Let $\| \cdot \|_p$ be the standard $p$-norm on $\mathbb{R}^n$. Accordingly, we define $\| \cdot \|_{\mathcal{B}, p} : \mathcal{B}^n \to [0, \infty)$ by 
\begin{align}
		\| \bm{a} \|_{\mathcal{B}, p} = \left\{ \begin{array}{ccc} 
		\big( \sum_{i = 1}^{n} \left| f_\cB(\bm{a}_i) \right|^{p} \big)^{1/p} &: & 1 \leq p < \infty \\
		\max_{1 \leq i \leq n} \left| f_\cB(\bm{a}_i) \right| &: & p = \infty \\
		\end{array} \right. \DIFaddend \label{BpNorm}
		\end{align}}
%\red{We now wish to determine norms $\| \cdot \|_{\mathcal{B}, p} : \mathcal{B}^n \to [0, \infty)$ under which $F_{\mathcal{B}} : \mathcal{B}^n \to \mathbb{R}^n$ is an isometry, that is, norms satisfying} $\| F_\cB (\bm{a}) \|_p = \| \bm{a} \|_{\mathcal{B}, p}$, for all $\bm{a} \in \mathcal{B}^n$, where $\| \cdot \|_p$ is the standard $p$-norm on $\mathbb{R}^n$. Accordingly, we define $\| \cdot \|_{\mathcal{B}, p} : \mathcal{B}^n \to [0, \infty)$ by 
%\red{\begin{align}
%\| \bm{a} \|_{\mathcal{B}, p} = \left\{ \begin{array}{ccc} 
%                \big( \sum_{i = 1}^{n} \left| f_\cB(\bm{a}_i) \right|^{p} \big)^{1/p} &: & 1 \leq p < \infty \\
%                \max_{1 \leq i \leq n} \left| f_\cB(\bm{a}_i) \right| &: & p = \infty \\
%                \end{array} \right. \DIFaddend \label{BpNorm}
%\end{align}}
 The following result is an immediate consequence from the definition of $\| \cdot \|_{\mathcal{B}, p}$.
\begin{lemma}
\label{BNormLemma}
For all $\bm{a} \in \mathcal{B}^n$, $\bm{x} \in \mathbb{R}^n$, and $1 \leq p \leq \infty$, $\| \cdot \|_{\mathcal{B}, p}$ is a norm satisfying 
\begin{align*}
\| F_\cB (\bm{a}) \|_p = \| \bm{a} \|_{\mathcal{B}, p} \ \ \ \text{and} \ \ \ \| \bm{x} \|_p = \| F_\cB^{-1} (\bm{x}) \|_{\mathcal{B}, p}.
\end{align*}
\end{lemma}
\noindent For any $1 \leq p \leq \infty$, we now have that $\mathcal{B}^n$ is a normed vector space with norm $\| \cdot \|_{\mathcal{B}, p}$. Additionally, we can define a norm for operators defined on $\mathcal{B}^n$.
\black{\begin{definition}
\label{BopNorm}
Let $m,n \in \mathbb{N}$ and $1 \leq p \leq \infty$. The induced p-norm on $\Phi : \mathcal{B}^n \to \mathcal{B}^m$ is given by
\begin{align*}
\| \Phi \|_{\mathcal{B}, p} = \sup \left\{ \frac{\| \Phi (\bm{a}) \|_{\mathcal{B}, p}}{\| \bm{a} \|_{\mathcal{B}, p}} : \bm{a} \in \mathcal{B}^n, \bm{a} \neq \bm{0}_{\mathcal{B}} \right\}.
\end{align*}
\end{definition}}
\noindent \black{As an immediate consequence of this definition, $\| \Phi (\bm{a}) \|_{\mathcal{B}, p} \leq \| \Phi \|_{\mathcal{B}, p} \| \bm{a} \|_{\mathcal{B}, p}$ holds for all $\bm{a} \in \mathcal{B}^n \setminus \bfz_\cB$. In a similar manner, given $1 \leq p \leq \infty$ we can define the norm $\| \cdot \|_{\mathcal{N}, p}$ over $\mathcal{N}$ satisfying a result analogous to Lemma \ref{BNormLemma} by replacing every $\mathcal{B}$ by $\mathcal{N}$ in (\ref{BpNorm}) and Definition \ref{BopNorm}.}

\subsection{Two's Complement}
\label{TwoCompSec}
From the end of Step 2 through the beginning of Step 5, the ZFP implementation stores the integer components of each block using a normalized fixed-point, two's complement binary integer representation. For $z \in \mathbb{Z}$, the two's complement representation is of the form
\begin{align*}
\text{Two's Complement:} \ \ \ \ z = - t_{N-1} 2^{N-1} + \sum_{i = 0}^{N-2} t_i 2^i,
\end{align*}
for some $N \in \mathbb{N}$ and some $t \in \mathcal{C}$. \black{Typically, the
value of $N$ is chosen beforehand to be the number of bits allotted for
storing each integer. Unfortunately, this aspect of the two's complement
representation does not lend itself to a construction of a vector
space, unlike $\mathcal{B}^n$ and $\mathcal{N}^n$  constructed in
Section~\ref{sec:bitvectorSpace}. Thus, instead of working explicitly in
two's complement, we will take care when defining operators for Steps 2
through 5 in $\cB^n$ to ensure they mimic the behavior of the
two's complement representation used in ZFP. }
%Note,
                                %there is a corresponding two's
                                %complement representation for all $a
                                %\in \cB$. 

%%%%%%%%%%%%%%%%%%%%%%%%%%%%%%%%%%%%%%%%%%%%%%%%%%%%%%%%%%%%%%%%%%%%

\subsection{Truncation Operator on $\mathcal{B}^n$ and $\mathcal{N}^n$}\label{ProjectionSubsection}
In this section, we consider some properties of the truncation operator, $T_{\mathcal{S}}{(\cdot)} $, over the normed vector spaces  $\mathcal{B}^n$ and $\mathcal{N}^n$. The usefulness of these definitions and results will be evident during the analysis of the ZFP compression algorithm in Section \ref{sec:bounds}.
%%%Note that one natural use for such a map is to take the set $\mathcal{S}$ to be finite. Then the map $F_{\mathcal{B}} P_{\mathcal{S}} F_{\mathcal{B}}^{-1} : \mathbb{R}^n \to \mathbb{R}^n$ returns a finite precision representation of the input. 

\begin{definition}
\label{def:emaxemin}
Let $\bm{x} \in \mathbb{R}^n$. The \emph{maximum exponent of $\bm{x}$
  with respect to $\mathcal{B}$} is 
\begin{align}
e_{max, \mathcal{B}}(\bm{x}) = \max_{1 \leq i \leq n} \max_{j} \left\{ j
\in \mathcal{I} \left( f_{\mathcal{B}}^{-1} \left(\bm{x}_i \right)
\right) \right\} \label{eqn:emax}, 
\end{align}
%%% Commented out old definition: e_{max}(\bx) =  \lfloor \log_2(\max_i{|x_i|})\rfloor
and \emph{minimum exponent of $\bm{x}$ with respect to $\mathcal{B}$} is
\begin{align}
e_{min, \mathcal{B}}(\bm{x}) = \min_{1 \leq i \leq n} \min_{j} \left\{ j \in \mathcal{I} \left( f_{\mathcal{B}}^{-1} \left(\bm{x}_i \right) \right) \right\}, \label{eqn:emin}
\end{align}
%%% Commented out old definition: e_{min}(\bx) =  \lfloor \log_2(\min_i{|x_i|})\rfloor, where $\lfloor \cdot \rfloor$ is the floor function.
provided that the minimum exists. Similarly, we define the \emph{maximum} and \emph{minimum exponent of $\bm{x}$ with respect to $\mathcal{N}$} by replacing $\mathcal{B}$ with $\mathcal{N}$ in (\ref{eqn:emax}) and (\ref{eqn:emin}). 
\end{definition}
\black{When it is clear from context which space, $\cB$ or $\cN$, the vector $\bm{x}$ is in, we will simply write $e_{max}$ or  $e_{min}$. The next result provides a relation between elements of $\mathcal{B}$ and $\mathcal{N}$ and $e_{max}$.}
\black{\begin{lemma}
\label{ProjectionLemma1}
For $\bm{a} \in \cB^n$, $\bm{d} \in \cN^n$: $(i)$ $\| \bm{a} \|_{\mathcal{B}, \infty} \geq 2^{e_{max, \mathcal{B}}(F_{\mathcal{B}}(\bm{a}))} \ \ \ (ii)$  $\| \bm{d} \|_{\mathcal{N}, \infty} \geq \frac{1}{3} 2^{e_{max, \mathcal{N}}(F_{\mathcal{N}}(\bm{d}))}$.
\end{lemma}}
\begin{proof}
$(i)$ follows immediately from the definition of $e_{max, \mathcal{B}}(\bm{x})$. Next, for $z \in \mathbb{Z}$ even we define $\mathcal{J}(z) = \{ 2j \in
\mathbb{Z} : 2j \leq z\}$ and for $z \in \mathbb{Z}$ odd we define
$\mathcal{J}(z) = \{ 2j+1 \in \mathbb{Z}:2j+1 \leq z\}$. Letting $e = e_{max, \mathcal{N}}(\bm{x})$, $(ii)$ follows from observing 
\begin{align*}
\| \bm{d} \|_{\mathcal{N}, \infty} \geq \left| (-2)^e - \sum_{i \in
  \mathcal{J}(e-1)} (-2)^{i} \right| = 2^e - \frac{2}{3} 2^e =
\frac{1}{3} 2^{e_{max, \mathcal{N}}(F_{\mathcal{N}}(\bm{d}))}.
\end{align*}
\end{proof}

We now provide a result establishing the relationship between $\bm{a}$
and $T_{\mathcal{S}} (\bm{a})$ given $\bm{a} \in \mathcal{B}^n$ or
$\bm{a} \in \mathcal{N}^n$. First, note that for certain choices of $m
\in \mathbb{Z}$ the constant term $2^{1 - m}$ will regularly occur in
error bounds established throughout this paper. Hence, we will let
$\epsilon_m := 2^{1-m}$ for any $m \in \mathbb{Z}$. For example, machine
epsilon \cite{higham2002accuracy} is defined as $\epsilon_k = 2^{1-k}$
for precision $k$.  
\black{\begin{lemma}
\label{ProjectionLemma2}
Suppose $p, q \in \mathbb{Z}$ and let $\mathcal{S} = \{ i \in \mathbb{Z} : i > p - q \}$. If $\bm{a} \in \cB^n$ and $\bm{d} \in \cN^n$ then
\begin{enumerate}
\item[$(i)$] $\bm{a} = T_{\mathcal{S}} (\bm{a}) + \Delta \bm{a}$, for some $\Delta \bm{a} \in \mathcal{B}^n$ satisfying $\| \Delta \bm{a} \|_{\mathcal{B}, \infty} \leq \epsilon_q 2^{p}$.
 \item[$(ii)$] $\bm{d} =T_{\mathcal{S}} (\bm{d}) + \Delta \bm{d}$, for some $\Delta \bm{d} \in \mathcal{N}^n$ satisfying $\| \Delta \bm{d} \|_{\mathcal{N}, \infty} \leq \frac{2}{3} \epsilon_q 2^{p}$.
\end{enumerate}
\end{lemma}}
\begin{proof}
Let $1 \leq m \leq n$. Observing 
\begin{align*}
\left| f_{\mathcal{B}} \left( \Delta \bm{a}_m \right) \right| = \left| f_{\mathcal{B}} \left( \bm{a}_m - T _{\mathcal{S}} (\bm{a})_m \right) \right| = \sum_{i \in \mathcal{I} \left( \bm{a}_m \ominus T_{\mathcal{S}} (\bm{a})_m \right)} 2^i \leq \sum_{i = -\infty}^{p - q} 2^i = 2^{{p} - q + 1}
\end{align*}
concludes the proof of $(i)$. Next, for $z \in \mathbb{Z}$ even we define $\mathcal{J}(z) = \{ 2j \in
\mathbb{Z} : 2j \leq z\}$ and for $z \in \mathbb{Z}$ odd we define
$\mathcal{J}(z) = \{ 2j+1 \in \mathbb{Z}:2j+1 \leq z\}$. Then $(ii)$ follows by observing that  
\begin{align*}
\left| f_{\mathcal{N}} \left( \Delta \bm{d}_m \right) \right| = \left|
f_{\mathcal{N}} \left( \bm{d}_m - T_{\mathcal{S}} (\bm{d})_m \right)
\right| = \left| \sum_{i \in \mathcal{I} \left( \bm{a}_m \ominus
  T_{\mathcal{S}} (\bm{a})_m \right)} (-2)^i \right| \leq \left| \sum_{i
  \in \mathcal{J}(p-q)} (-2)^i \right| = \frac{2}{3} 2^{p -q + 1}. 
\end{align*}
\end{proof}

Hence, if $\mathcal{S} = \{ i \in \mathbb{Z} : i > {e_{\max}(F_{\mathcal{B}}(\bm{a}))} - q \}$ then the additional round-off error incurred by the truncation
operator is dependent on $q$ and the magnitude of the input. Using
Lemma~\ref{ProjectionLemma2}, we now observe that the component-wise
relative error introduced by the truncation operator is bounded by 
\begin{align} 
\max_{m, f_\cB(\bm{a}_m) \neq 0} f_\cB \left( \left | \frac{\bm{a}_m - T_\cS(\bm{a}_m)}{\bm{a}_m} \right | \right) & \leq \frac{ \left \|\Delta \ba\right \|_{\cB,\infty} }{\min_{m, f_\cB(\bm{a}_m) \neq 0} f_\cB(| \bm{a}_m |)} = \epsilon_q \black{2^{e_{max}(F_{\mathcal{B}}(\bm{a}))-e_{min}(F_{\mathcal{B}}(\bm{a}))}},   \label{eqn:componentwiseerror}
\end{align}
for $\ba \in \cB^n $ with $F_\cB(\ba) \neq \bfz$. So the component-wise
error relative to the input is dependent on $q$ and the exponent range, i.e.,
$\rho =
e_{max}(F_{\mathcal{B}}(\bm{a}))-e_{min}(F_{\mathcal{B}}(\bm{a}))$. In
\cite{zfp}, it was noted that in many real-world examples, the exponent
range was reasonable ($\rho \leq 8$). Thus, depending on $\rho$, $q$
could be chosen to ensure the component-wise error relative to the input
remains smaller than machine epsilon. 

%%% Commented out: Combined this lemma with the one above
\iffalse
\begin{lemma}
\label{ProjectionLemma3}
Let $\ba \in \{ \bb \in \cN ^n: i\geq 0 \quad \forall i \in \cI(\bb)\}$ and $\beta \in \mathbb{N}_0$ be given. Suppose that $\mathcal{T} = \{ i \in \mathbb{N} : i \geq i_{max} - \beta \}$, where $i_{max} = \max(\cI(\ba))$ . Then
\begin{align*}
\bm{a} = T_{\mathcal{T}} (\bm{a}) + \Delta \bm{a},
\end{align*}
for some $\Delta \bm{a} \in \mathcal{N}^n$ satisfying $\| \Delta \bm{a} \|_{\mathcal{N}, \infty} \leq 2^{i_{max} - \beta}$ and  $\| \Delta \bm{a} \|_{\mathcal{N}, \infty} \leq 2^{-\beta+1} \| \bm{a} \|_{\mathcal{N}, \infty}$.
\end{lemma}
\begin{proof}
Suppose $1 \leq i \leq n$. Observing that
\begin{align*}
\left | f_{\mathcal{N}} \left( \Delta \bm{a}_i \right) \right |&= \left| f_{\mathcal{N}} \left( \bm{a}_i \ominus T_{\mathcal{T}} (\bm{a})_i \right) \right | \\ 
=& \left | \sum_{j=0}^{i_{max} - \beta-1} a_i(-2)^j  \right |\\ 
&\leq \sum_{j=0}^{i_{max} - \beta-1} 2^j \\ 
&\leq 2^{i_{max} - \beta}-1 \\ 
&\leq  2^{i_{max} - \beta}. \\
\end{align*}
concludes the proof of the first claim. 
Notice, $\|\ba\|_{\cN,\infty} \geq 2^{i_{max}-1}$, then we can conclude the second claim 
\[ \|\Delta \bm{a} \|_{\cN, \infty} \leq  2^{-\beta +1}\|\ba\|_{\cN,\infty}.\] 
\end{proof}
\fi
%%% End comment out

%%%%%%%%%%%%%%%%%%%%%%%%%%%%%%%%%%%%%%%%%%%%%%%%%%%%%%%%%%%%%%%%%%%%

\subsection{Shift Operator on $\mathcal{B}^n$}\label{ShiftSubsection}
\black{We now wish to determine what can be said about the norm of the shift operator defined on $\mathcal{B}^n$.} For $a \in \mathcal{B}$ and $\ell \in \mathbb{Z}$, observe that
\begin{align*}
\left| f_\cB(s_{\ell} (a)) \right| = \left| \sum_{i \in \mathcal{I}(a)} 2^{i - \ell} \right| = 2^{-\ell} \left| \sum_{i \in \mathcal{I}(a)} 2^{i} \right| = 2^{-\ell} \left| f_\cB(a) \right|.
\end{align*}
\black{This observation, together with the definition of $\| \cdot \|_{\mathcal{B}, p}$, yields the following result.}

\begin{lemma}
\label{ShiftLemma}
Suppose $\ell \in \mathbb{Z}$ and $1\leq p \leq \infty$. Then $\| S_{\ell} \|_{\mathcal{B}, p} = 2^{-\ell}$ and $\| S_{\ell}^{-1} \|_{\mathcal{B}, p} = 2^{\ell}$.
\end{lemma}
%%%I commented this out as it is trivial from the definition of the norm and the observation presented in the beginning of this section.
\iffalse
\begin{proof}
Suppose that $\bm{a} \in \mathcal{B}^n$ with $\bm{a} \neq \bm{0}_{\mathcal{B}}$. Then for $1\leq p <\infty$ we have
\begin{align*}
\| S_l (\bm{a}) \|_{\mathcal{B}, p} = \left( \sum_{i = 1}^{n} \left| f_\cB \left(s_l(\bm{a}_i) \right) \right|^p \right)^{1/p} = \left( \sum_{i = 1}^{n} \left| 2^{-l} f_\cB \left( \bm{a}_i \right) \right|^p \right)^{1/p} = 2^{-l} \| \bm{a} \|_{\mathcal{B}, p}
\end{align*}
and 
\begin{align*}
\| S_l (\bm{a}) \|_{\mathcal{B}, \infty} = \max_{1 \leq i \leq n} \| s_l(\bm{a}_i) \|_{\mathcal{B}} = \max_{1 \leq i \leq n} \left\{ 2^{-l} \| \bm{a}_i \|_{\mathcal{B}} \right\} = 2^{-l} \max_{1 \leq i \leq n} \| \bm{a}_i \|_{\mathcal{B}} = 2^{-l} \| \bm{a} \|_{\mathcal{B}, \infty}.
\end{align*}
Hence, $\displaystyle \frac{\| S_l (\bm{a}) \|_{\mathcal{B}, p}}{\| \bm{a} \|_{\mathcal{B}, p}} = 2^{-l}$, for all $1\leq p <\infty$. Since $\bm{a} \in \mathcal{B}^n$ was taken to be arbitrary, we conclude that $\| S_l \|_{\mathcal{B}, p}=2^{-l}.$
As $S_l^{-1} =S_{-l}$, $\| S_l^{-1}  \|_{\mathcal{B}, p}=2^{l}$ immediately follows. 
\end{proof}
\fi

\black{To summarize, we have constructed the normed vector spaces $\cB^n$
and $\cN^n$ and bijective maps between $\cB^n$ and $\mathbb{R}^n$ and
$\cN^n$  and $\mathbb{R}^n$. We can represent floating-point or fixed-point 
representations by applying the truncation operator
$T_\cS{(\cdot)} $ on elements of $\cB^n$ or $\cN^n$, and
multiplication by powers of two in $\mathbb{R}^n$ is equivalent to
applying shift operator, $S_{\ell}$, to elements in $\cB^n$. 
We now have the tools to define operators for each step of the ZFP compression
algorithm, as described is Section~\ref{sec:zfp}. }

%% file: analysis.tex
\section{Error Analysis of Individual Steps of the ZFP Compression Algorithm}
\label{sec:analysis}
The goals of this section are to define operators for each step of the ZFP
compression algorithm and to determine the error resulting from each
step. For each step of ZFP, we define an operator that carries out
the implemented version of ZFP compression and decompression as well as
a lossless version. The lossless version of each operator will be useful
in determining the error introduced at each step of the
algorithm. Decompression for steps corresponding to invertible
compression steps are merely the inverse operators of the compression
step. Since ZFP compression is lossy in nature, some of the steps
implemented in the compression phase are not invertible. For such steps,
the corresponding decompression step is defined by an injective map
that restores only the information that has not been lost to the correct
format for the next step of decompression. For the sake of brevity, any
step of the algorithm that does not affect the error analysis will not
be considered in much depth.  
%All operators can be represented in the infinite bit vector spaces $\cB^n$ and $\cN^n$, respectively, but will be designed to imitate what happens in the subsets $\cB^n_k$ and $\cN^n_k$ for some precision $k \in \mathbb{Z}$. 

\subsection{Partition $d$-dimensional array into blocks of $4^d$ values}
\label{Step1Sec}
In Step 1, the $d$-dimensional array is partitioned into blocks of size $4^d$. Since Steps 2 through 8 are then applied to each $4^d$ block individually, it is not necessary to consider Step 1 in the error analysis. Accordingly, we do not define any operators for this step.

\subsection{Block-Floating-Point Transform}
\label{Step2Sec}
Suppose that $\bx \in \mathbb{R}^{4^d}$ such that $F_\cB^{-1}(\bx) \in
\cB_k^{4^d}$ for some precision $k$ (i.e., every element in $\bx$ can be
represented with at most $k$-consecutive bits). For frequently used
IEEE floating-point types, $k \in \{24, 53\}$. This assumption on
$\bm{x}$ implies that we are working 
with a floating point representation of a real number. To perform Step
2, we first convert each component in $\bx$ to its corresponding
representation in $\mathcal{\cB}$. Each element is then shifted to the
left by a deterministic number of bits and truncated. As a by-product of
type-casting to an integer, in the implementation of ZFP, each value is
rounded down to zero. Applying the shift operator followed by the
truncation operator, as outlined above, results in the same outcome.  

\black{The operator in Step 2 is dependent on the fixed set $\mathcal{S} := \{
i \in \mathbb{Z} : i \geq 0 \}$ and the value $q \in
\mathbb{N}$, where $q$ denotes the maximum number of nonzero consecutive
bits (precision) that can be used for the representation of each component of the
input.} ZFP requires each value to have one bit as a safe-guard
against overflow, which occurs when the calculation produces a result
that exceeds the capacity of the finite bit representation. \black{In the
current ZFP implementation, if the input values are IEEE single or double precision, $q \in \{30, 62\}$ in $\cB$, since one 
bit is used to represent the sign bit and another to represent the
overflow guard bit. %In practical terms, this means that the user needs to provide 
%	higher fidelity data to ZFP if they desire a larger value for $q$. For example, 
%	IEEE single precision data provided to ZFP will result in $q = 30$, however, ZFP can easily be modified so that $q$ can be user defined by any natural number.
}Step 2 is defined by 
the map $\tilde{C}_2:\mathbb{R}^{4^d} \rightarrow \cB^{4^d}$ where 
\begin{align*}*
\tilde{C}_2 (\bm{x}) := T_{\mathcal{S}} S_{\ell} F_{\cB}^{-1}(\bm{x}),
\text{ for all } \bm{x} \in \mathbb{R}^{4^d}, 
\end{align*}
where \red{$\ell = e_{max}(F_\cB^{-1}(\bx))- q+1$.} We define the lossless
operator, $C_2$, by removing all noninvertible maps from
$\tilde{C}_2$. Hence, 
\begin{align*}
C_2(\bx ) := S_{\ell}F_{\cB}^{-1} (\bx).
\end{align*}
%The decompression operator for Step 2 is then defined as the inverse of
%$C_2$, that is, $D_2: {\cB}^{4^d}\rightarrow \mathbb{R}^{4^d}$ with
%$D_2(\ba) := F_{\cB} S_{-\ell}(\ba)$, for all $\bm{a} \in
%\mathcal{{\cB}}^{4^d}$. We conclude our discussion of this step by
%presenting a result that will be useful during the error analysis in
%Section~\ref{sec:bounds}. 
%
%\begin{prop}
%\label{Step2Prop}
%Suppose $\bm{x} \in \mathbb{R}^{4^d}$. Then $\| \tilde{C}_2 \bx - C_2
%\bx \|_{\infty} \leq 2^{-\ell} \epsilon_{q} \| \bm{x} \|_{\infty}$. 
%\end{prop}
%\begin{proof}
%Observe that 
%\begin{align}
%\| \tilde{C}_2 \bx - C_2 \bx \|_{\mathcal{B}, \infty} = \|
%T_{\mathcal{S}} S_{\ell}F_\cB^{-1}(\bm{x}) - S_{\ell} F_\cB^{-1}(\bm{x})
%\|_{\mathcal{B}, \infty}&\leq  \epsilon_q \| S_\ell F_\cB^{-1}(\bm{x})
%\|_{\mathcal{B}, \infty} \label{4.1eq0},\\  
%&\leq 2^{-\ell} \epsilon_q  \| \bm{x} \|_{ \infty}, \label{4.1eq1}
%\end{align}
%where the inequality in (\ref{4.1eq0}) follows from Lemma
%\ref{ProjectionLemma2} (i) and Lemma
%\ref{ProjectionLemma1} (i). From Lemma~\ref{ShiftLemma}, we have that
%$\|S_{\ell} \|_{\mathcal{B}, \infty} = 2^{-\ell}$, and from
%Lemma~\ref{BNormLemma}, we have
%$\| F_\cB^{-1}(\bm{x}) \|_{\mathcal{B},\infty}= \| \bm{x}
%\|_{\infty}$. These together yield the inequality in (\ref{4.1eq1}). 
%\end{proof}

{The decompression operator for Step 2 converts the block-floating point back to its original floating-point representation that is representable in $\cB_k$ for $k \in \mathbb{N}$. In IEEE, the $q\in \{30,62\}$ consecutive bits must be converted back to $k \in \{24,53\}$ with its respective exponent information. This conversion can be seen as a typical floating-point round off error. The lossy decompression operator for Step 2 is then defined by undoing the shift performed in $\tilde{C}_2$ and converting each component back to a floating point representation. Hence, $\tilde{D}_2: {\cB}^{4^d}\rightarrow \mathbb{R}^{4^d}$ is defined by 
	\begin{align*}
	\black{\tilde{D}_2(\ba) := F_{\cB} S_{-\ell} fl_k (\ba)}, \ \text{for all} \ \bm{a} \in
	\mathcal{{\cB}}^{4^d},
	\end{align*}
	where \black{$fl_k (\ba)_i  = t_{\cR_{ik}} (\ba_i)$ with $\cR_{ik} = \{ j \in \mathbb{Z}: j > e_{max, \cB} (\bm{a}_i) - k \}$, for all $1 \leq i \leq 4^d$. Note that the $fl_k$ operator converts each component of $\bm{a}$ to a floating point representation with $k$ mantissa bits in a bit vector format.} The lossless decompression operator is then defined as  $D_2: {\cB}^{4^d}\rightarrow \mathbb{R}^{4^d}$ with
	$D_2(\ba) := F_{\cB} S_{-\ell}(\ba)$, for all $\bm{a} \in
	\mathcal{{\cB}}^{4^d}$. We conclude our discussion of this step by
	presenting a result that will be useful during the error analysis in
	Section~\ref{sec:bounds}. 
	\begin{prop}Suppose $\bm{x} \in \mathbb{R}^{4^d}$ and $\bm{a} \in \cB^{4^d}$, such that \red{$e_{max,\cB}(F_\cB (\ba)) \geq q-1.$} 
		\label{Step2Prop}
		\begin{itemize}
			\item[(i)] Then $\| \tilde{C}_2 \bx - C_2
			\bx \|_{\infty} \leq 2^{-\ell} \epsilon_{q} \| \bm{x} \|_{\infty}$. 
			\item[(ii)] Then $\| \tilde{D}_2 \ba - D_2
			\ba \|_{\infty} \leq 2^{\ell} \epsilon_{k} \| \bm{a} \|_{\infty}$. 
		\end{itemize}
	\end{prop}
	\begin{proof} For $(i)$, observe that 
		\begin{align}
		\| \tilde{C}_2 \bx - C_2 \bx \|_{\mathcal{B}, \infty} = \|
		T_{\mathcal{S}} S_{\ell}F_\cB^{-1}(\bm{x}) - S_{\ell} F_\cB^{-1}(\bm{x})
		\|_{\mathcal{B}, \infty}&\leq  \epsilon_q \| S_\ell F_\cB^{-1}(\bm{x})
		\|_{\mathcal{B}, \infty} \label{4.1eq0},
		\\  
		&\leq 2^{-\ell} \epsilon_q  \| \bm{x} \|_{ \infty}, \label{4.1eq1}
		\end{align}
		where the inequality in (\ref{4.1eq0}) follows from Lemma
		\ref{ProjectionLemma2} (i) and Lemma
		\ref{ProjectionLemma1} (i). From Lemma~\ref{ShiftLemma}, we have that
		$\|S_{\ell} \|_{\mathcal{B}, \infty} = 2^{-\ell}$, and from
		Lemma~\ref{BNormLemma}, we have
		$\| F_\cB^{-1}(\bm{x}) \|_{\mathcal{B},\infty}= \| \bm{x}
		\|_{\infty}$. These together yield the inequality in (\ref{4.1eq1}). 
		Next, $(ii)$ follows from an argument similar to the proof of $(i)$ by observing that
		\black{$\displaystyle \| \tilde{D}_2 \ba - D_2 \ba \|_{ \infty} = \|
		F_{\cB} S_{-\ell} fl_k (\ba) - F_{\cB} S_{-\ell}(\ba)
		\|_{\infty} \leq 2^\ell \|fl_k (\ba) - \ba
		\|_{\mathcal{B}, \infty} \leq 2^{\ell} \epsilon_k  \| \bm{a} \|_{\cB,\infty}$.}
	\end{proof}
}

\subsection{Decorrelating Linear Transform}
\label{Step3Sec} 
In Step 3, the output from Step 2 is acted on by a linear
transformation, $L$. \red{$L$ is a near-orthogonal transform that is similar to the discrete cosine transform, both of which possess the energy compaction property \cite{Rao1990}, i.e., most of the signal energy is confined to the first, lowest-frequency transform coefficients.} In $d$-dimensions, the transform operator can be
applied to each dimension separately, and the operator can be
represented as a Kronecker product. For $A \in \mathbb{R}^{n_1,m_1}$ and
$B \in \mathbb{R}^{n_2,m_2}$, the Kronecker product is defined as  
\begin{align*}
A\otimes B =\begin{bmatrix}a_{1,1}B &a_{1,2} B & \cdots a_{1,m_1}B \\  \vdots & \ddots& \vdots  \\ 
a_{n_1,1}B &a_{n_1,2} B & \cdots a_{n_1,m_1}B \end{bmatrix}. 
\end{align*}
Then, the total forward transform operator used in of ZFP is defined as 
\begin{align*}
\cL_d = \underbrace{\cL\otimes \cL \otimes \cdots \otimes \cL}_\text{$(d - 1)$-products},
\end{align*}
where $\cL \in \mathbb{R}^{4\times 4}$ is defined by
\begin{align}
\cL = \frac{1}{16} \begin{bmatrix}
\begin{array}{rrrr}
4 & 4 & 4 & 4 \\
5 & 1 & -1 & -5 \\
-4 & 4 & 4 & -4 \\
-2 & 6 & -6 & 2
\end{array}
\end{bmatrix} \quad \text{and} \quad \cL^{-1}= \frac{1}{4} \begin{bmatrix}
\begin{array}{rrrr}
4 & 6 & -4 & -1 \\
4 & 2 & 4 & 5 \\
4 & -2 & 4 & -5 \\
4 & -6 & -4 & 1
\end{array}
\end{bmatrix}. %\label{eqn:T}
\end{align}
Note that $\| \cL\|_{\infty} = 1$ and $\| \cL^{-1} \|_{\infty} = 15/4$.

First, we define the lossless compression operator for Step 3 by $C_3:
\cB^{4^d} \rightarrow \cB^{4^d}$, where
\begin{align}\label{c3}
C_3(\ba) = F_\cB^{-1}\cL_dF_\cB(\ba), \text{ for all } \bm{a} \in \mathcal{B}^{4^d}.
\end{align}
In order to define the lossy operator used in the implementation, it is
necessary to account for the finite bit constraint on a machine. Based
on Step 2 of ZFP compression, the components provided as the input for
Step 3 represent integers. Hence, for some $q \in \mathbb{N}$, it
follows that the input for Step 3 is an element of
$\mathcal{B}_{q}^{4^d}$. Here, $q$ represents the number of bits
available for storing each component. As $\mathcal{B}_q$ is not closed
under addition and multiplication, given $a, b \in \mathcal{B}_q$,
addition or multiplication of $a$ and $b$ may not result in an element
of $\mathcal{B}_q$ and must be rounded. This circumstance is referred to as round-off; error that occurs when the calculation produces a result that
exceeds the 
capacity of the finite bit representation. Since the transformation
could result in round-off, the operator used in the
implementation of the algorithm will be defined as   
\begin{align}\label{tildec3}
\tilde{C}_3 = F_\cB^{-1} \tilde{\cL}_dF_\cB(\ba), \text{ for all } \bm{a} \in \mathcal{B}^{4^d},\end{align}
where $\tilde{\cL}_d$ is an operator such that $\tilde{\cL}_d F_\cB(\ba) \in \cB^{\black{4^d}}_q$, for all $\bm{a} \in \mathcal{B}^{4^d}$. 

As the linear transform operator, $\cL$, is invertible, the lossless
decompression operator $D_3: \cB^{4^d} \rightarrow \cB^{4^d}$ is defined
as  
\[ D_3(\ba) = F_\cB^{-1}\cL_d^{-1}F_\cB(\ba), \text{ for all } \bm{a} \in \mathcal{B}^{4^d}. \]
Again, since the operation $\cL_d^{-1}$ may result in round-off, the operator used in the implementation is defined as
\[ \tilde{D}_3(\ba) = F_\cB^{-1}\tilde{\cL}_d^{-1}F_\cB(\ba), \text{ for all } \bm{a} \in \mathcal{B}^{4^d},  \]
where $\tilde{\cL}_d^{-1}$ is an approximation of $\cL_d^{-1}$. %satisfying a forward error bound analogous to (\ref{eqn:T}). 

From \cite{higham2002accuracy} (Equation (3.12)), the forward error bound of the floating-point representation of a matrix-vector product, $\cL_d\bx \in \mathbb{R}^{4^d}$, is 
\begin{equation} 
\left\| \cL_d\bx  - \tilde{\cL}_d \bx \right \|_p \leq \gamma \left \|\cL_d \right \|_p \left \|\bx \right\|_p, \label{eqn:T} 
\end{equation}
where $\gamma = 4^d \epsilon_m/ (1 - 4^d \epsilon_m)$ and $\epsilon_m =
2^{1-m}$ represents machine epsilon with precision
$m$. From~\cite{tensor}, we have that $\|\cL_d\|_p \leq\|\cL\|_p^d$, for
$1 \leq p \leq \infty$. Hence, (\ref{eqn:T}) yields 
\begin{equation} 
\left\| \cL_d\bx  - \tilde{\cL}_d \bx \right \|_p \leq \gamma \left \|\cL \right \|^d_p \left \|\bx \right\|_p. \label{eqn:Td} 
\end{equation}
Note that $\tilde{\cL}_d^{-1}$ satisfies a forward error bound analogous
to (\ref{eqn:Td}). The forward error bound represented in (\ref{eqn:Td})
is the worst possible error that can occur for an arbitrary linear
transform. As ZFP uses particular transformations, we aim to establish bounds specific to the
transformations $\cL_d$ and $\cL_d^{-1}$. Accordingly, we note that the
action of $\cL$ and $\cL^{-1}$ can be written in a very efficient C
implementation. The action of $\cL$ and $\cL^{-1}$ on $\ba = [ \ba_1,
  \ba_2, \ba_3, \ba_4 ]^T \in \cB^4$ under this implementation is
outlined in Table \ref{table:actionT}.

\begin{table}[!h]
	\begin{adjustbox}{width=\textwidth} 
		\begin{tabular}{| L  L   L || L   L  L | }
			\hline
			\multicolumn{3}{| c ||}{$\cL$} & \multicolumn{3}{ c |}{$\cL^{-1}$} \\ \hline 
			\ba_1 \leftarrow \ba_1+ \ba_4 & \ba_1 \leftarrow s_1(\ba_1) &\ba_4 \leftarrow \ba_4 - \ba_1 &\ba_2 \leftarrow \ba_2 + s_1(\ba_4)&\ba_4 \leftarrow \ba_4- s_1(\ba_2) & \\
			\ba_3 \leftarrow \ba_3+\ba_2 & \ba_3 \leftarrow s_1(\ba_3)&\ba_2 \leftarrow \ba_2- \ba_3 &\ba_2 \leftarrow \ba_2+ \ba_4 &\ba_4 \leftarrow s_{-1}(\ba_4) &\ba_4 \leftarrow \ba_4- \ba_2 \\
			\ba_1 \leftarrow \ba_1+\ba_3 & \ba_1 \leftarrow s_1(\ba_1)&\ba_3 \leftarrow \ba_3- \ba_1 &\ba_3 \leftarrow \ba_3 + \ba_1&\ba_1 \leftarrow s_{-1}(\ba_1) &\ba_1 \leftarrow \ba_1- \ba_3  \\
			\ba_4 \leftarrow \ba_4+\ba_2 & \ba_4 \leftarrow s_1(\ba_4) &\ba_2 \leftarrow \ba_2 - \ba_4 &\ba_2 \leftarrow \ba_2 + \ba_3 &\ba_3 \leftarrow s_{-1}(\ba_3) &\ba_3 \leftarrow \ba_3- \ba_2  \\
			\ba_4 \leftarrow \ba_4 + s_1(\ba_2) & \ba_2  \leftarrow \ba_2 - s_1(\ba_4) & &\ba_4 \leftarrow \ba_4 + \ba_1   &\ba_1 \leftarrow s_{-1}(\ba_1) & \ba_1\leftarrow \ba_1 - \ba_4 \\ \hline
		\end{tabular}
	\end{adjustbox}
	\caption{\label{table:actionT} Bit arithmetic steps for ZFP's forward (left) and backward (right) transform (read from left to right).
	} 
\end{table}
This implementation is straightforward and efficient as it only requires bit addition/subtraction and division/multiplication by two. %, i.e., single bit shifts. 
In ZFP, the bit vectors are padded so that any overflow that may occur
is represented (i.e., for each component in the block one extra bit is
allotted to ensure that, if a calculation results is a value greater
than what can be represented in $q$ bits, then the value is not
approximated). Thus, for the following analysis, it suffices to
calculate the error due to round-off. Additionally, as the components of
the input for Step 3 represent signed integers, round-off can only occur
during division by two (i.e., one bit shift to the right using
$s_1(\cdot)$).

As noted in Section~\ref{TwoCompSec}, care must be taken in Step 3,
since the implementation of ZFP uses a two's complement representation
of each integer. For our error analysis, the main concern is that
rounding to an integer in two's complement after a right bit shift
always results in rounding towards negative infinity. However, under the
representation defined in $\mathcal{B}$, this same sequence of
operations results in rounding towards zero. So, in order to mimic the
implementation, we define the operator $r : \mathcal{B} \to \mathcal{B}$
by  
 \begin{align*}
r (a) := \begin{cases}
t_\cS s_{1}(a) & :  f_{\mathcal{B}} (a) \geq 0, \\
t_\cS s_{1} \left( a - 1_\cB \right) & : f_{\mathcal{B}} (a) < 0, 
\end{cases}
\end{align*}
for all $a \in \mathcal{B}$. Since $s_{1}(\cdot)$ performs a single
right bit shift and $t_\cS (\cdot)$, where $\cS = \{ i \in \mathbb{Z} :
i\geq 0\}$, rounds the value towards zero, $r (\cdot)$ will always round
the right bit shift toward negative infinity. The following lemma
considers the error of $r (\cdot)$ when compared to $s_1 (\cdot)$. 
\begin{lemma}\label{lemma:boundr}
Suppose $\cS = \{ i \in \mathbb{Z} : i \geq 0\}$ and $p \in \mathbb{Z}$.
If $a = f_\cB^{-1} (p)$, then $\left\| r (a) - s_1(a)
\right\|_{\mathcal{B}, \infty} \leq \frac{1}{2}.$ 
\end{lemma}
\begin{proof}
As $p \in \mathbb{Z}$ and $a = f_\cB^{-1} (p)$, we have that
$\mathcal{I} (a) \subseteq \{ i \in \mathbb{Z} : i \geq 0 \}$. Now
suppose $p \geq 0$. Then  
\begin{align*}
\mathcal{I} ( r (a) ) = \mathcal{I} ( t_{\mathcal{S}} s_1(a) ) = \mathcal{I} ( s_1(a) ) \setminus \{ -1 \}. 
\end{align*}
Hence, $\left\| r (a) - s_1 (a) \right\|_{\mathcal{B}, \infty} \leq \frac{1}{2}$. 

On the other hand, suppose $p < 0$. If $p$ is even, then $p = 2k$ for
some $k \in \mathbb{Z}$, and $f_{\mathcal{B}} ( r (a) ) =
f_{\mathcal{B}} ( t_{\mathcal{S}} s_1 (a - 1_{\mathcal{B}}) ) = k.$ So  
\begin{align*}
\left\| r (a) - s_1 (a) \right\|_{\mathcal{B}, \infty} = \left| f_{\mathcal{B}} ( r (a) ) - f_{\mathcal{B}} ( s_1 (a) ) \right| = \left| k - \frac{p}{2} \right| = | k - k | = 0.
\end{align*}
If $p$ is odd, then $p = 2k - 1$ for some $k \in \mathbb{Z}$, and
$f_{\mathcal{B}} ( r (a) ) = f_{\mathcal{B}} ( t_{\mathcal{S}} s_1 (a -
1_{\mathcal{B}}) ) = k - 1$. Hence, $\left\| r (a) - s_1 (a)
\right\|_{\mathcal{B}, \infty} = \left| f_{\mathcal{B}} ( r (a) ) -
f_{\mathcal{B}} ( s_1 (a) ) \right| = \left| k - 1 - \frac{p}{2} \right|
= \frac{1}{2}$. 
\end{proof}
Thus, by replacing $s_1(\cdot)$ by $r (\cdot)$ in
Table~\ref{table:actionT}, we obtain the analogous lossy operators,
denoted $\tilde{\cL}$ and $\tilde{\cL}^{-1}$, outlined in
Table~\ref{table:actionTlossy}.
 \begin{table}[!h]
	\begin{adjustbox}{width=\textwidth} 
		\begin{tabular}{| L  L   L || L   L  L | }
			\hline
			\multicolumn{3}{| c ||}{$\tilde{\cL}$} & \multicolumn{3}{ c |}{$\tilde{\cL}^{-1}$} \\ \hline 
			\ba_1 \leftarrow \ba_1+ \ba_4& \ba_1 \leftarrow r(\ba_1) &\ba_4 \leftarrow \ba_4 - \ba_1 &\ba_2 \leftarrow \ba_2 + r(\ba_4)&\ba_4 \leftarrow \ba_4- r(\ba_2) & \\
			\ba_3 \leftarrow \ba_3+\ba_2 & \ba_3 \leftarrow r(\ba_3)&\ba_2 \leftarrow \ba_2- \ba_3 &\ba_2 \leftarrow \ba_2+ \ba_4 &\ba_4 \leftarrow s_{-1} (\ba_4) &\ba_4 \leftarrow \ba_4- \ba_2 \\
			\ba_1 \leftarrow \ba_1+\ba_3 & \ba_1 \leftarrow r(\ba_1)&\ba_3 \leftarrow \ba_3- \ba_1 &\ba_3 \leftarrow \ba_3 + \ba_1&\ba_1 \leftarrow s_{-1}(\ba_1) &\ba_1 \leftarrow \ba_1- \ba_3  \\
			\ba_4 \leftarrow \ba_4+\ba_2 & \ba_4 \leftarrow r(\ba_4) &\ba_2 \leftarrow \ba_2 - \ba_4 &\ba_2 \leftarrow \ba_2 + \ba_3 &\ba_3 \leftarrow s_{-1} (\ba_3) &\ba_3 \leftarrow \ba_3- \ba_2  \\
			\ba_4 \leftarrow \ba_4 + r(\ba_2) & \ba_2  \leftarrow \ba_2 - r(\ba_4) & &\ba_4 \leftarrow \ba_4 + \ba_1   &\ba_1 \leftarrow s_{-1} (\ba_1) & \ba_1\leftarrow \ba_1 - \ba_4 \\ \hline
		\end{tabular}
	\end{adjustbox}
	\caption{\label{table:actionTlossy} Bit arithmetic steps for the lossy implementation of ZFP's forward (left) and backward (right) transform.
	} 
\end{table}

%{From Table~\ref{table:actionTlossy} we observe that the last two steps for the forward transform is exactly reversed by the first two steps of the backwards transform.   From our notation we have  
%$$\begin{array}{llll}
%(i)&\ba_4 \leftarrow \ba_4 + r(\ba_2) & \Rightarrow \ba_4 = \ba_4 + r(\ba_2),&\\
%(ii)& \ba_2  \leftarrow \ba_2 - r(\ba_4)  & \Rightarrow \ba_2 =\ba_2-r(\ba_4 + r(\ba_2)), &\\ 
%(iii) & \ba_2 \leftarrow \ba_2 + r(\ba_4) & \Rightarrow \ba_2 =\ba_2-r(\ba_4 + r(\ba_2)) +r(\ba_4 + r(\ba_2) )&= \ba_2,\\ 
%(iv) & \ba_4 \leftarrow \ba_4- r(\ba_2) & \Rightarrow \ba_4 = (\ba_4 + r(\ba_2)) - r(\ba_2) &= \ba_4. \\
%\end{array}$$
%Thus, the final two steps of the compression and the first two steps of the decompression operators exactly reversible. \red{Thu and can be ignored for the error analysis. For the remainder of the section we will define $\tilde{L}$ and $L$ as the forward linear transform operator without the final two steps in Table~\ref{table:actionTlossy}. T
\red{Now that all the required notation and tools have been discussed, the following lemma establishes a forward error bound for $\tilde{L}$. }  
\begin{lemma}
\label{LemmaL1DBound}
Suppose $\bx \in \mathbb{Z}^4$ such that \red{$e_{max}(\bx) \geq q-1$} and $\bx \neq \bfz$. Given the bit arithmetic implementation in Table \ref{table:actionT} and Table \ref{table:actionTlossy} for ZFP's forward linear transforms, we have
\red{\begin{align*}
\|\cL\bx - \tilde{\cL}\bx\|_\infty \leq \frac{7}{4} \epsilon_q \|\bx\|_\infty.
\end{align*}}
%and 
%\begin{align*}
%\|\cL^{-1} \bx - \tilde{\cL}^{-1} \bx\|_\infty \leq \frac{5}{8}\epsilon_q \|\bx\|_\infty. 
%\end{align*}
\end{lemma}
\begin{proof}
\red{First the action of ${L}$ and $\tilde{L}$ will be formed as a composite operator of each step, as depicted in Table \ref{table:actionT} and \ref{table:actionTlossy}, respectively. Then a bound on the error between the action of $L$ and $\tilde{L}$ will be constructed using the Lemma \ref{lemma:boundr} and the triangle inequality.} Define $\ba = [ \ba_1, \ba_2, \ba_3, \ba_4 ]^T  = F_\cB^{-1}(\bx) \in
\cB^4$ to be the representation of $\bx$ in $\mathcal{B}^4$. From
Table~\ref{table:actionT}, the lossless operator for the first two steps
can be written as  
\[  \ba_1 \leftarrow \ba_1+ \ba_4   \quad \Rightarrow \quad \cL_{\cB,1} ( \ba )  = [ \ba_1 + \ba_4, \ba_2, \ba_3, \ba_4 ]^T ,\] 
and
\[  \ba_1 \leftarrow s_1(\ba_1), \quad \Rightarrow \quad \cL_{\cB,2} ( \ba )  = [s_1(\ba_1), \ba_2, \ba_3,\ba_4 ]^T ,  \] 
where we write $\cL_{\cB,i}$ to represent the action at the $i$th step
of the operator $\cL$ on an element of $\mathcal{B}^4$ from Table \ref{table:actionT}. The composite
operator for the first two steps can now be expressed as  
\[\cL_{\cB,2} \circ \cL_{\cB,1} ( \ba ) = [ s_1 (\ba_1 + \ba_4), \ba_2, \ba_3, \ba_4 ]^t.  \]
\red{Let} $\cL_{\cB}$ denote the action of $\cL$ in the vector space
$\mathcal{B}^4$. Continuing in the same manner as above, we have  
\red{\begin{align*}
\cL_\cB \ba &= \cL_{\cB,14} \circ \cdots \circ \cL_{\cB,1} (\ba), \\
&= \begin{bmatrix} s_{1} (s_{1} (\ba_4 + \ba_1) + s_{1} (\ba_2 + \ba_3)) \\
\bz - s_1(\bz+\bw) -  s_1(s_1( \bz+\bw)+ s_1(\bz - s_1(\bz+\bw))) \\
%\ba_2 - s_{1} (\ba_2 + \ba_3)-  s_{1} (\ba_4 - s_{1}  (\ba_4 + \ba_1)+ \ba_2 - s_{1} (\ba_2 + \ba_3))  \\
s_1 (\ba_2 + \ba_3) - s_1 ( s_1 ( \ba_4 + \ba_1 ) + s_1 (\ba_2 + \ba_3)) \\
s_1( \bz+\bw)+ s_1(\bw- s_1(\bz+\bw))
%s_1 (\ba_4 - s_{1}  (\ba_4 + \ba_1)+ \ba_2 - s_{1} (\ba_2 + \ba_3)) 
\end{bmatrix}, 
\end{align*}
where $\bz = \ba_4 - s_1(\ba_4+\ba_1)$ and $\bw = \ba_2-s_1(\ba_2+\ba_3)$. } Next, from Table \ref{table:actionTlossy}, we obtain the analogous lossy operator, denoted $\tilde{\cL}_{\cB}$. By replacing $s_1 (\cdot)$ by $r (\cdot)$ we obtain $\tilde{\cL}_{\cB} \ba$. Now 
\begin{align}
\|\cL\bx - \tilde{\cL}\bx\|_\infty = \| \cL_{\cB} \ba - \tilde{\cL}_{\cB} \ba \|_{\cB, \infty} = \max_{1 \leq i \leq 4} | f_\cB((\cL_{\cB} \ba)_i) - f_\cB((\tilde{\cL}_{\cB} \ba)_i) |. \label{LopMax}
\end{align}
In particular, we found that the maximum in
(\ref{LopMax}) is attained for $i = 2$. Using Lemma~\ref{lemma:boundr}
and using $s_{1}(c) = f_\cB^{-1}( f_{\cB} (c)/2) = \frac{1}{2}c$ for all
$c \in \cB$, we derive the following bound. Letting $\bm{y} = \cL_{\cB}
\ba$ and
$\tilde{\bm{y}} = \tilde{\cL}_{\cB} \ba$, we find that 
\red{\begin{align}
	\| \bm{y}_2 - \tilde{\bm{y}}_2 \|_{\cB,1} &= \left\|  \frac{1}{4}\left(\ba_2 +\ba_3 -\ba_1 -\ba_4\right) -   \left[\bz - s_1(\bz+\bw) -  s_1(s_1( \bz+\bw)+ s_1(\bz - s_1(\bz+\bw))) \right] \right\|_{\cB,1},  \nonumber \\ 
	&\leq \frac{1}{2} + \left\|  \frac{1}{4}\left(\ba_2 +\ba_3 -\ba_1 -\ba_4\right) -  \left[\bz - \frac{3}{2}s_1(\bz+\bw) -\frac{1}{2} s_1(\bz - s_1(\bz+\bw)) \right] \right\|_{\cB,1},  \label{LErrBoundEq1}\\
		&\leq \frac{1}{2} + \frac{1}{4} +\left\|  \frac{1}{4}\left(\ba_2 +\ba_3 -\ba_1 -\ba_4\right) -  \left[\frac{3}{4}\bz -  \frac{5}{4} s_1(\bz+\bw) \right] \right\|_{\cB,1},\nonumber \\
			&\leq \frac{1}{2} + \frac{1}{4} + \frac{5}{8} + \left\|  \frac{1}{4}\left(\ba_2 +\ba_3 -\ba_1 -\ba_4\right) -  \left[\frac{1}{8}\bz - \frac{5}{8}\bw \right] \right\|_{\cB,1}, \nonumber\\
			&\leq \frac{1}{2} + \frac{1}{4} + \frac{5}{8} +\frac{1}{16} + \frac{5}{16} = \frac{28}{16}\nonumber
	\end{align}}
%\begin{align}
%\| \bm{y}_2 - \tilde{\bm{y}}_2 \|_{\cB,1} &= \left\|  \frac{1}{4}\left(\ba_2 +\ba_3 -\ba_1 -\ba_4\right) -    s_1 (\ba_2 + \ba_3) + s_1 ( s_1 ( \ba_4 + \ba_1 ) + s_1 (\ba_2 + \ba_3)) \right\|_{\cB,1},  \nonumber \\ 
%&\leq \frac{1}{2} + \left\|  \frac{1}{4}\left(\ba_2 +\ba_3 -\ba_1 -\ba_4\right) -     \frac{1}{2} s_1 (\ba_2 + \ba_3) + \frac{1}{2}  s_1 ( \ba_4 + \ba_1 )  \right\|_{\cB,1},  \label{LErrBoundEq1}\\
%&\leq \frac{1}{2} + \frac{1}{4} +\frac{1}{4}+  \left\|  \frac{1}{4}\left(\ba_2 +\ba_3 -\ba_1 -\ba_4\right) -     \frac{1}{4}(\ba_2 + \ba_3) + \frac{1}{4} ( \ba_4 + \ba_1 )  \right\|_{\cB,1},  \nonumber \\  
%&\leq 1,  \nonumber \\  
%\end{align}
where (\ref{LErrBoundEq1}) follows from the triangle inequality and Lemma \ref{lemma:boundr}. Similarly, \red{$\| \bm{y}_i - \tilde{\bm{y}}_i \|_{\cB,1} \leq  \frac{28}{16}$}, for all $1 \leq i \leq 4$.  Since \red{$\|\bx \|_\infty \geq 2^{q-1}$}, it now follows that \red{
\[ \|\cL\bx - \tilde{\cL}\bx \|_\infty \leq  \frac{28}{16} 2^{-q+1} \|\bx\|_\infty = \frac{7}{4} \epsilon_q \|\bx\|_\infty. \] }
%Similarly, using the bit arithmetic implementation of $\cL^{-1}$ and $\tilde{\cL}^{-1}$ from Tables \ref{table:actionT} and \ref{table:actionTlossy}, it can be shown that  
%\[\|\cL^{-1}\bx - \tilde{\cL}^{-1}\bx \|_\infty  \leq \frac{5}{4}2^{-q} \|\bx\|_\infty= \frac{5}{8} \epsilon_q \|\bx\|_\infty. \]
\end{proof}
\red{The following result extends the 1-$d$ error caused by the lossy forward transform operator established in Lemma \ref{LemmaL1DBound} to $d$ dimensions. }
\begin{lemma}
\label{lemma:boundT}
Suppose $\bx \in \mathbb{Z}^{4^d}$ such that $e_{max, \mathcal{B}}(\bx) = q$. Then %Given the bit arithmetic implementation of the action of $\cL$ and $\cL^{-1}$ in Table \ref{table:actionTlossy}, we have
\begin{align*}
\|\cL_d\bx - \tilde{\cL}_d \bx \|_\infty & \leq k_\cL\epsilon_q \|\bx\|_\infty, 
\end{align*}
\red{where $k_\cL = \frac{7}{4} \left( 2^d - 1 \right)$. }
%\begin{align*}
%\|\cL^{-1}_d\bx -\tilde{\cL}^{-1}_d \bx \|_\infty & \leq   k_{\cL^{-1}}\epsilon_q \|\bx\|_\infty := (2^d - 1) \frac{5}{8} \left (\frac{15}{4} \right)^{d-1} \frac{3}{2}\left( 1- \left(\frac{1}{3} \right)^d \right)\epsilon_q \|\bx\|_\infty.
%\end{align*}
\end{lemma}
\begin{proof}
Let $\Delta \cL$ represent a perturbation of the action of  $\cL$ such that $ \tilde{\cL} = \cL+\Delta \cL$. From Lemma \ref{LemmaL1DBound}, we have \red{$\|\Delta \cL \by \|_\infty \leq \frac{7}{4} \epsilon_q \|\by\|_\infty$}, for all \red{$\by \in \{\bz \in \mathbb{Z}^4 : e_{max}(\bz)\geq q-1\}$.} Hence, 
\red{\begin{align*}
\|(I_{4^{d}} \otimes \Delta \cL) \bx \|_\infty = \| (\Delta \cL \otimes {I_{4^{d}}}) \bx \|_\infty \leq \frac{7}{4} \epsilon_q \| \bx \|_\infty.
\end{align*}} Using the inequalities $\|\Delta \cL\|_\infty \leq 1$ and $\| \cL\|_\infty =1$, we have
\red{\begin{align*}
\|\cL_d\bx - \tilde{\cL}_d \bx \|_\infty & = \| \left({\cL}\otimes \cdots\otimes  {\cL}\right) \bx - \left((\cL+\Delta \cL)  \otimes \cdots\otimes (\cL+\Delta \cL)\right) \bx \|_\infty ,\\
& \leq \left( \sum_{i=1}^d \left (\begin{matrix}d \\ i \end{matrix} \right )\|L\|^{d-i}_\infty \|\Delta L\|^{i-1}_{\infty} \right) \frac{7}{4} \epsilon_q \|\bx\|_\infty, \\
& = \frac{7}{4} \left( \sum_{i=1}^d \left (\begin{matrix}d \\ i \end{matrix} \right ) \right ) \epsilon_q \|\bx\|_\infty, \\
& = \frac{7}{4} \left( 2^d - 1 \right) \epsilon_q \|\bx\|_\infty.
%& = k_\cL \epsilon_q  \|\bx\|_\infty.
\end{align*}}
%\begin{align*}
%\|\cL_d\bx - \tilde{\cL}_d \bx \|_\infty & = \| {\cL}\otimes \cdots\otimes  {\cL} \bx - (\cL+\Delta \cL)  \otimes \cdots\otimes (\cL+\Delta \cL) \bx \|_\infty ,\\
%& \leq \left( \sum_{i=1}^d \left (\begin{matrix}d \\ i \end{matrix} \right )\|L\|^{d-i}_\infty \|\Delta L\|^{i-1}_{\infty} \right)\frac{7}{8} \epsilon_q \|\bx\|_\infty, \\
%& = \left( \|L\|^{d-1}_\infty \sum_{i=1}^d \left (\begin{matrix}d \\ i \end{matrix} \right )\|L\|^{-i+1}_\infty \|\Delta L\|^{i-1}_{\infty} \right )\frac{7}{8} \epsilon_q \|\bx\|_\infty, \\
%& \leq \frac{7}{8} \epsilon_q  \|\bx\|_\infty \|L\|^{d-1}_\infty (2^d - 1)\sum_{i=1}^d  \|L\|^{-i+1}_\infty \|\Delta L\|^{i-1}_{\infty}, \\
%& = \frac{7}{8} \epsilon_q  \|\bx\|_\infty (2^d - 1) \sum_{i=1}^d  \frac{7}{4}^{i-1}, \\
%& = \frac{7}{8} \epsilon_q  \|\bx\|_\infty (2^d - 1) \frac{4}{3}\left( \left( \frac{7}{4} \right)^d - 1\right), \\
%& = k_\cL \epsilon_q  \|\bx\|_\infty.
%\end{align*}
%The bound on $\|\cL^{-1}_d\bx -\tilde{\cL}^{-1}_d \bx \|_\infty$ can be derived by a similar argument.  
\end{proof}
\red{At this point, it remains to consider ZFP's backward linear transform with respect to Table~\ref{table:actionT} and Table \ref{table:actionTlossy}. For this particular implementation of ZFP, if Steps 3 through 8 of the compression algorithm are applied before the backwards transform, no additional error occurs\footnote{\red{The first two steps of the backwards transform operator, depicted in Table \ref{table:actionTlossy}, may result in round-off. However, the additional error that may occur depends on the user-defined parameters that define the action of Step 8. If Step 8 is performed losslessly, i.e., no bit planes are discarded, then each step of the backward transform, in bit arithmetic, undoes the associated step of the forward transform. 
If at least $2d$ bit planes are discarded at Step 8 (see Section \ref{Step8Sec}  for details), then the first two steps of the backwards transform will not introduce additional error.  
If between $1$ and $2d - 1$ bit planes are discarded, additional error may occur in the decompression step. However, since ZFP will result in a low compression ratio if only between $1$ and $2d-1$ bit planes are discarded, the remainder of the paper will assume at least $2d$ bit planes are discarded.
%, i.e., $\beta \leq q - 2d$}. 
See Appendix \ref{sec:appendixb} for details.}}.}
The decompression operator for the particular implementation of
ZFP is defined as the corresponding lossless operator  
\[ \tilde{D}_3(\ba) = {D}_3(\ba)= F_\cB^{-1}{\cL}_d^{-1}F_\cB(\ba), \text{ for all } \bm{a} \in \mathcal{B}^{4^d}.   \]

%%%%%%%%%%%%%%%%%%%%%%%%%%%%%%%%%%%%%%%%%%%%%%%%%%%%%%%%%%%%%%%%
\subsection{Reorder coefficients by total sequency}
\label{Step4Sec}
The fourth step performs a deterministic permutation on the components of the input. As such, it is an invertible operation. We define $C_4: \mathcal{B}^{4^d} \to \mathcal{B}^{4^d}$ to be the map that takes the components of a block in row-major order and permutes them so that the resulting block is in total sequency order \cite{zfp-doc}. The decompression operator performs the inverse permutation such that $D_4 C_4 (\bm{a}) = \bm{a}$, for all $\bm{a} \in \mathcal{B}^{4^d}$. We summarize the key details for these operators below.
\begin{prop}\label{Step4Prop}
Suppose $\bm{a} \in \mathcal{B}^{4^d}$. Then $\| C_4 (\bm{a}) \|_{\mathcal{B}, p} = \| \bm{a} \|_{\mathcal{B}, p} =\| D_4 (\bm{a}) \|_{\mathcal{B}, p}$, for all $1 \leq p \leq \infty$.
\end{prop}

\subsection{Convert signed two's complement to negabinary}
\label{Step5Sec} 
At Step 5 of the algorithm, each component is converted from its two's complement representation to a negabinary representation. As we are representing values using a signed binary representation instead of a two's complement representation for our analysis, we will need to convert each signed binary representation to a negabinary representation. Using the operators defined in Section \ref{sec:notation}, we define the operator $C_5 : \mathcal{B}^{4^d} \to \mathcal{N}^{4^d}$ by
\begin{align*}
C_5 (\bm{a}) := F_{\cN}^{-1} F_\cB (\ba), \ \ \text{for all} \ \bm{a} \in \mathcal{B}^{4^d}.
\end{align*}
%As the inputs at this step are integers, it is not necessary to consider error due to round off since the elements in the output can be represented using the same number of bits as the input (including the sign bit). Hence, Step 5 is lossless. Lastly, it follows that the decompression operator is defined as $D_5 := F_{\cB}^{-1}F_\cN$. The following result summarizes the key result from this step used in the analysis in~\secref{sec:bounds}.

{A valid concern for this step is that the range of representable integers for an $N$-bit two's complement representation is not the same as the range of representable integers for an $N$-bit negabinary representation, for any integer $N \geq 2$. To account for this difference, ZFP uses an $(N-1)$-bit two's complement representation with one bit left unused, called a guard bit. In Step 3, the guard bit was required for the decorrelating transform but is unnecessary for the remaining steps. Thus, when the two's complement representation is converted to a negabinary representation in ZFP the guard bit is freed and used instead for an $N$-bit negabinary representation to ensure that the integer can be represented. Additionally, since the magnitude of each component is not increased in the following steps, the components can be converted back to two's complement without introducing any error due to round off. Hence, Step 5 is lossless. Lastly, it follows that the decompression operator is defined as $D_5 := F_{\cB}^{-1}F_\cN$. The following result summarizes the key result from this step used in the analysis in~\secref{sec:bounds}.}
%As the inputs at this step are signed integers with a guard bit, it is not necessary to consider error due to round off since the elements in the output can be represented exactly in negabinary using the allowable number of bits ($q +2$ bits including the sign bit and the guard bit). 

\begin{prop}\label{Step5Prop}
Suppose $\bm{a} \in \mathcal{B}_{k}^n$. Then $\| C_5 (\bm{a}) \|_{\mathcal{N}, p} = \| \bm{a} \|_{\mathcal{B}, p}$ for all $1 \leq p \leq \infty$.
\end{prop}

\subsection{Boolean matrix transposition}
\label{Step6Sec}
Next, the bit vectors are reordered by their bit index instead of their associated binary representation.
Under the bit vector representation, this corresponds to transposing the entire block. Since this operation is lossless and does not result in altering the representation of any element in the block, we do not define an operator here. For simplicity, we will work under the assumption that the transposition did not take place.

\subsection{Embedded block coding}
\label{Step7Sec} 
In Step 7, each bit plane of $4^d$ bits is individually coded with a variable-length code that is one to one and reversible (see \cite{zfp-doc} for details).  For purposes of the analysis, since Step 7 is lossless, we chose not to consider the encoding in Step 7 since the error analysis can be considered in any format. Hence, for the purposes of
simplifying the analysis, we take $C_7 = D_7 = I_{\mathcal{N}}$. 

%At this point, six of the eight steps have been considered. Before
%considering Step 7, we note that, while there are three different
%options for Step 8, our current focus is the fixed precision mode. The
%fixed precision mode encodes a fixed number of bit planes and discards
%the remaining bit planes. For purposes of the analysis, we do not need
%to consider the encoding in Step 7 when using fixed precision mode
%during Step 8 since Step 7 losslessly encodes each bit plane
%individually. Since the fixed precision option keeps a fixed number of
%bit planes, the bit planes provided as input at Step 8 can be considered
%in any format for the error analysis. Hence, for the purposes of
%simplifying the analysis, we take $C_7 = D_7 = I_{\mathcal{N}}$. 

\subsection{Finite-precision: Bit stream truncation}
\label{Step8Sec} 
\black{Step 8 is dependent on one parameter, denoted $\beta \geq
0$, and an index set dependent on $\beta$ and the input, denoted as $\mathcal{P}$.} Here, $\beta$ represents the number of most significant bit planes to keep
during Step 8 and any discarded bit plane is replaced with all-zero bits. \black{Note that 
	the value of $\beta$ corresponds to the parameter \texttt{zfp\_stream.maxprec} in ZFP and can be set to any positive integer by the user in the fixed precision %DIF > or expert 
	mode of ZFP.} The operator for Step 8 is given by $\tilde{C}_8:
\cN^{4^d} \rightarrow \cN^{4^d}$ and defined as 
\begin{align*}
\tilde{C}_8(\bm{d}) := T_{\mathcal{P}} (\bm{d}), \text{ for all } \bm{d} \in \mathcal{N}^{4^d},
\end{align*}
where \black{$\mathcal{P} = \{ i \in \mathbb{Z} : i > q + 1 - \beta \}$,} $q \in \mathbb{N}$ \black{is the value} from Step 2, and $T_\cP$ is the truncation operator with respect to set $\cP$. The lossless compression and decompression operators are then 
defined by $C_8 := I_\cN$ and $D_8 := I_\cN$, respectively. \black{We conclude this step with a proposition that immediately follows from Lemma~\ref{ProjectionLemma2}.}
%\begin{prop}
%\label{Step8Prop}
%Suppose $\bm{d} \in \mathcal{N}^{4^d}$. Then $\| \tilde{C}_8 \bm{d} - C_8 \bm{d} \|_{\mathcal{N}, \infty} \leq 2\epsilon_{\beta} \| \bm{d} \|_{\mathcal{N}, \infty}$.
%\end{prop}
%\begin{proof}
%Follows immediately from Lemma \ref{ProjectionLemma2}.
%\end{proof}
\begin{prop}
	\label{Step8Prop}
	Suppose $\bm{a} \in \mathcal{B}^{4^d}$ such that $F_{\mathcal{B}} (\bm{a}) \in \mathbb{Z}^{4^d}$ and \red{$e_{max, \mathcal{B}} (F_{\mathcal{B}} (\bm{a})) \geq q-1$.} Then $\| \tilde{C}_8 C_5 \bm{a} - C_8 C_5 \bm{a} \|_{\mathcal{N}, \infty} \leq \frac{8}{3}  \epsilon_\beta \| \bm{a} \|_{\mathcal{B}, \infty}$.
\end{prop}
\begin{proof}
	Let $\bm{d} = C_5 \bm{a}$. From Lemma \ref{ProjectionLemma2} ($ii$) we have that
	\black{\begin{align}
	\| \tilde{C}_8 \bm{d} - C_8 \bm{d} \|_{\mathcal{N}, \infty} &= \| T_{\mathcal{P}} \bm{d} - \bm{d} \|_{\mathcal{N}, \infty} \leq \frac{2}{3} \epsilon_{\beta} 2^{q + 1}. \label{Step8PropEQ0}
	\end{align}}
	From the assumption \red{$e_{max, \mathcal{B}} (F_{\mathcal{B}} (\bm{a})) \geq q-1$} it now follows that \red{$\|\bm{a}\|_\infty \geq 2^{q-1}$} and 
	\begin{align*}
	\| \tilde{C}_8 C_5 \bm{a} - C_8 C_5 \bm{a} \|_{\mathcal{N}, \infty} &\leq \frac{8}{3 }\epsilon_{\beta} \| \bm{a} \|_{\infty}. 
	\end{align*} 
\end{proof}

To conclude this section, it should be noted that the inputs at Step 5 of ZFP satisfy the hypotheses of Proposition \ref{Step8Prop} as each component is encoded as an integer up to precision $q$.

\subsection{Defining the ZFP Compression Operator}
\label{ZFPCompOperator}
To conclude this section, we define the ZFP fixed precision compression
and decompression operators by composing the operators defined for each
step of the algorithm. In order to simplify the definition of each
operator, we omit $C_7$, $D_7$, $C_8$, and $D_8$ from the composition, as
they were defined to be the identity operator $I_{\mathcal{N}}$. 

{\begin{definition}
		The lossy fixed precision compression operator, $\tilde{C}: \mathbb{R}^{4^d} \to \mathcal{N}^{4^d}$, is defined by 
		\begin{align*}
		\tilde{C} (\bm{x}) = \left( \tilde{C}_{8} \circ {C}_{5} \circ {C}_{4} \circ \tilde{C}_{3} \circ \tilde{C}_{2} \right) (\bm{x}), \ \ \ \text{for all} \ \bm{x} \in \mathbb{R}^{4^d},
		\end{align*}
		where $\circ$ denotes the usual composition of operators. The lossless fixed precision compression operator, $C: \mathbb{R}^{4^d} \to \mathcal{N}^{4^d}$, is defined by 
		\begin{align*}
		C(\bm{x}) = \left(C_{5} \circ C_{4} \circ C_{3} \circ C_{2} \right) (\bm{x}), \ \ \ \text{for all} \ \bm{x} \in \mathbb{R}^{4^d}.
		\end{align*} 
		Lastly, the lossy fixed precision decompression operator, $\tilde{D}:  \mathcal{N}^{4^d} \to \mathbb{R}^{4^d}$, is defined by 
		\begin{align*}
		\tilde{D}(\bm{d}) = \left( \tilde{D}_{2} \circ {D}_{3} \circ D_{4} \circ D_{5} \right) (\bm{d}), \ \ \ \text{for all} \ \bm{d} \in \mathcal{N}^{4^d},
		\end{align*} 
		and the the lossless fixed precision decompression operator ${D} : \mathcal{N}^{4^d} \to \mathbb{R}^{4^d}$ is defined by 
		\begin{align*}
		{D}(\bm{d}) = \left( D_{2} \circ {D}_{3} \circ D_{4} \circ D_{5} \right) (\bm{d}), \ \ \ \text{for all} \ \bm{d} \in \mathcal{N}^{4^d}.
		\end{align*} 
\end{definition}}

%% file: bounds.tex
\section{Error Bounds for ZFP Compression and Decompression}
\label{sec:bounds}
Now that the ZFP fixed precision compression and decompression operators
have been defined, we can establish a bound on the forward error for an
arbitrary input that is compressed and decompressed. We begin by
analyzing the error introduced during compression. Recall that $\beta$
is the fixed precision parameter, i.e., $\beta$ bits for each of the ZFP
transform coefficients will be kept during compression.  
\begin{lemma}\label{lemma:difftildeCandC} Assume $\bx \in \mathbb{R}^{4^d}$ with $\bm{x} \neq \bm{0}$ such that $F^{-1}_{\mathcal{B}} (\bx) \in \cB_k^{4^d}$, for some precision $k$. Let {$\beta \geq  0$} be the fixed precision parameter. Then 
\begin{align*}
\|\tilde{C} \bx - C \bx\|_{\mathcal{N}, \infty} \leq 2^{-\ell} \left( \frac{8}{3} \epsilon_\beta+ \epsilon_q \left(1+ \frac{8}{3} \epsilon_\beta \right) \left(k_\cL(1+\epsilon_q)+1 \right )\right) \|\bx\|_\infty,
\end{align*}
where $q \in \mathbb{N}$ is the precision for the block-floating point representation in Step 2, \red{$\ell = e_{max,\cB}(\bx) - q +1$}, and \red{$k_\cL= \frac{7}{4} (2^d-1)$.} %and $\epsilon_L = \frac{7}{4} \epsilon_{\cS}$. 
\end{lemma}
\begin{proof} Define $c(\bm{x}) := \|\tilde{C}\bx - C \bx\|_{\cN,\infty}$. First, we find that
\begin{align*}
c (\bm{x}) &= \| \tilde{C}_8 C_5 C_4 \tilde{C}_3 \tilde{C}_2 \bx - C \bx \|_{\cN,\infty}, \\
 &= \| \tilde{C}_8 C_5 C_4 \tilde{C}_3 \tilde{C}_2 \bx - C_5 C_4 \tilde{C}_3 \tilde{C}_2 \bx +C_5 C_4 \tilde{C}_3 \tilde{C}_2 \bx- C \bx \|_{\cN,\infty}, \\
&\leq \| \tilde{C}_8 C_5 C_4 \tilde{C}_3 \tilde{C}_2 \bx - C_5 C_4 \tilde{C}_3 \tilde{C}_2 \bx \|_{\mathcal{N}, \infty} + \| C_5 C_4 \tilde{C}_3 \tilde{C}_2 \bx - C \bx \|_{\mathcal{N}, \infty}. 
\end{align*}
By the definition of $C_4$ and $C_5$, we have $\| C_5 C_4 \tilde{C}_3 \tilde{C}_2 \bx - C \bx \|_{\mathcal{N}, \infty} = \| \tilde{C}_3 \tilde{C}_2 \bx - C_3 C_2 \bx \|_{\mathcal{B}, \infty}$. Additionally, $\| \tilde{C}_8 C_5 C_4 \tilde{C}_3 \tilde{C}_2 \bx - C_5 C_4 \tilde{C}_3 \tilde{C}_2 \bx \|_{\mathcal{N}, \infty} \leq 8 \epsilon_\beta \|\tilde{C}_3\tilde{C}_2 \bx \|_{\mathcal{B}, \infty}$ follows by applying Proposition \ref{Step8Prop} and the definition of $C_4$. Hence,
\begin{align*}
c (\bm{x}) &\leq 8 \epsilon_\beta \| \tilde{C}_3 \tilde{C}_2 \bx \|_{\mathcal{B}, \infty} + \| \tilde{C}_3 \tilde{C}_2 \bx - C_3 C_2 \bx \|_{\mathcal{B}, \infty} ,\\
& \leq \left( 1 + \frac{8}{3} \epsilon_\beta \right) \| \tilde{C}_3 \tilde{C}_2 \bx - C_3 \tilde{C}_2 \bm{x} \|_{\mathcal{B}, \infty} + \frac{8}{3} \epsilon_\beta \| {C}_3 \tilde{C}_2 \bx \|_{\mathcal{B}, \infty} + \|{C}_3\tilde{C}_2 \bx -C_3 C_2 \bx \|_{\mathcal{B}, \infty}, \\
& \leq \left( 1 + \frac{8}{3} \epsilon_\beta \right) \| \tilde{C}_3 \tilde{C}_2 \bx - C_3 \tilde{C}_2 \bm{x} \|_{\mathcal{B}, \infty} + \frac{8}{3} \epsilon_\beta \| \tilde{C}_2 \bx \|_{\mathcal{B}, \infty} + \| \tilde{C}_2 \bx - C_2 \bx \|_{\mathcal{B}, \infty},
\end{align*}
where the final inequality follows from the linearity of $C_3$ and $\|
C_3 \|_{\mathcal{B}, \infty} \leq 1$. By the definition of
$\tilde{C}_2$, we have that \red{$e_{max, \cB} (\tilde{C_2}\bx) \geq
q - 1$}. Hence, Lemma \ref{lemma:boundT} yields that $\|
\tilde{C}_3\tilde{C}_2 \bx -C_3\tilde{C}_2\bx\|_{\cB,\infty} \leq k_\cL
\epsilon_q \| \tilde{C}_2 \bm{x} \|_{\cB, \infty}$. Lastly, using
Proposition \ref{Step2Prop}, we have that $\| \tilde{C}_2 \bx - C_2 \bx
\|_{\mathcal{B}, \infty} \leq 2^{-\ell} \epsilon_q \|\bx\|_\infty$,
which yields the inequality $\| \tilde{C}_2 \bx\|_{\mathcal{B}, \infty}
\leq 2^{-\ell} (1+\epsilon_q) \|\bx\|_\infty$. 
Combining these observations provides the desired result.
%\begin{align*}
%\|\tilde{C}_p\bx - C_p \bx\|_{\cN,\infty}  & \leq 2^{-l} \left( \epsilon_\beta+ \frac{11}{4} \epsilon_\cS + O(\epsilon_\beta\epsilon_\cS)\right) \|\bx\|_\infty.
%\end{align*}
\end{proof}
The following result provides {bound on the error} resulting from compressing then decompressing a $4^d$ block using ZFP. 
\begin{theorem}\label{thm:diffDCandDC} 
Assume $\bx \in \mathbb{R}^{4^d}$ with $\bm{x} \neq \bm{0}$ such that $F_{\mathcal{B}} (\bx) \in \cB_k^{4^d}$, for some precision $k$. \red{Let $0\leq \beta \leq  q- 2d+2$ be the fixed precision parameter.}\footnote{\red{In other words, it is assumed that at least $2d$ least significant bit planes are discarded in Step 8. If less than $2d$ bit planes are discarded, i.e.,  $q- 2d +2 <\beta <q+2 $, error will occur in Step 3 from round-off that may occur by the decompression operator, which is not taken into account in Theorem \ref{thm:diffDCandDC}.  Theorem \ref{thm:diffDCandDCappendix} is the generalization of Theorem \ref{thm:diffDCandDC} for the assumption $q- 2d +2 <\beta <q+2 $.  See Appendix \ref{sec:appendixb} for details.}} Then 
\begin{align}
\| \tilde{D}\tilde{C} \bx - \bx \|_\infty &\leq K_\beta \|\bx\|_\infty
\end{align}
where $q \in \mathbb{N}$ is the precision for the block-floating point representation in Step 2,
\begin{align}
K_\beta := \left( \frac{15}{4} \right)^d \left(  (1+\epsilon_k)\left( \frac{8}{3} \epsilon_\beta+ \epsilon_q \left(1+ \frac{8}{3} \epsilon_\beta \right) \left(k_\cL(1+\epsilon_q)+1 \right )\right)+\epsilon_k \right),
\end{align}
\red{and $k_\cL= \frac{7}{4} (2^d-1)$.} 
\end{theorem}
\begin{proof} 
Observe that
\begin{align}
 \|\tilde{D} \tilde{C}\bx - D C \bx \|_\infty &= \| \tilde{D}_2{D}_3 D_4 D_5 \tilde{C} \bx - D_2 D_3 D_4 D_5 C \bx \|_\infty, \nonumber \\ 
&\leq  \| \tilde{D}_2{D}_3 D_4 D_5 \tilde{C} \bx -{D}_2{D}_3 D_4 D_5 \tilde{C} \bx\|_\infty + \|{D}_2{D}_3 D_4 D_5 \tilde{C} \bx - D_2 D_3 D_4 D_5 C \bx \|_\infty, \nonumber \\ 
&\leq 2^\ell\epsilon_k\|{D}_3 D_4 D_5 \tilde{C} \bx\|_{\cB,\infty}+\| D_2{D}_3 D_4 D_5\|   \| \tilde{C} \bx -  C \bx \|_{\mathcal{N}, \infty}, \label{5.2eq1} \\
%%cut out this step to save space%% &= \| S_l \|_{\cB,\infty} \| \tilde{D}_3 D_4 D_5 \tilde{C}_p \bx - D_3 D_4 D_5 \tilde{C}_p \bx + D_3 D_4 D_5 \tilde{C}_p \bx - D_3 D_4 D_5 C_p \bx \|_\infty \\
&\leq 2^{\ell} \left ( \frac{15}{4} \right)^d \left(\epsilon_k\| \tilde{C} \bx \|_{\black{\mathcal{N}}, \infty} +  \| \tilde{C} \bx -  C \bx \|_{\mathcal{N}, \infty}\right),  \label{5.2eq2} \\
&\leq 2^{\ell} \left ( \frac{15}{4} \right)^d \left((1+\epsilon_k) \| \tilde{C} \bx -  C \bx \|_{\mathcal{N}, \infty} + \black{\epsilon_k}\|{C} \bx \|_{\mathcal{N}, \infty} \right), \nonumber \\ %\tag{$\triangle$-inequality }
&\leq 2^{\ell} \left ( \frac{15}{4} \right)^d \left((1+\epsilon_k) \| \tilde{C} \bx -  C \bx \|_{\mathcal{N}, \infty} + \black{2^{-\ell}\epsilon_k}\|\bx \|_{\infty} \right), \label{5.2eq3} 
\end{align}
where (\ref{5.2eq1}) follows from Proposition~\ref{Step2Prop}($ii$) and (\ref{5.2eq2}) follows from the linearity of $D_2$, $D_3$, $D_4$, and $D_5$ and \red{$\ell = e_{max,\cB} (\bx) - q + 1$}. 
Applying Lemma~\ref{lemma:difftildeCandC} in (\ref{5.2eq3}) yields the desired result.
\end{proof}
\DIFdelbegin %DIFDELCMD < \noindent %%%
\DIFdelend \DIFaddbegin 

%DIF >  Discussion moved to appendix but left here in case we want to put it back in this section
\iffalse
\DIFadd{Since the constant $K_{\beta}$ appears in the bound and it is dependent on $k$, $q$, $d$, and $\beta$, we provide a brief discussion on how $K_{\beta}$ is affected by changes in these values. First, since $q$ and $k$ depend on the precision of the data provided to ZFP and $d$ is the dimension of the input data it is important to note that three of the values in $K_{\beta}$ are dependent solely on the input data. The remaining value, $\beta$, can be set to any positive integer when using the fixed precision mode of ZFP, as noted in Section \ref{Step8Sec}. Figure \ref{fig:KBetaPlot} helps illustrate how $K_{\beta}$ varies with respect to the dimension of the data, $d$, and $\beta$. The lines on the surface in Figure \ref{fig:KBetaPlot} represent where the value of $K_{\beta}$ exceeds $10^{-6}$ and $1$. As suspected from the formula for $K_{\beta}$, we observe that a larger value of $d$ has a greater effect on the value of $K_{\beta}$ for small values of $\beta$.
}

\begin{figure}[h!]
\centering
\includegraphics[width=\linewidth]{figures/KBPlot.png}
\caption{\DIFaddFL{Visualization of $K_{\beta}$ for $\beta \in [1, 32]$ and dimension $d \in [2, 5]$}}
\label{fig:KBetaPlot}
\end{figure}
\fi

\black{Since the constant $K_{\beta}$ appears in the bound, which is dependent on $k$, $q$, $d$, and $\beta$, we provide a brief discussion on $K_{\beta}$ in Appendix \ref{sec:appendixc}.} Note that Theorem \ref{thm:diffDCandDC} yields the following bound on the maximum of the component-wise relative error: 
\begin{align}
\max_{i, \bx_i \neq 0} \left | \frac{(\tilde{D} \tilde{C} \bx)_i -\bx_i}{ \bx_i} \right| & \leq  \frac{1}{\min_{i, \bx_i \neq 0}|\bx_i| }\left  \| \tilde{D} \tilde{C}\bx -\bx\right \|_\infty \leq K_\beta \black{2^{e_{max,\cB}(\bx) - e_{min,\cB}(\bx)}}.
\end{align}

{So far, the discussion and error analysis has focused on the fixed precision mode. However, as mentioned during the introduction, ZFP also has a fixed accuracy and fixed rate mode. While we will not spend much time providing details for the fixed accuracy and fixed rate modes, it should be noted that the error bound in Theorem \ref{thm:diffDCandDC} allows us to develop error bounds for the fixed accuracy and fixed rate modes.}

{In the fixed accuracy mode, the transform coefficients in each $4^d$ block are encoded up to a minimum bit plane number. The index of the minimum bit plane will be dependent on the largest absolute magnitude and the constant $K_\beta$ found in Theorem~\ref{thm:diffDCandDC}. The following theorem is an extension of Theorem~\ref{thm:diffDCandDC} for fixed accuracy mode of ZFP. 
\begin{theorem}\label{thm:diffDCandDCFixedAccuracy} 
Assume $\bx \in \mathbb{R}^{n}$ with $\bm{x} \neq \bm{0}$ such that $F_{\mathcal{B}} (\bx) \in \cB_k^{n}$, for some precision $k$ and let $\hat{\bx}$ represent the compressed and decompressed values from using the fixed accuracy mode of ZFP. To guarantee $b\in \mathbb{N}$ bits of accuracy, i.e., $\| \hat{\bx} - \bx \|_\infty  \leq 2^{-b}$, $\beta$ must satisfy: 
\begin{equation} \label{eqn:fixedaccuracy} \beta \geq \log_2\left(\frac{ \frac{16}{3}(1+ c)} {\frac{\left(\left( \frac{4}{15} \right)^d  2^{-b-e_{max}} - \epsilon_k\right)}{(1+\epsilon_k)}- c} \right),  \end{equation}
where $c = \epsilon_q \left(k_\cL(1+\epsilon_q)+1\right)$ and $e_{max}:= e_{max}(\bx)$. 
\end{theorem}
\begin{proof}
Let $\bx^i$ denote the $i$-th $4^d$ block of the $d$-dimensional data $\bx \in \mathbb{R}^{n}$. Then we have that $e_{max}:= e_{max}(\bx) = \max_i e_{max}(\bx^i)$. From the hypothesis it follows that $K_{\beta}\leq 2^{-b-e_{max}}$. From Theorem ~\ref{thm:diffDCandDC} and the fact that $\|\bx^i\|_\infty\geq 2^{e_{max}}$ for all $i$, we conclude that 
\begin{align*}
\| \hat{\bx} - \bx \|_\infty &= \max_i \| \tilde{D} \tilde{C}\bx^i - \bx^i \|_\infty \leq  \max_i K_{\beta} \|\bx^i\|_\infty \leq 2^{-b}.
\end{align*}
\end{proof}}
\noindent If we assume $k = q = \infty$, i.e., infinite precision, then Equation (\ref{eqn:fixedaccuracy}) simplifies to  
\begin{equation}  \beta \geq \log_2\left(\left( \frac{15}{4} \right)^d  \frac{16}{3}2^{b+e_{max}}  \right).  \end{equation}
{Similarly, an upper bound for the fixed rate mode of ZFP can be obtained using Theorem ~\ref{thm:diffDCandDC}. For the purposes of understanding the following result, it suffices to know that for the fixed rate mode of ZFP the user provides a maximum rate, denoted $r$, or number of bits per value to be stored.  
\begin{theorem}\label{thm:diffDCandDCFixedRate} 
Assume $\bx \in \mathbb{R}^{4^d}$ with $\bm{x} \neq \bm{0}$ such that $F_{\mathcal{B}} (\bx) \in \cB_k^{4^d}$, for some precision $k$. Let $b_e$ be the number of bits to encode the exponent and let $\hat{\bx}$ represent the compressed and decompressed values in the fixed rate mode with rate, $r\in \mathbb{N}$. For some $\beta \in \mathbb{N}$, if 
\[ \|\hat{\bx} - \bx \|_\infty \leq K_\beta \|x\|,  \]
then $r \geq\frac{4^d\beta +b_e}{4^d} +1$. 
\end{theorem}
\begin{proof}
In the worst case scenario, the first bit plane is all-ones, which would imply $4^d - 1$ positive group tests (see \cite{zfp-doc} for details) and thus, $4^d-1$ bits to encode the the group tests. Each bit plane would then take $4^d$ bits to encode. Note that, ZFP uses one bit to indicate if the block is all zeros. Thus, if given rate $r$, there is a total of $4^d r$ bits that can be used to encode the block, $b_e$ of those bits must be used to encode the block floating-point exponent in Step 2, one bit is used for the leading all-zeros bit and $4^d-1$ bits for the group testing, leaving $(4^dr - b_e-1-(4^d-1))$ bits to encode the bit planes. Implying 
\begin{equation}\beta \leq \frac{(4^d(r-1) - b_e)}{4^d} \Rightarrow  r \geq \frac{4^d\beta +b_e}{4^d} +1.\end{equation}
\end{proof}}

Now that we have established bounds on the error introduced by ZFP compression and decompression, we consider several numerical experiments in order to observe the tightness of these bounds.
%\aedit{The implication of Theorem \ref{IterativeProp2}, is that the error introduced by repeated compression and decompression in conjunction with an advancement operator can be bounded by factors dependent only on the advancement operator, the truncation introduced by ZFP, and the magnitude of the data. In the following section, various numerical experiments are provided.}

%% file: results.tex
\section{Numerical Experiments}
\label{sec:results}
In the following numerical tests, we consider two types of error, which
we will refer to as \emph{block relative error} and \emph{componentwise
  relative error}, and their respective bounds:  
\begin{itemize}
\item[]{Block Relative Error:} $\displaystyle \frac{\|\tilde{D} \tilde{C} \bx -\bx\|_\infty}{\|\bx\|_\infty} \leq K_{\beta}$,
\item[]{Componentwise Relative Error:}  $\displaystyle \max_{i, \bx_i \neq 0} \frac{| \tilde{D} \tilde{C} \bx_i - \bx_i |} {|\bx_i|} \leq K_{\beta} \black{2^{e_{max}-e_{min}}}$.
\end{itemize}
As observed in Section~\ref{sec:analysis}, Steps 2, 3, and 8 of ZFP are
the sources of round-off error. So, in constructing our numerical tests,
we considered what conditions will vary the round-off error at these
steps. Information is only lost in Step 2, if, when the block is
converted from floating-point to block-floating-point, the exponent
range of components in the block is above some threshold, as can be seen
in equation (\ref{eqn:componentwiseerror}).  
%For example, in single precision ($ q = 31$, $k = 24$), from equation (\ref{eqn:componentwiseerror}), if $e_{max} - e_{min} \leq 8$,  then no distortion greater than machine epsilon will occur. 
At Step 3, a linear transform is applied to the block, and information
is lost whenever round-off occurs. At Step 8, if the number of
compressed bit-planes does not coincide with the negabinary precision,
we will again lose information. Since varying the exponent range of the
input is an easy parameter to control, the value $e_{\max} - e_{\min}$
is used as a parameter in many of the numerical tests. Additionally, we
chose to vary the number of bit planes kept in Step 8, denoted $\beta$,
in the following numerical experiments. 

The first numerical experiment is designed to test how well the bound
established in Section~\ref{sec:bounds} captures the round-off error
introduced by ZFP as the range of exponents and the number of bit-planes
varies within a single block with dimension $d$.  While the first test
works on data generated to demonstrate the worst-case behavior, the
second experiment shows the behavior for a data set taken from an actual
physical simulation.
%The final two experiments demonstrate the
%application of ZFP inline, with repeated decompression and recompression of the
%data inbetween updates for canonical time-evolving PDEs, to demonstrate
%the applicability of the bound in Theorem~\ref{IterativeProp2}. 

\subsection{Generated $4^d$ Block}
\label{sec:generatedBlock}
In the first numerical test, a 4 by 4 block was formed with absolute
values ranging from $2^{e_{min}}$ to $2^{e_{max}}$. The exponent
$e_{min}$ remains stationary while $e_{max}$ varies, depending on the
chosen exponent range. The interval $[e_{min}, e_{max}]$ was divided
into 16 evenly spaced subintervals. Each value of the block was randomly 
selected from a uniform distribution in the range
$[2^{e_{min}+(h-1)\frac{e_{max}-e_{min}}{16}},
  2^{e_{min}+h\frac{e_{max}-e_{min}}{16}}]$ with {sub}interval
$h\in \{1,\dots,16 \}$ and uniform randomly assigned sign. The block was
then randomly permuted, using the C$++$ standard library function
{\it{random\_shuffle}}, to remove any bias in the total sequency order
and then compressed and decompressed with precision, $\beta$. This
specific construction of data is designed to {mimic} the worst
possible input for ZFP for a chosen exponent range. For a given exponent
range, $\rho = e_{max}-e_{min} $, we expect the componentwise relative
error and block relative error to increase as the precision
decreases. However, {as the bound is only dependent on $\beta$,}
the block relative error bound will remain constant as the exponent
range varies. The componentwise relative error should increase as the
exponent range increases as the representable numbers in a
block-floating-point representation are dependent on the largest
exponent of the block and the value of $q$.   

For
Figures~\ref{fig:componenterror_precision_4by4}-\ref{fig:normerror_exp_4by4},
the data is represented and compressed in single-precision (32-bit IEEE
standard) with $e_{min}  = 0$, while $e_{max}$ varies with respect to
the required exponent range. Similar results {can be produced} for
any value of $e_{min}$. Figure~\ref{fig:componenterror_precision_4by4}
shows how the {componentwise relative error} (top) and block relative
error (bottom) vary with respect to the fixed precision parameter,
$\beta$, for a fixed exponent range.  For a single $\beta$, {one}
%million blocks generated using the above routine were compressed and
decompressed. The blue band represents the sampled maximum and minimum
of the true componentwise relative error or block relative error, i.e.,  
\begin{align*}
\max_{i, \bx_i \neq 0} \frac{|\tilde{D}\tilde{C}\bx_i -\bx_i|} {|\bx_i|} \quad {\text{and}}  \quad \frac{\|\tilde{D}\tilde{C}\bx -\bx\|_\infty}{\|\bx\|_\infty},
\end{align*}
respectively, of all 1 million runs. The red line depicts the
theoretical bound and the dashed green line represents the asymptotic
behavior of the bound, i.e., the smallest predictive value of the
theoretical bound. For $ e_{max} - e_{min} = 0$, meaning that the
magnitude of the absolute values of the 4 by 4 block are similar, the
componentwise relative error increases as $\beta$ decreases. As the
exponent range increases, {fewer bits will be
  used to represent the smaller values in each block during Step 2 of
  ZFP, which will result in a larger relative error}. As anticipated,
{in Figure \ref{fig:componenterror_precision_4by4}} the entire plot
{shifts} toward the upper right corner as the exponent range
increases, indicating that the componentwise relative error increases
with respect to the range of compressed values. With respect to the
block relative error, the block relative error remains the same for all
$\beta$ as the exponent range varies.  

Similar trends can also be seen in Figure~\ref{fig:normerror_exp_4by4},
where the exponent range varies for a single $\beta$. In
Figure~\ref{fig:normerror_exp_4by4}, for each $\beta$, the componentwise 
relative error (top) increases as the exponent range increases while the
block relative error remains constant, as expected. As $\beta$ increases,
both the componentwise relative error and the block relative error plots
are shifted upwards, indicating an increase in error. For $\beta = 32$, there is a gap between the bound and the observed error, which corresponds to the gap in the far right of the plots in Figure~\ref{fig:componenterror_precision_4by4}, i.e., the theoretical error bound is limited to precision of the IEEE representation. It can be concluded, in
Figures~\ref{fig:componenterror_precision_4by4} and 
\ref{fig:normerror_exp_4by4}, that the theoretical bound (red) completely
bounds the maximum sampled error (blue). 

Next, we repeated this same experiment using a different machine
precision. Note that Figures~\ref{fig:2d_double_1} and
\ref{fig:2d_double_4} represent the same two-dimensional test outlined
above, but the values in the block are represented using double-precision
(64-bit IEEE standard). The same relationships can be concluded for the
double precision case.  

Finally, since the dimensionality of the block plays an important part
in ZFP, Figures~\ref{fig:1d_double_1} and \ref{fig:1d_double_4} and
Figures~\ref{fig:3d_double_1} and \ref{fig:3d_double_4} represent
results in double precision for one-dimensional and three-dimensional
blocks, respectively. Again, similar relationships can be seen as those
outlined in the two-dimensional, single-precision experiment.
%{The minimum ratio between the
%  theoretical bound and the maximum sampled error is $\approx 2.23$, for all tests so far presented.} 

%%%%%%%%%%%% %%%%%%%%%%%%%%%%%%%%%%%%%%%%%%%%%%%%%%%
%2D, single precision  examples 
%%%%%%%%%%%% %%%%%%%%%%%%%%%%%%%%%%%%%%%%%%%%%%%%%%%

\begin{figure}
    \centering
    \begin{subfigure}[b]{0.32\textwidth}
        \includegraphics[width=\textwidth]{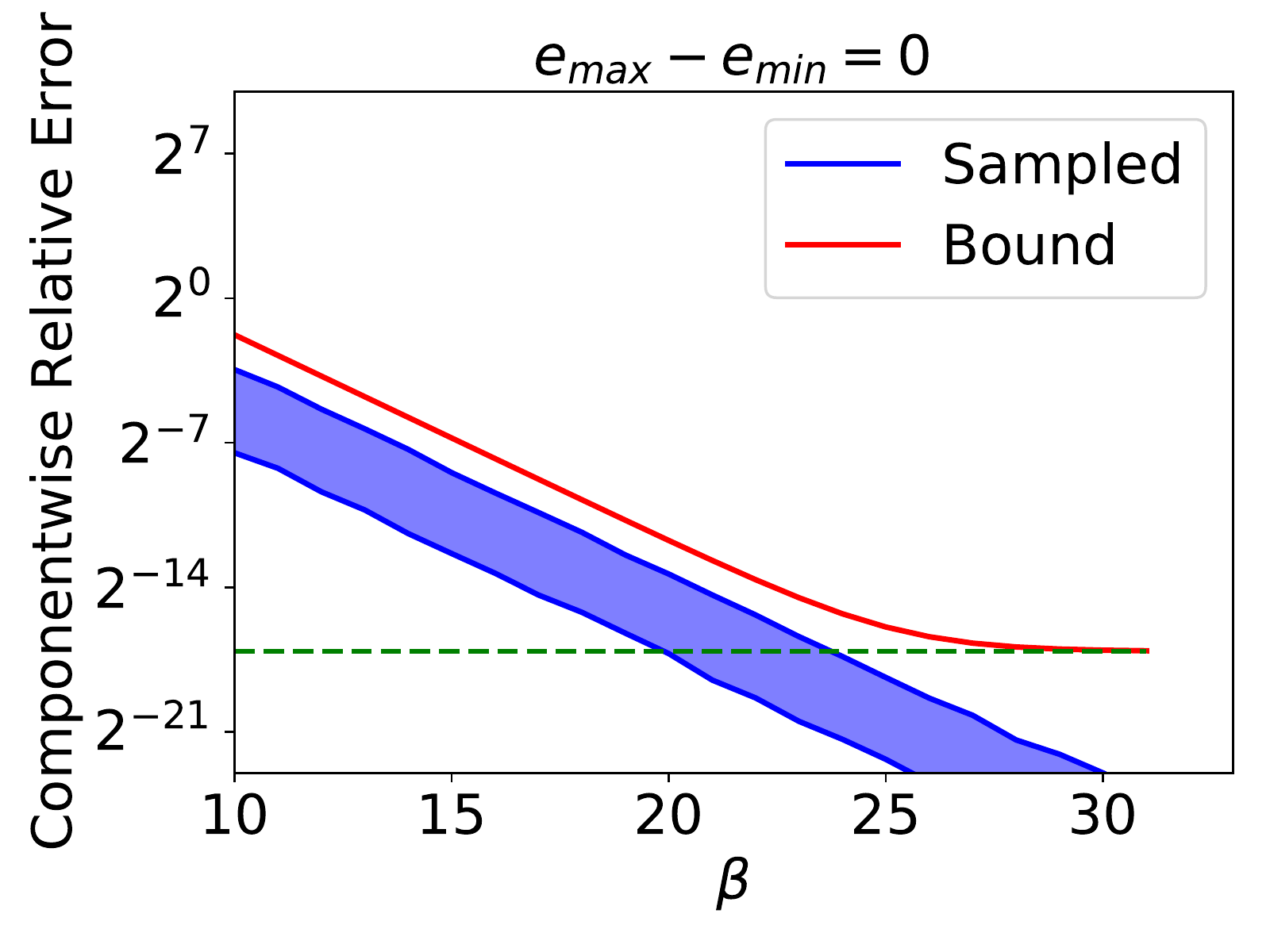}
        \label{fig:componenterror_precision_4by4parta}
               % \caption{$e_{max}- e_{min}= 0$}
    \end{subfigure}
      \begin{subfigure}[b]{0.32\textwidth}
        \includegraphics[width=\textwidth]{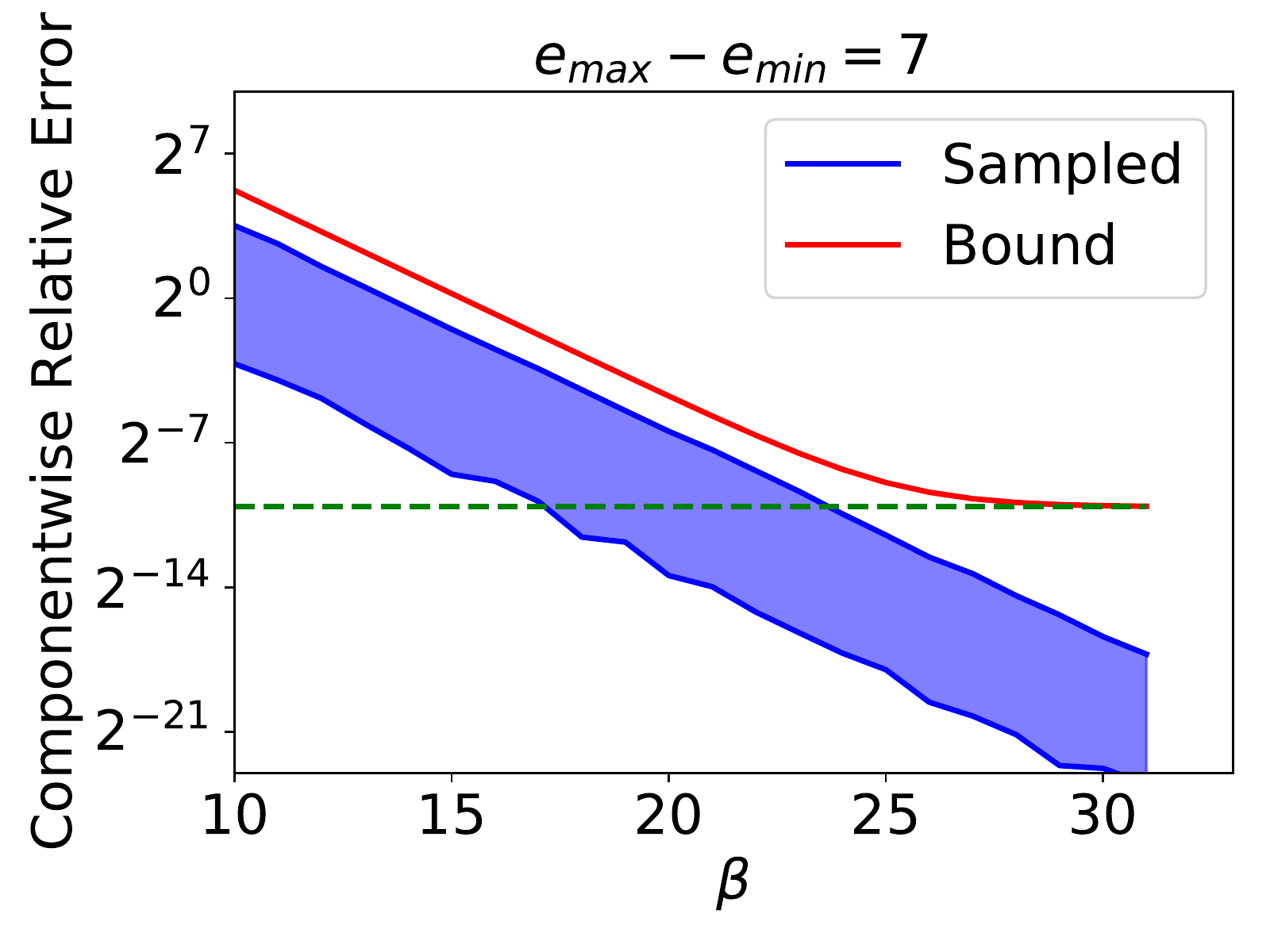}
%\caption{$e_{max}- e_{min}= 7$}
        \label{fig:componenterror_precision_4by4partb}
    \end{subfigure}
       \begin{subfigure}[b]{0.32\textwidth}
        \includegraphics[width=\textwidth]{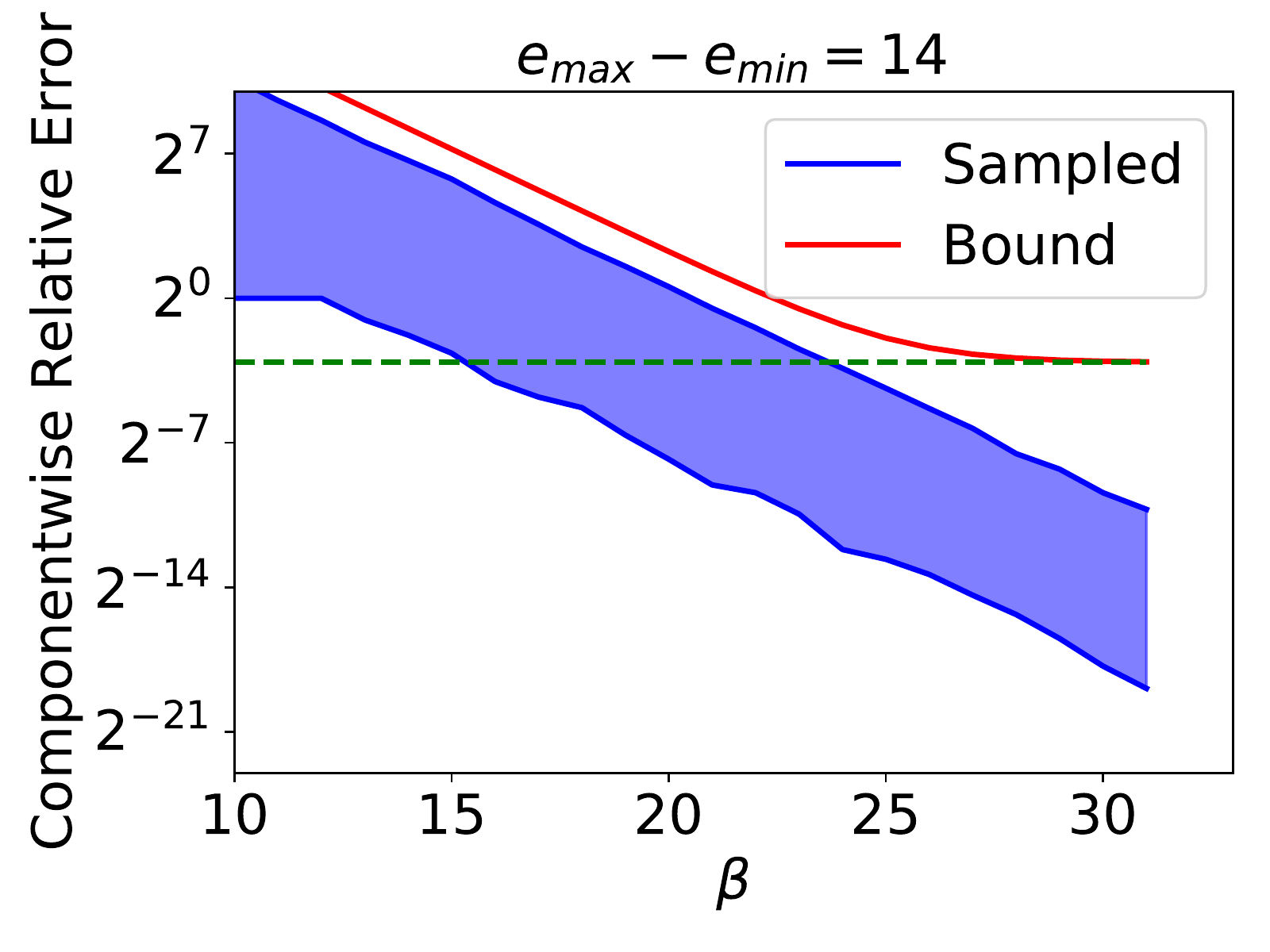}
%\caption{$e_{max}- e_{min}= 14$}
        \label{fig:componenterror_precision_4by4partc}

    \end{subfigure}
    \begin{subfigure}[b]{0.32\textwidth}
        \includegraphics[width=\textwidth]{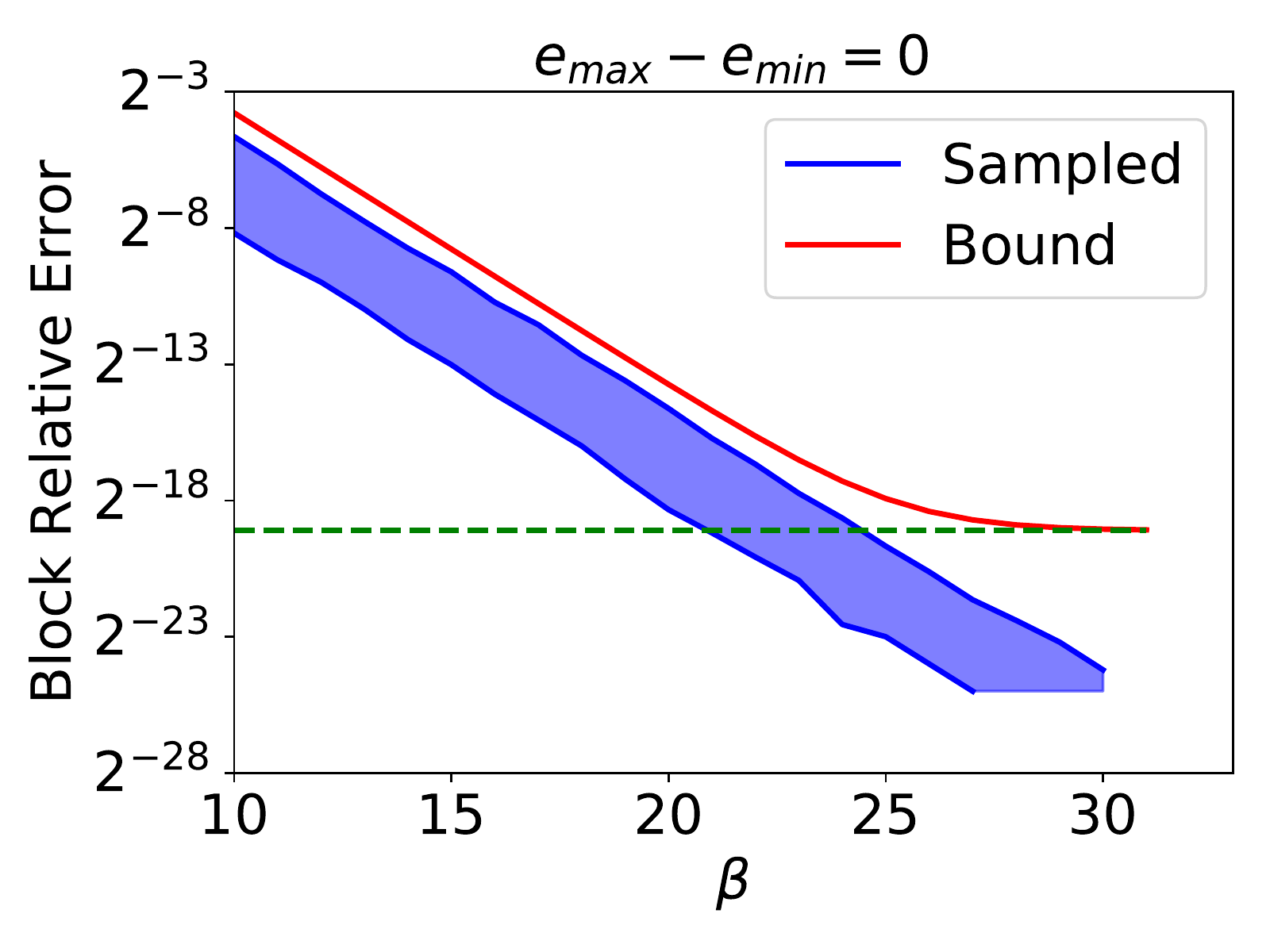}
  %      \caption{$e_{max}- e_{min}= 0$}
                \label{fig:normerror_precision_4by4parta}
    \end{subfigure}
        \begin{subfigure}[b]{0.32\textwidth}
        \includegraphics[width=\textwidth]{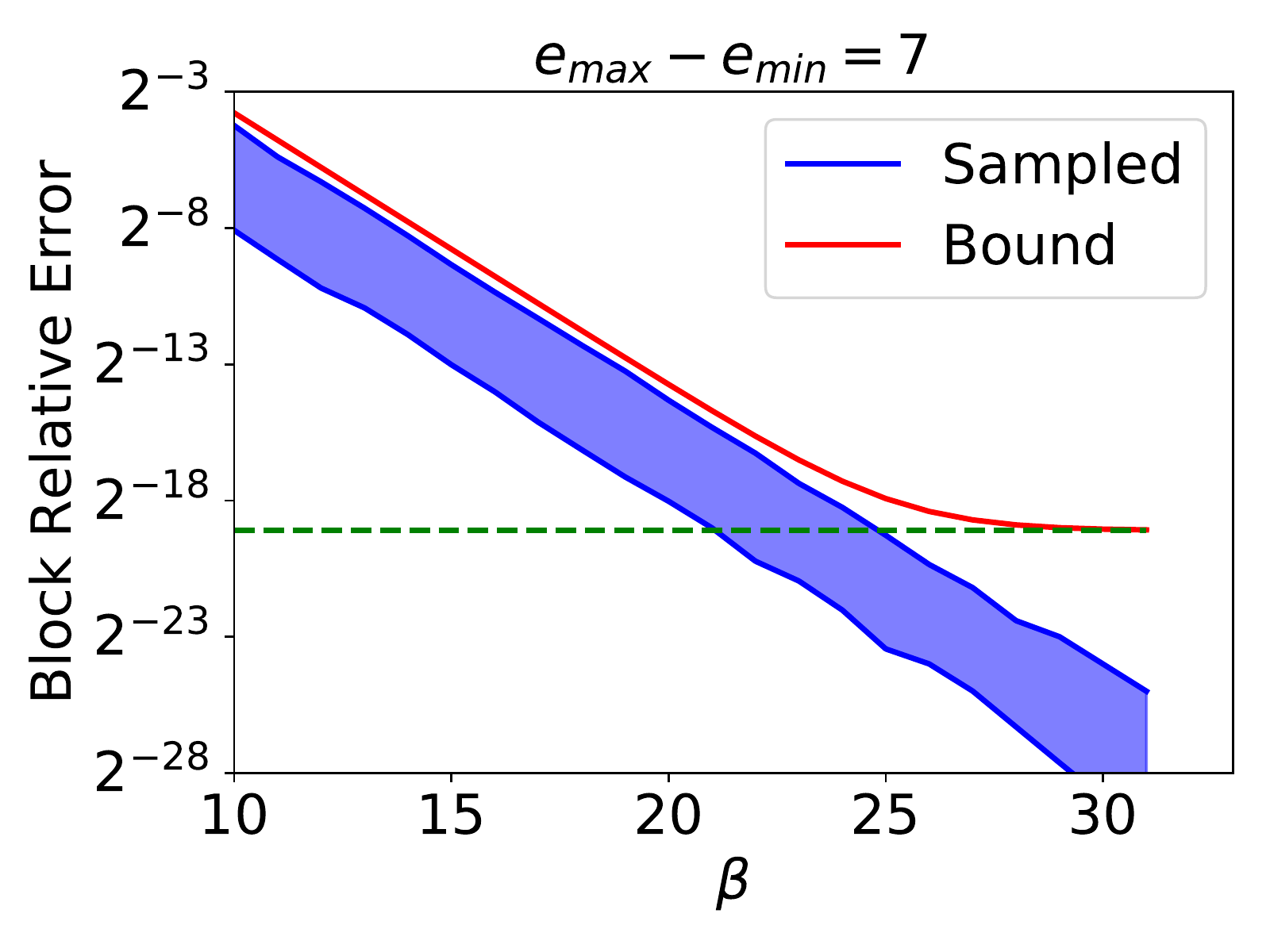}
%\caption{$e_{max}- e_{min}= 7$}
  \label{fig:normerror_precision_4by4partb}
    \end{subfigure}
    \begin{subfigure}[b]{0.32\textwidth}
        \includegraphics[width=\textwidth]{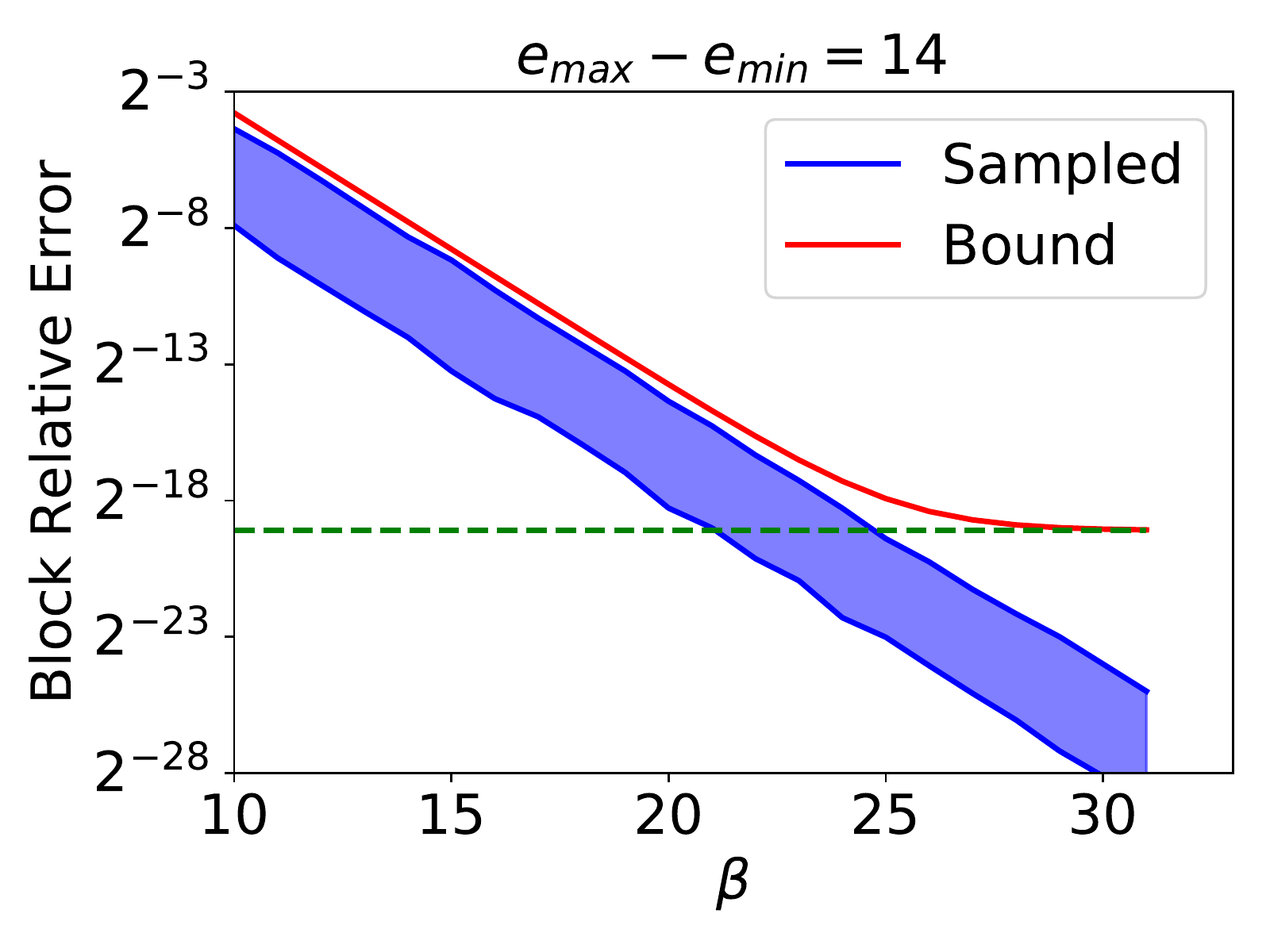}
%\caption{$e_{max}- e_{min}= 14$}
  \label{fig:normerror_precision_4by4partc}
    \end{subfigure}
    \caption{2-d Example with single precision:  componentwise relative error (top) and block relative error (bottom) with respect to {the precision parameter ($\beta$) for $ e_{max}-e_{min} \in \{ 0,7,14\} $.  The blue band represents the sampled maximum and minimum error, the red line depicts the theoretical bound, and the dashed green line represents the asymptotic behavior of the theoretical bound. } }
                \label{fig:componenterror_precision_4by4}

\end{figure}

\begin{figure}
    \centering
        \begin{subfigure}[b]{0.32\textwidth}
        \includegraphics[width=\textwidth]{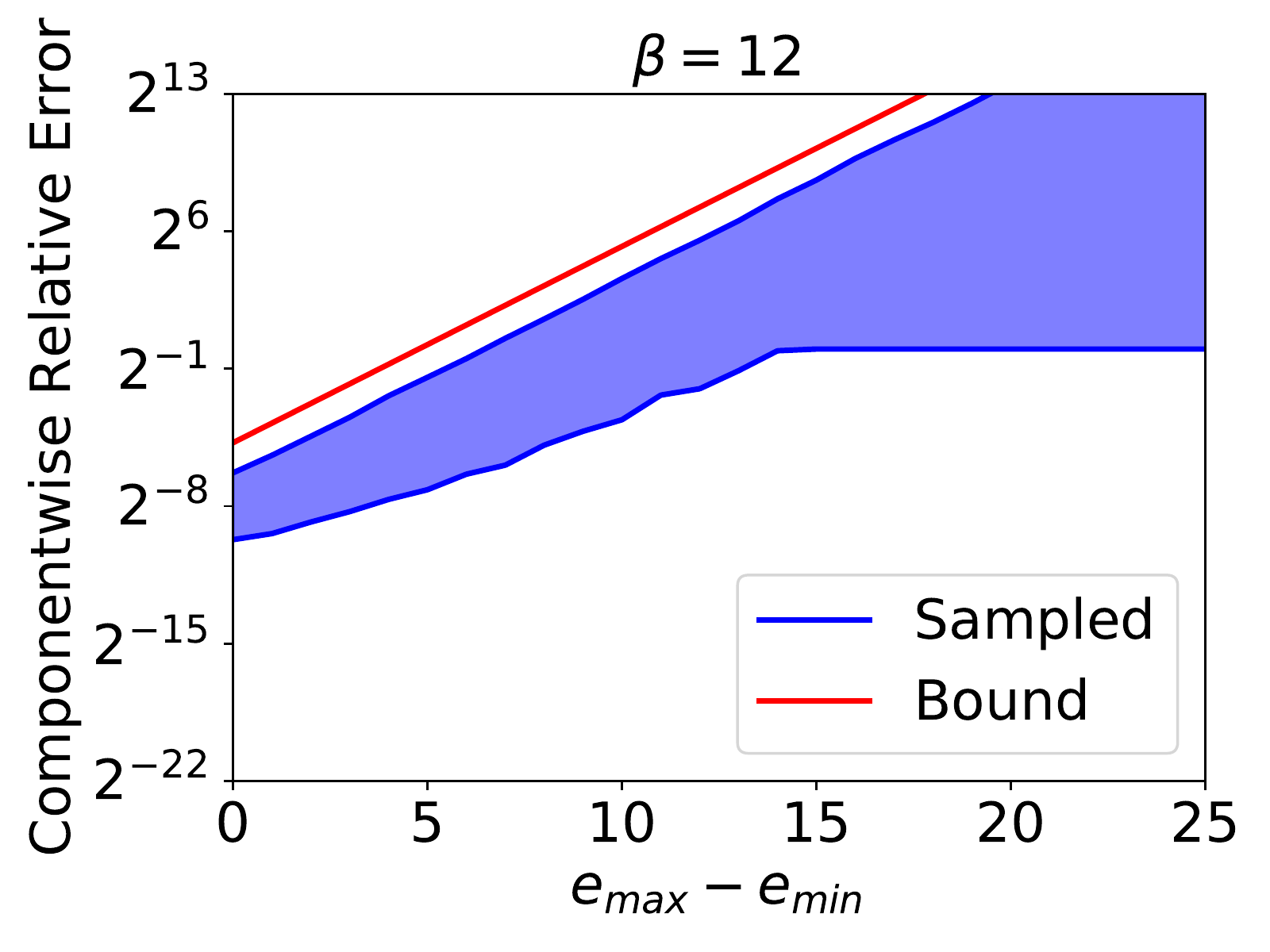}
%\caption{$\beta = 22$}
    \end{subfigure}
  \begin{subfigure}[b]{0.32\textwidth}
        \includegraphics[width=\textwidth]{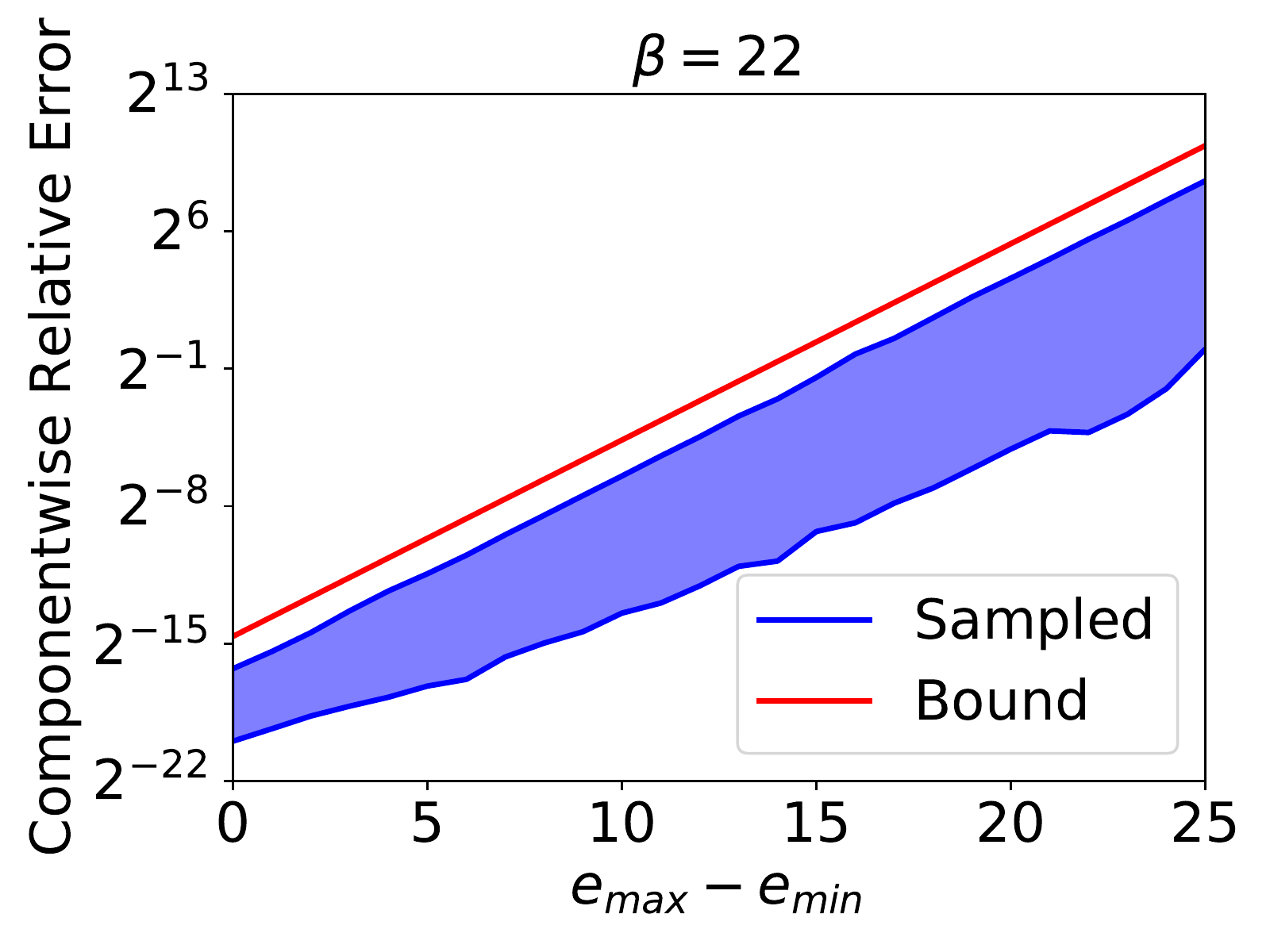}
%\caption{$\beta = 27$}
    \end{subfigure}
    \begin{subfigure}[b]{0.32\textwidth}
        \includegraphics[width=\textwidth]{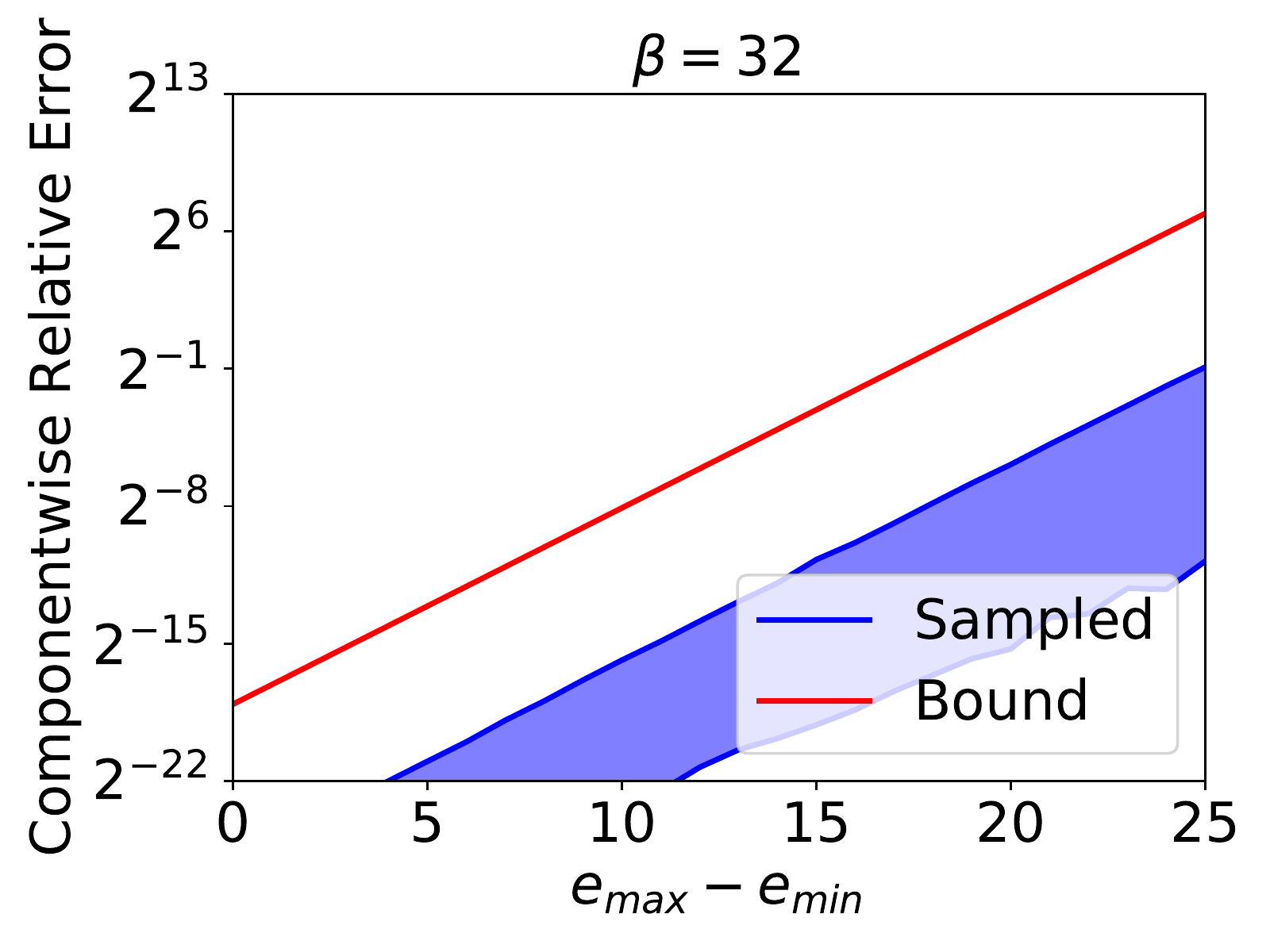}
   %     \caption{$\beta = 32$}
    \end{subfigure}
          
    \begin{subfigure}[b]{0.32\textwidth}
        \includegraphics[width=\textwidth]{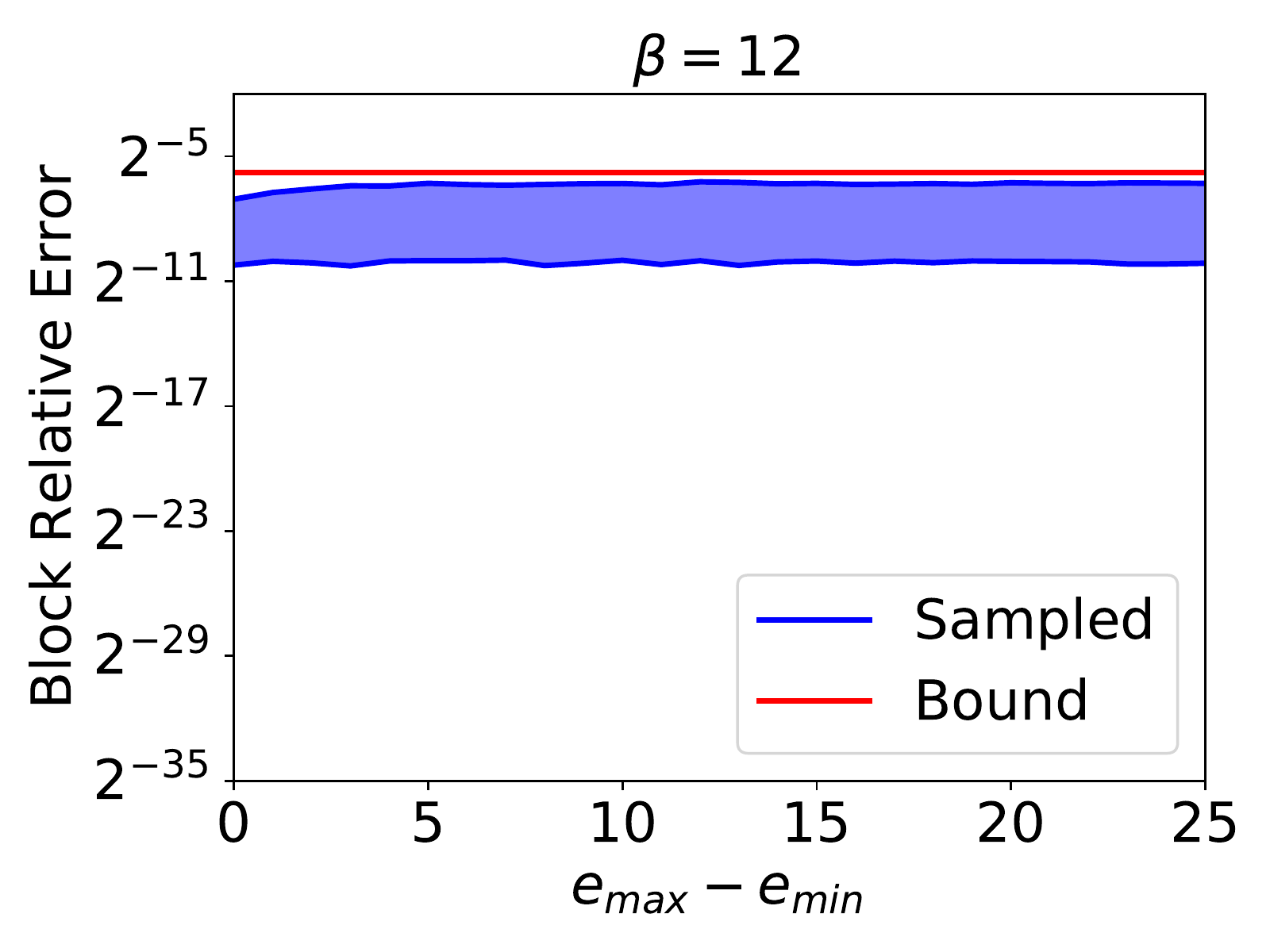}
%\caption{$\beta = 22$}
    \end{subfigure}
    \begin{subfigure}[b]{0.32\textwidth}
        \includegraphics[width=\textwidth]{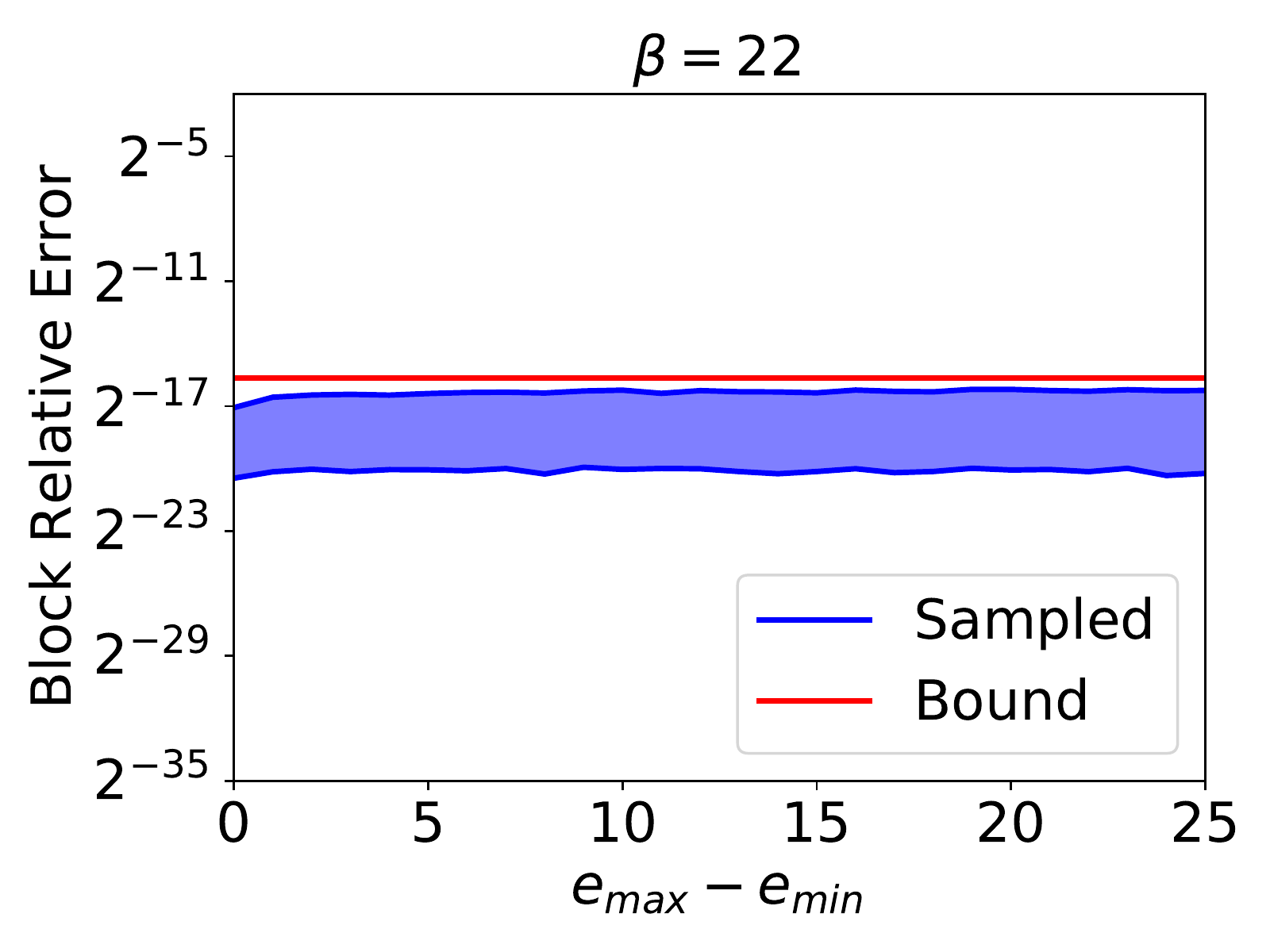}
%\caption{$\beta = 27$}
\end{subfigure}
        \begin{subfigure}[b]{0.32\textwidth}
        \includegraphics[width=\textwidth]{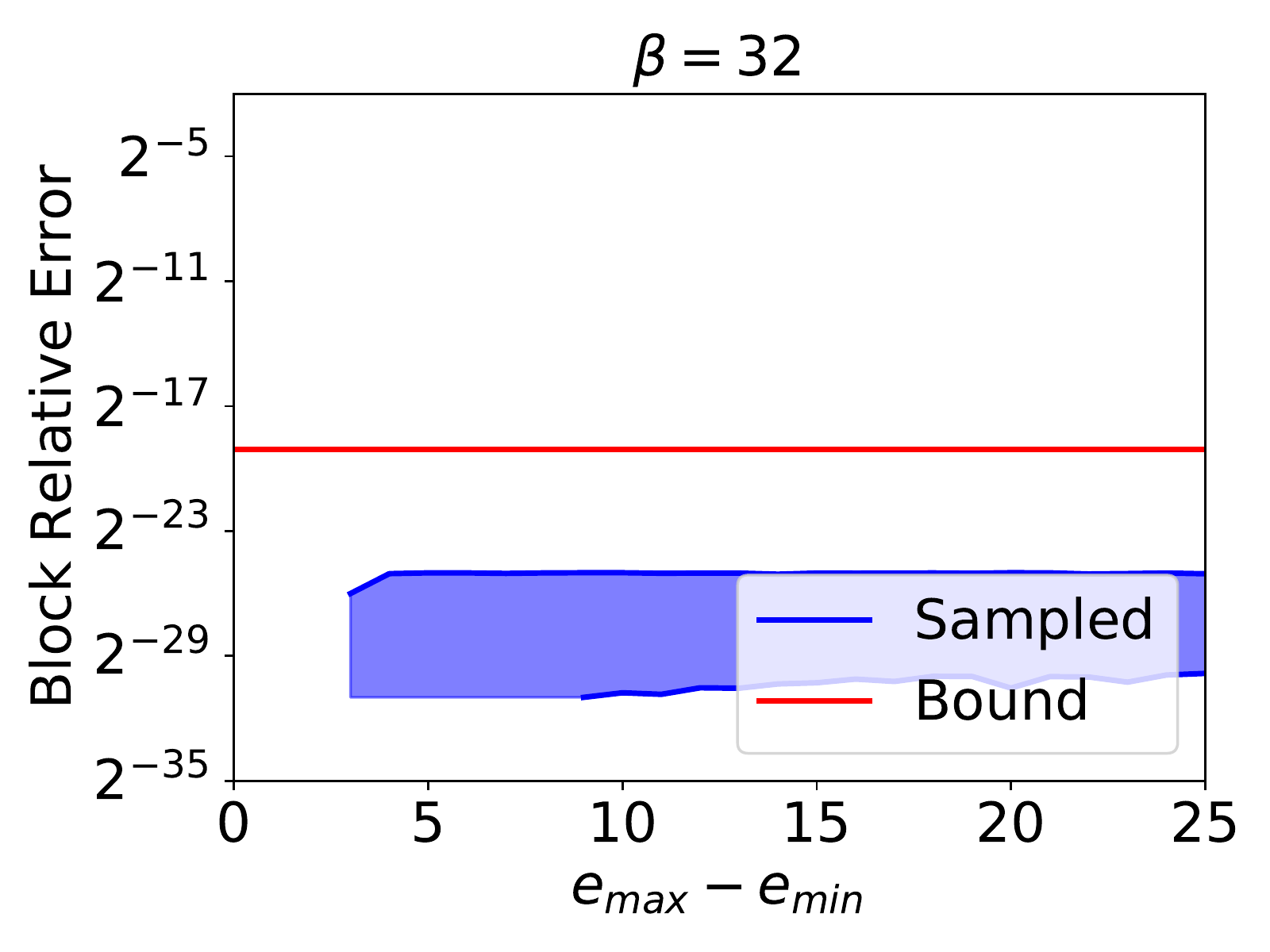}
%        \caption{$\beta= 32$}
    \end{subfigure}
      \caption{2-d Example with single precision:  componentwise relative error (top) and block relative error (bottom) with respect to the difference in exponents ($e_{max}-e_{min}$) for $\beta \in \{12,22,32\}${. The blue band represents the sampled maximum and minimum error and the red line depicts the theoretical bound. }}
    \label{fig:normerror_exp_4by4}
\end{figure}
%%%%%%%%%%%% %%%%%%%%%%%%%%%%%%%%%%%%%%%%%%%%%%%%%%%
%2D, double precision  examples 
%%%%%%%%%%%% %%%%%%%%%%%%%%%%%%%%%%%%%%%%%%%%%%%%%%%

\begin{figure}
    \centering
    \begin{subfigure}[b]{0.32\textwidth}
        \includegraphics[width=\textwidth]{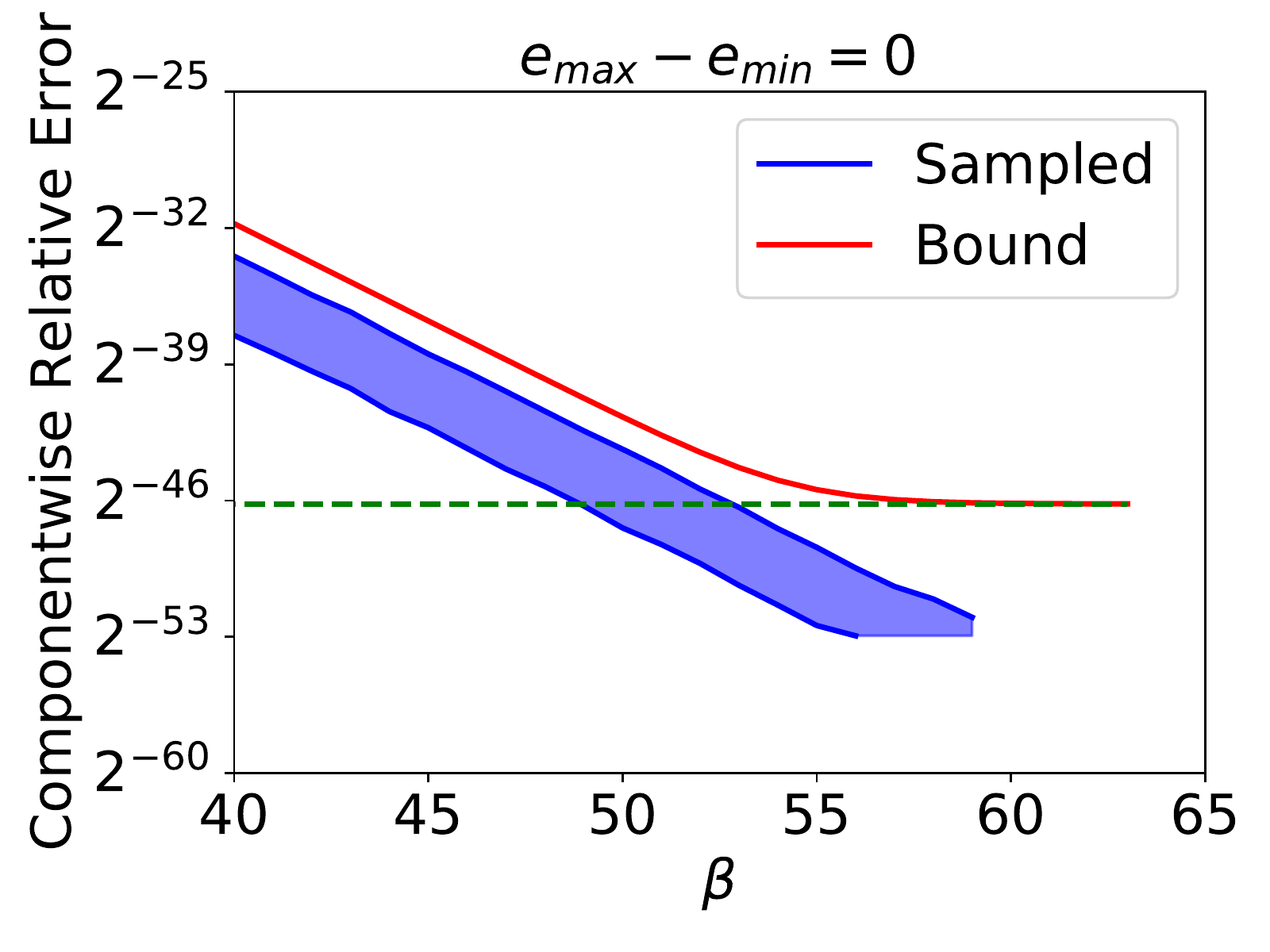}
      %  \caption{$e_{max}- e_{min}= 0$}
    \end{subfigure}
      \begin{subfigure}[b]{0.32\textwidth}
        \includegraphics[width=\textwidth]{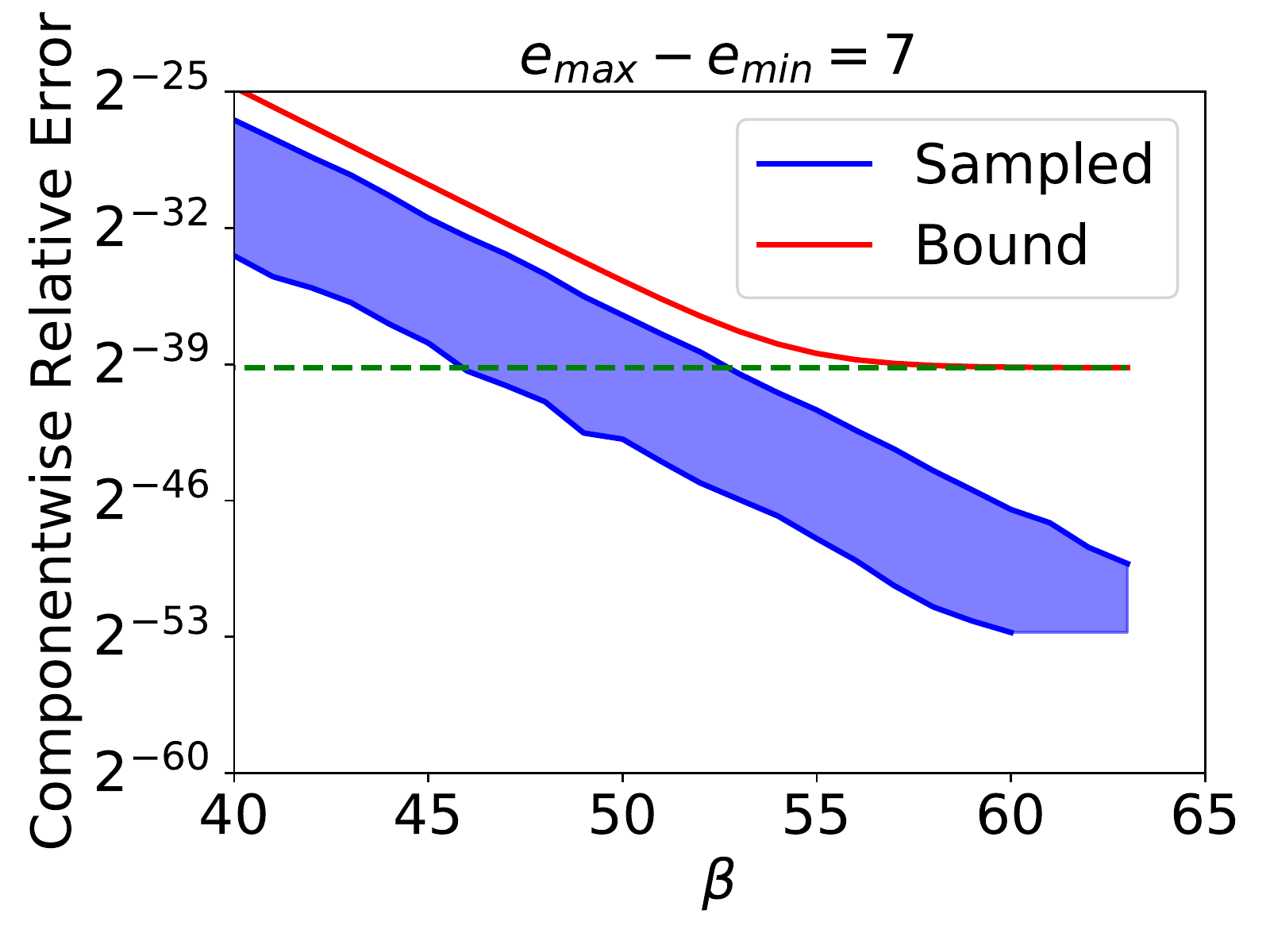}
%\caption{$e_{max}- e_{min}= 7$}
    \end{subfigure}
       \begin{subfigure}[b]{0.32\textwidth}
        \includegraphics[width=\textwidth]{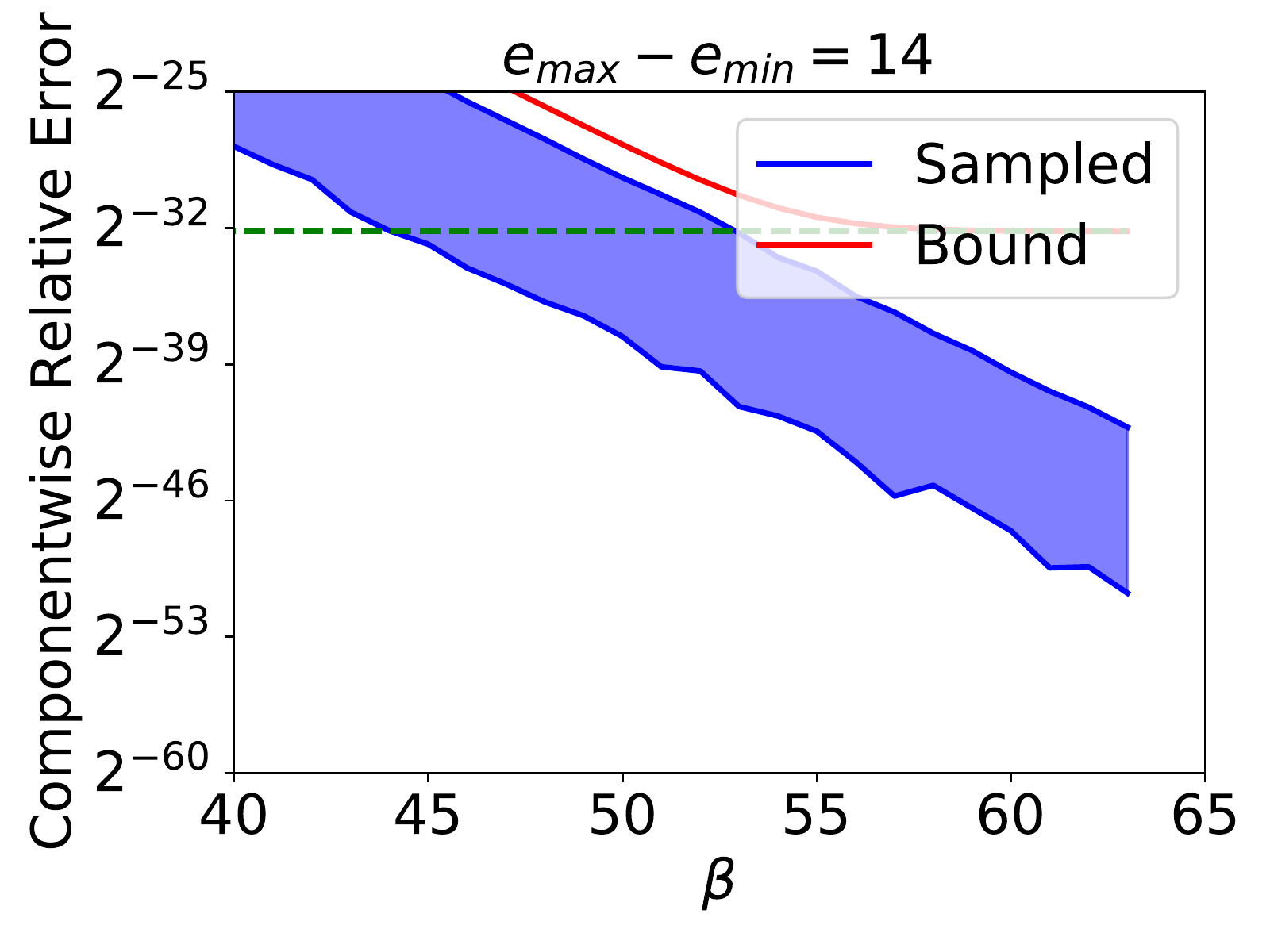}
%\caption{$e_{max}- e_{min}= 14$}

    \end{subfigure}
     \begin{subfigure}[b]{0.32\textwidth}
        \includegraphics[width=\textwidth]{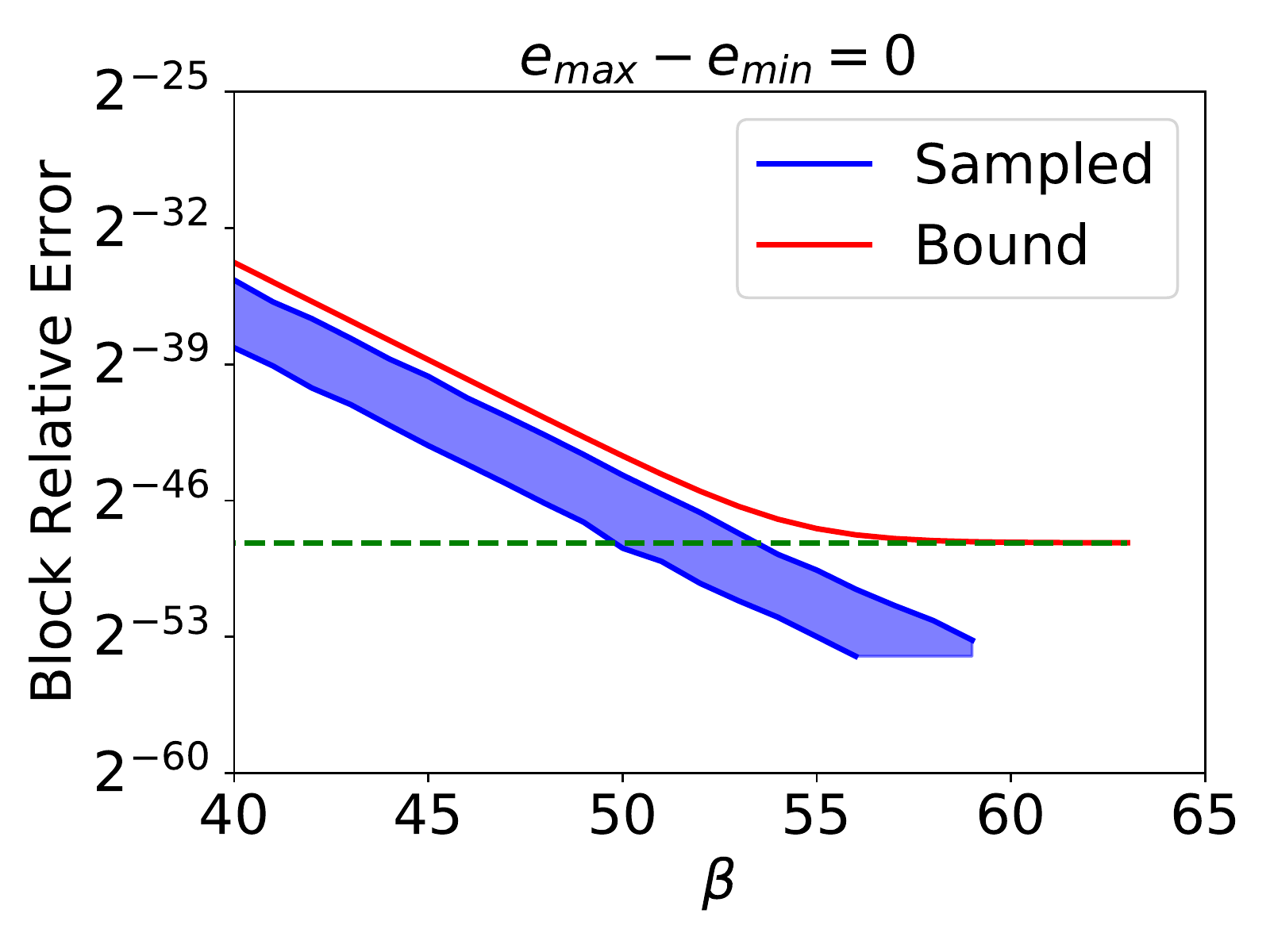}
   %     \caption{$e_{max}- e_{min}= 0$}
    \end{subfigure}
        \begin{subfigure}[b]{0.32\textwidth}
        \includegraphics[width=\textwidth]{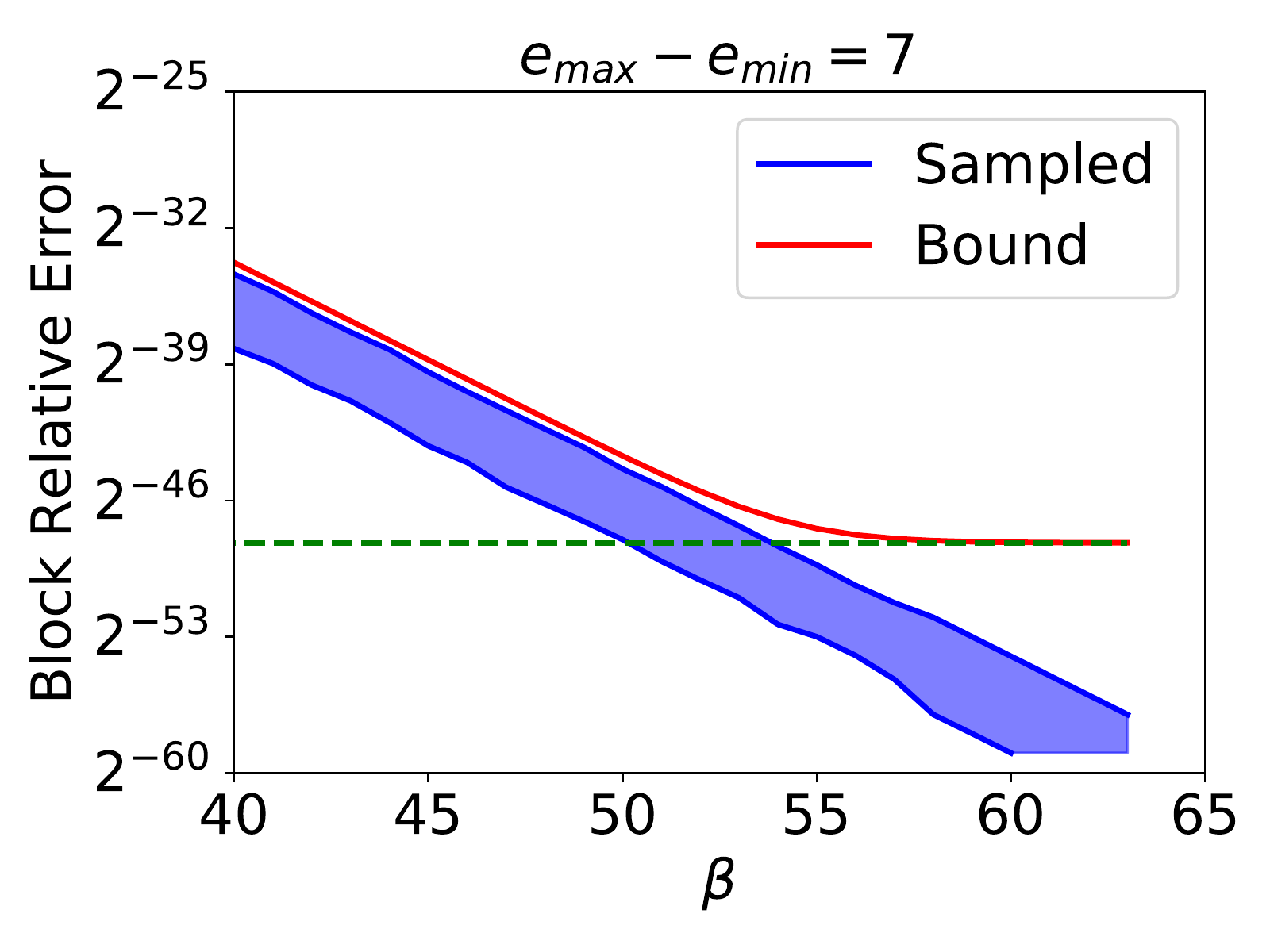}
%\caption{$e_{max}- e_{min}= 7$}
    \end{subfigure}
    \begin{subfigure}[b]{0.32\textwidth}
        \includegraphics[width=\textwidth]{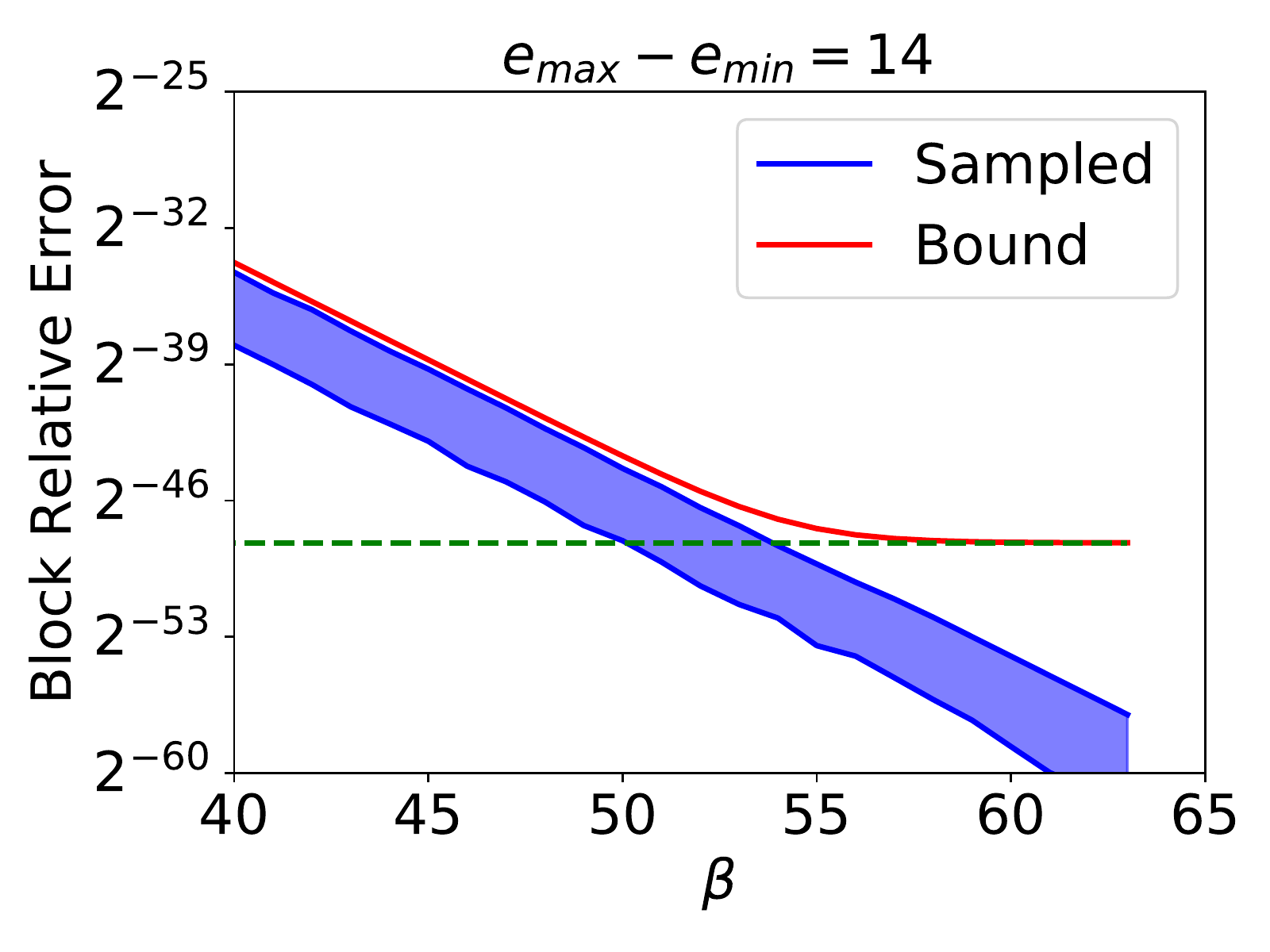}
%\caption{$e_{max}- e_{min}= 14$}
    \end{subfigure}
  \caption{2-d Example with double precision:  componentwise relative error (top) and block relative error (bottom) with respect to the precision parameter ($\beta$) for $ e_{max}-e_{min} \in \{0,7,14\}${. The blue band represents the sampled maximum and minimum error, the red line depicts the theoretical bound, and the dashed green line represents the asymptotic behavior of the theoretical bound. }}
    \label{fig:2d_double_1}
\end{figure}

\begin{figure}
    \centering
    \begin{subfigure}[b]{0.32\textwidth}
        \includegraphics[width=\textwidth]{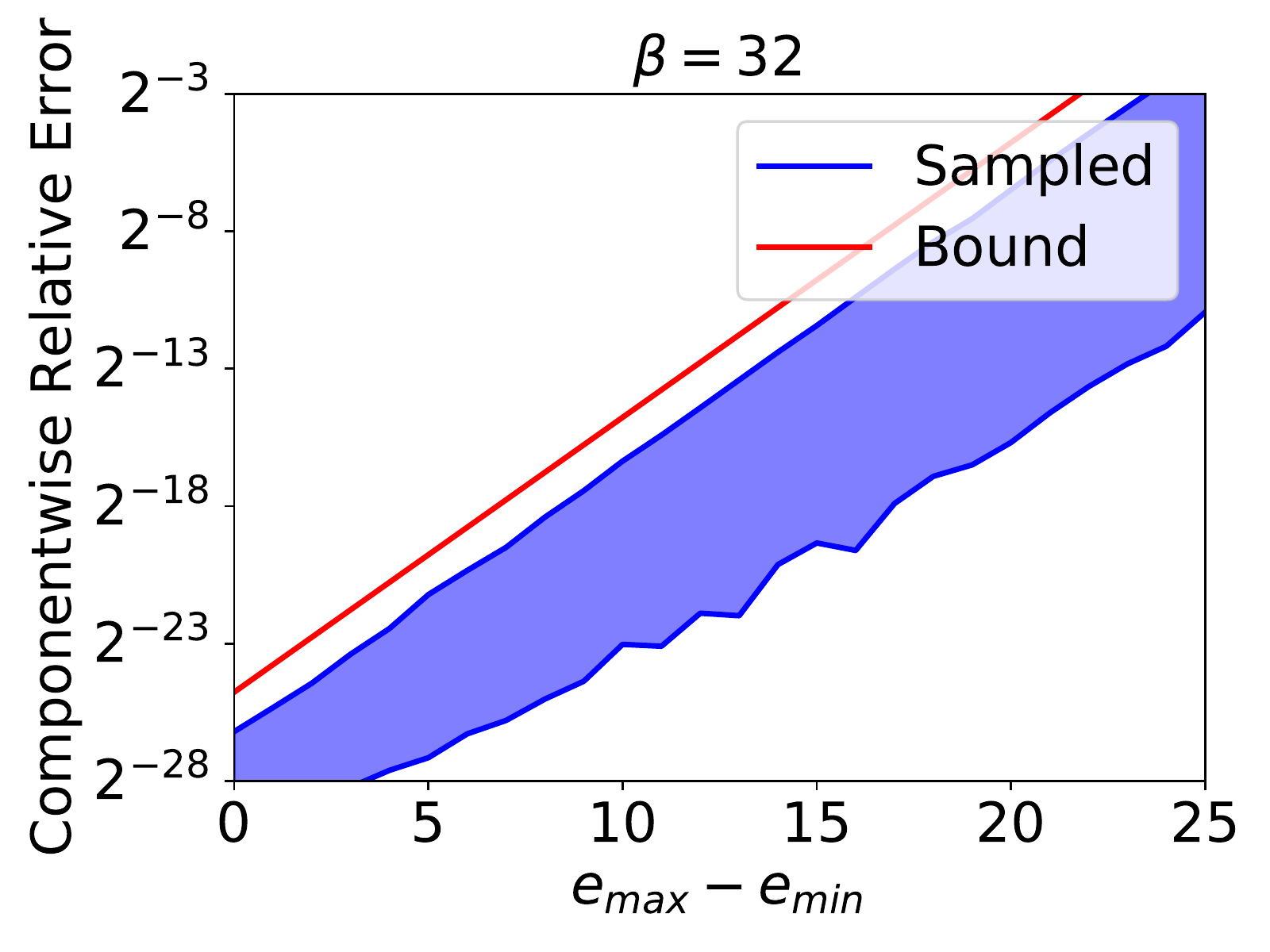}
    %    \caption{$\beta = 64$}
    \end{subfigure}
        \begin{subfigure}[b]{0.32\textwidth}
        \includegraphics[width=\textwidth]{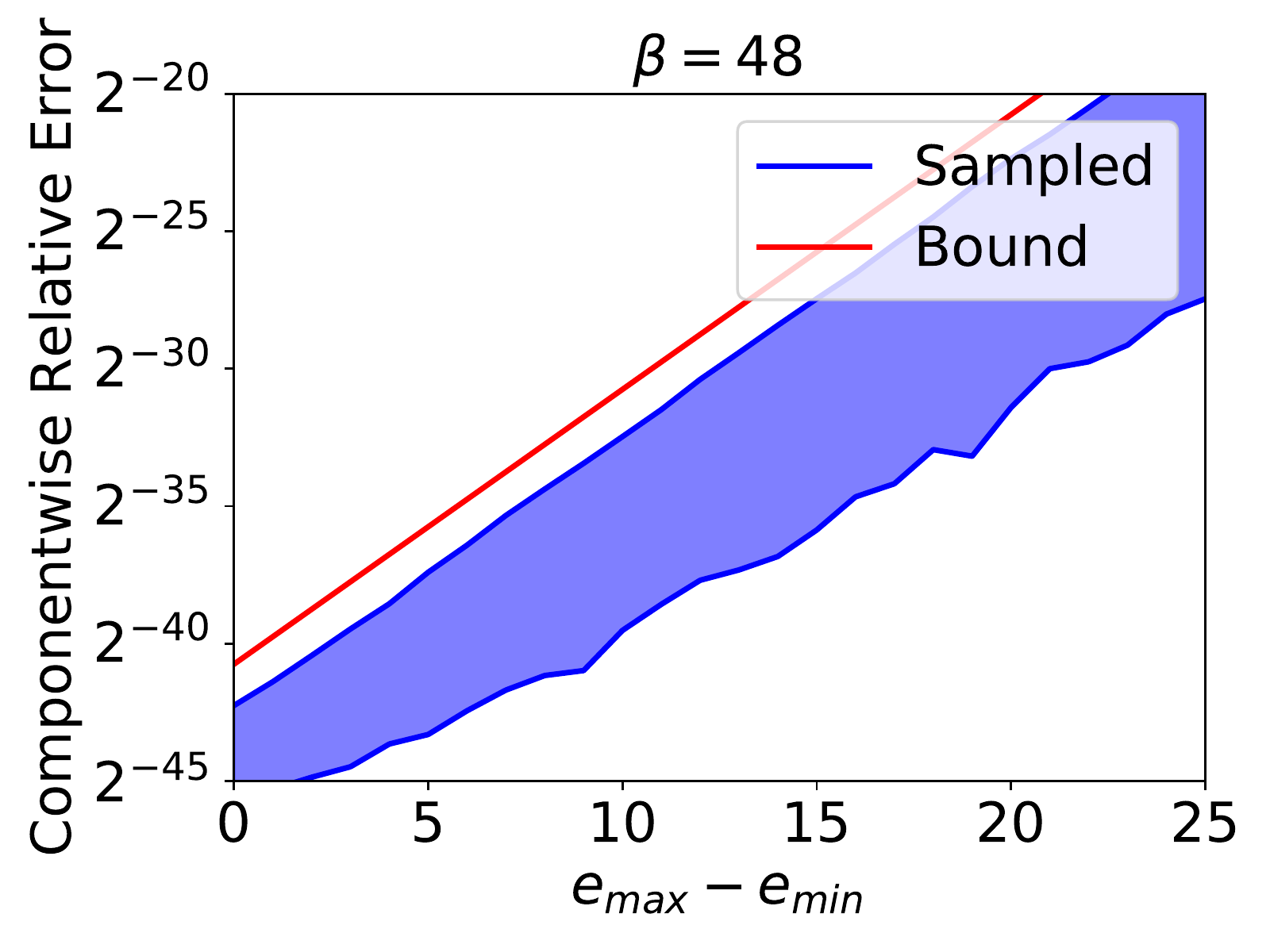}
%\caption{$\beta = 60$}
    \end{subfigure}
        \begin{subfigure}[b]{0.32\textwidth}
        \includegraphics[width=\textwidth]{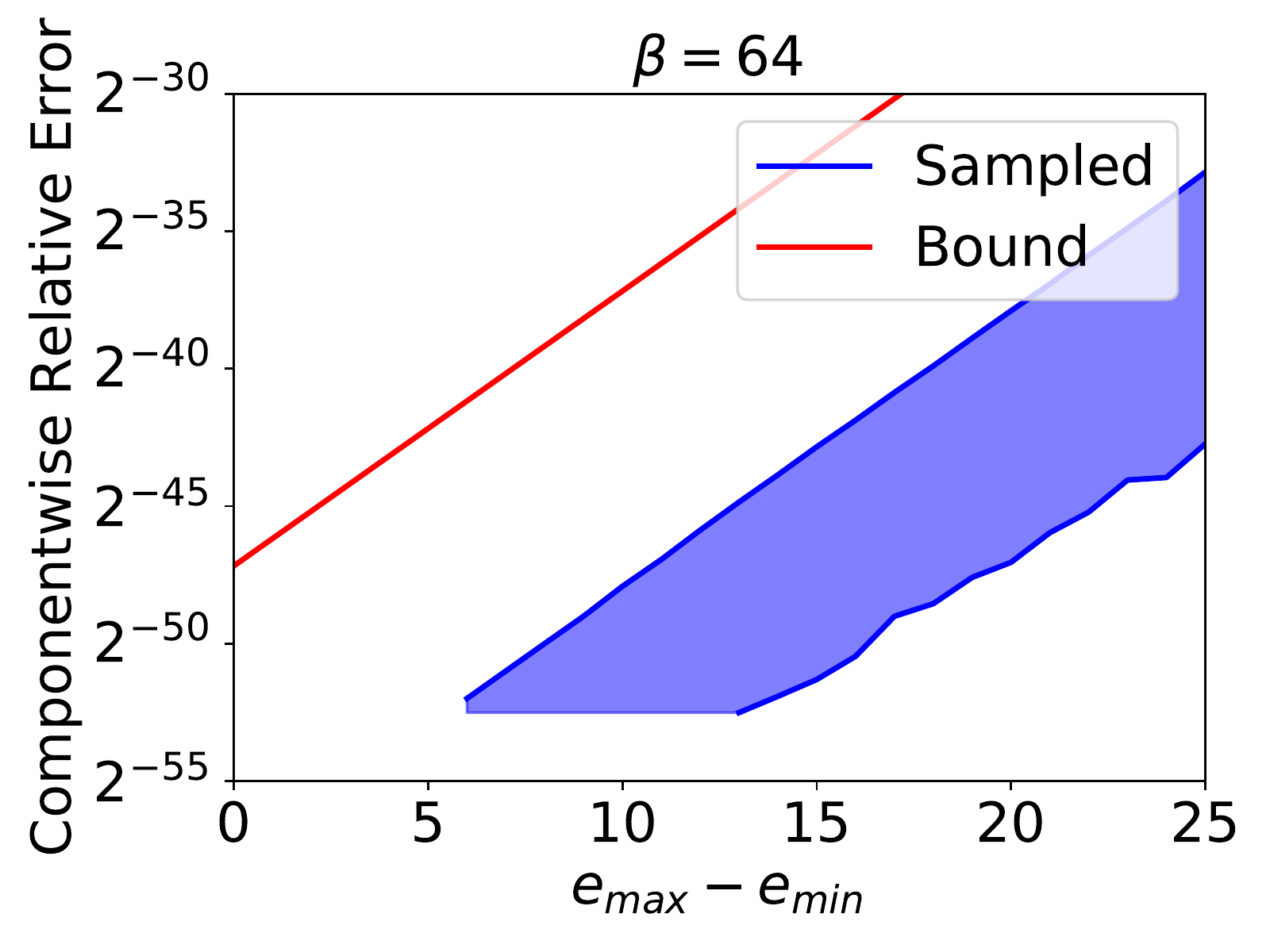}
%\caption{$\beta = 56$}
\end{subfigure}
    \begin{subfigure}[b]{0.32\textwidth}
        \includegraphics[width=\textwidth]{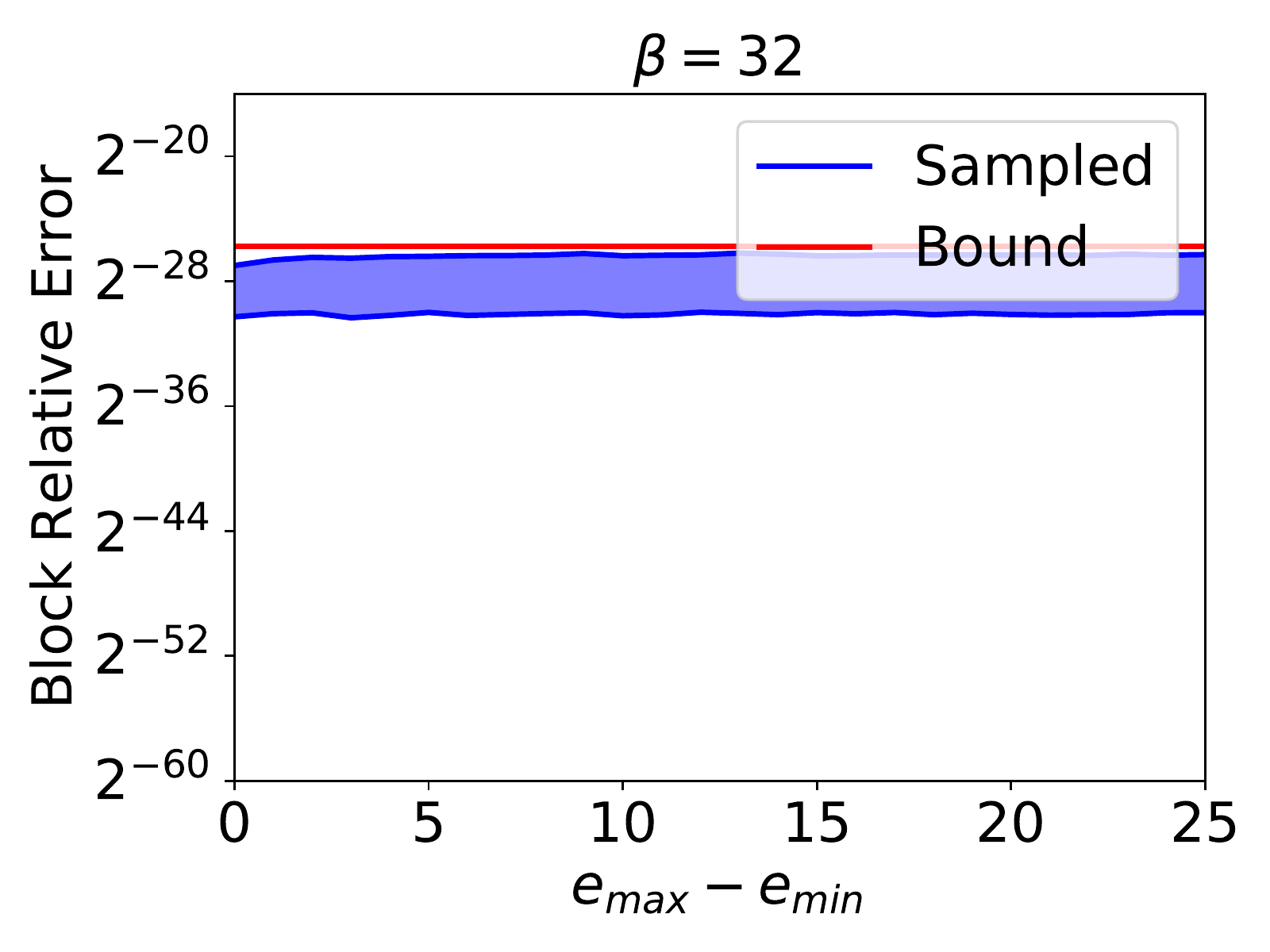}
 %       \caption{$\beta= 64$}
    \end{subfigure}
    \begin{subfigure}[b]{0.32\textwidth}
        \includegraphics[width=\textwidth]{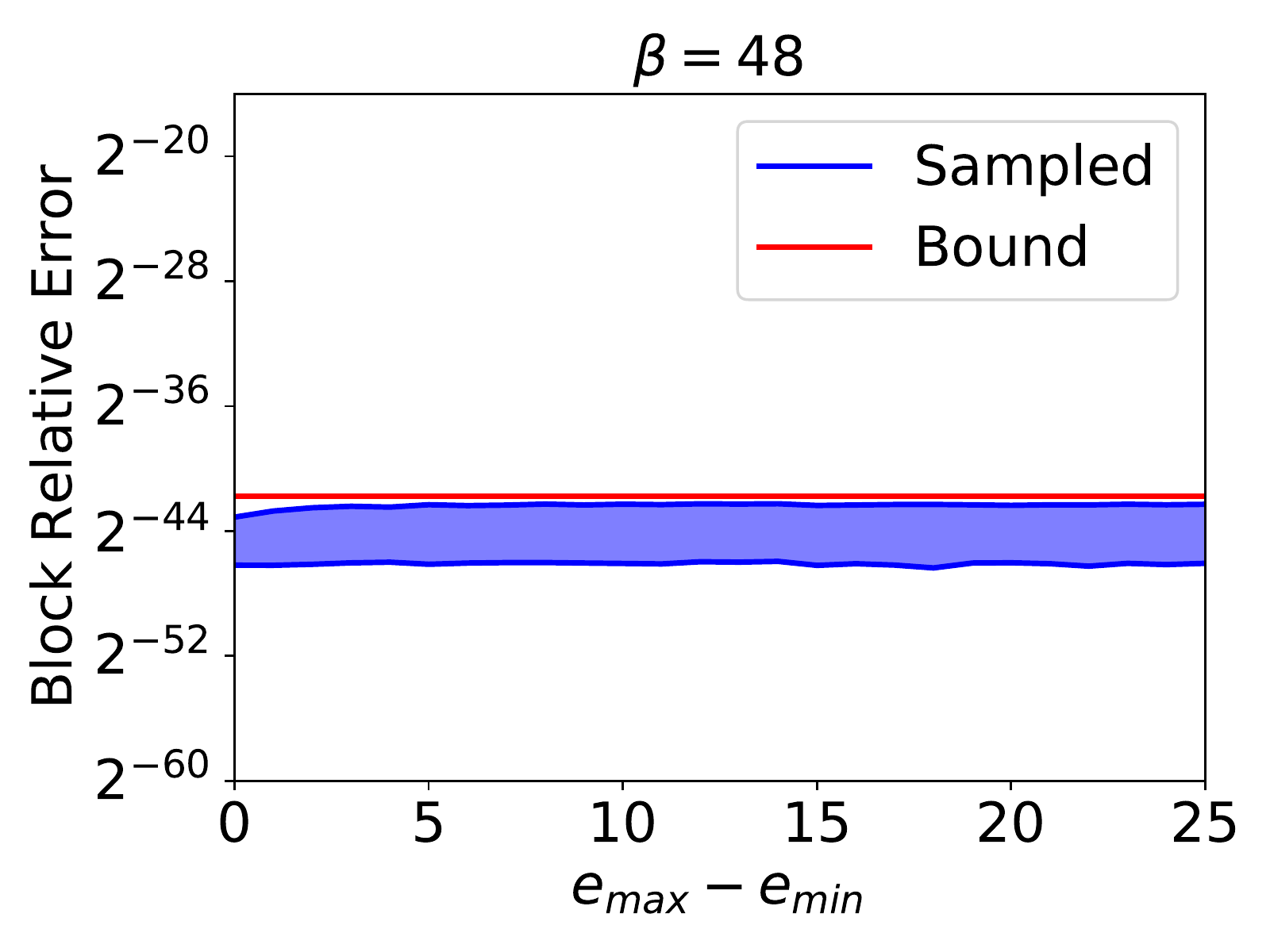}
%\caption{$\beta = 60$}
    \end{subfigure}
    \begin{subfigure}[b]{0.32\textwidth}
        \includegraphics[width=\textwidth]{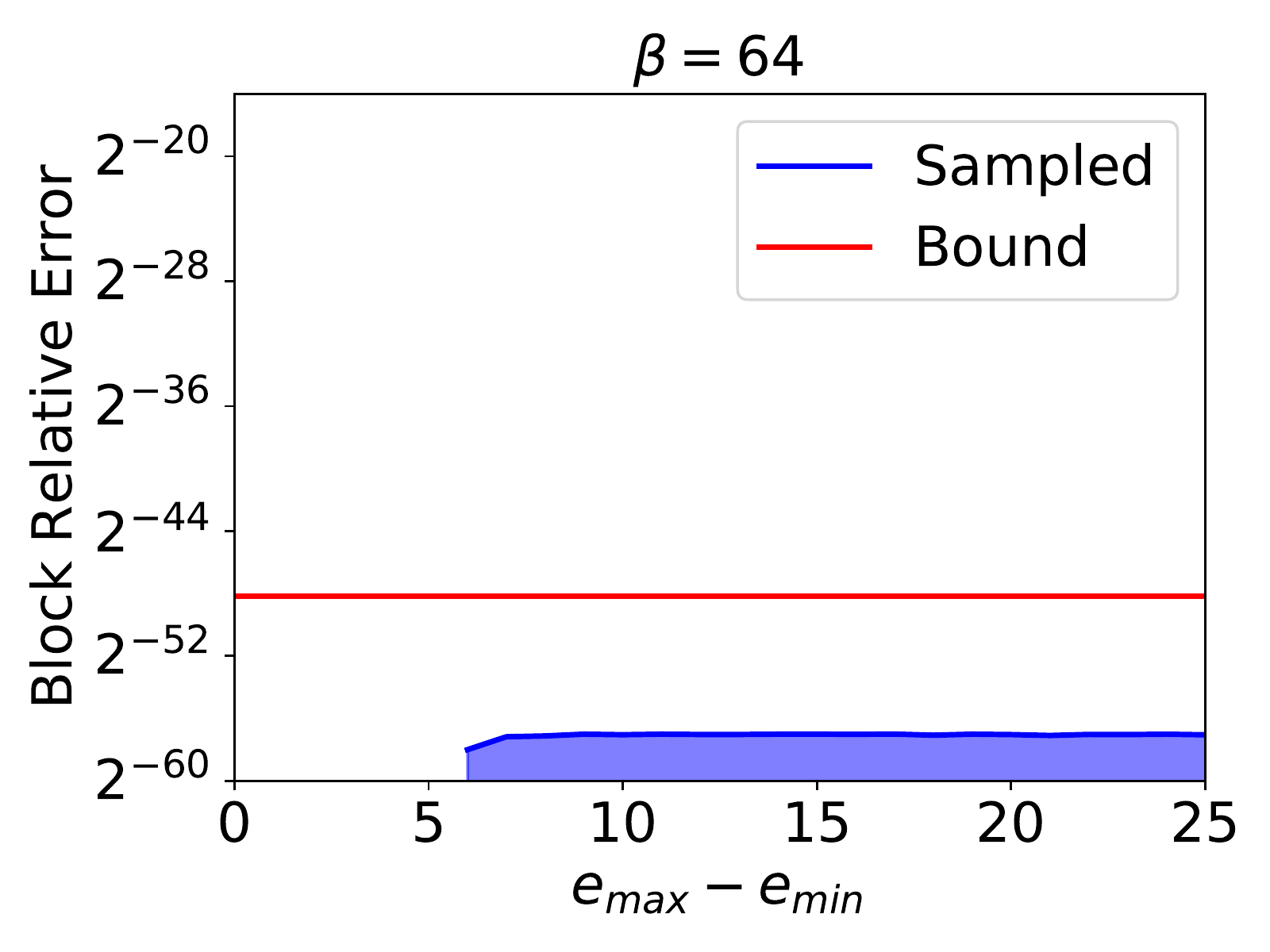}
%\caption{$\beta = 56$}
    \end{subfigure}
 \caption{2-d Example with double precision:  componentwise relative error (top) and block relative error (bottom) with respect to the difference in exponents ($e_{max}-e_{min}$) for $\beta \in \{32,48,64\}${. The blue band represents the sampled maximum and minimum error and the red line depicts the theoretical bound.}}
     \label{fig:2d_double_4}
\end{figure}

%%%%%%%%%%%% %%%%%%%%%%%%%%%%%%%%%%%%%%%%%%%%%%%%%%%
%1D, double precision  examples 
%%%%%%%%%%%% %%%%%%%%%%%%%%%%%%%%%%%%%%%%%%%%%%%%%%%

\begin{figure}
    \centering
    \begin{subfigure}[b]{0.32\textwidth}
        \includegraphics[width=\textwidth]{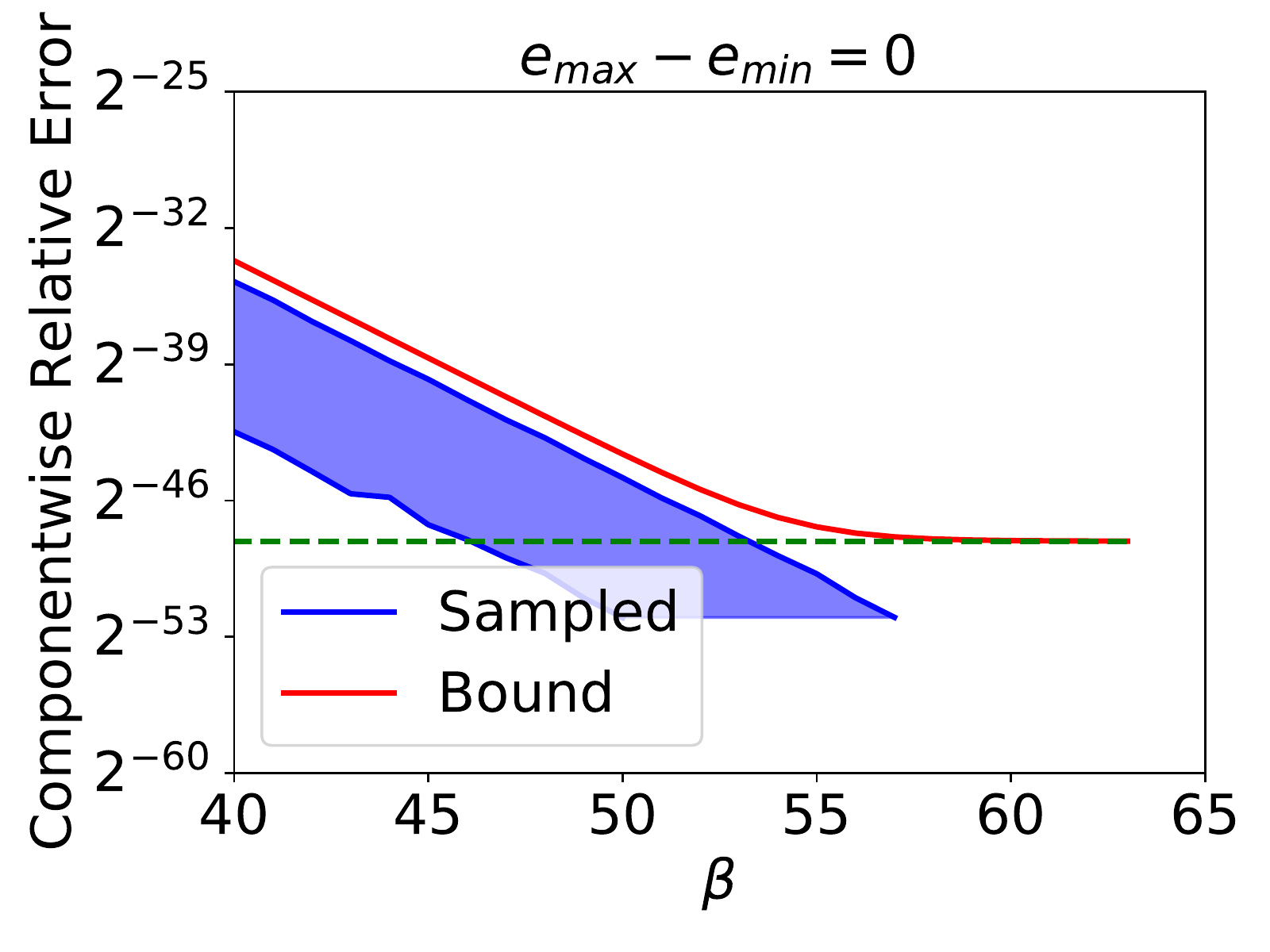}
     %   \caption{$e_{max}- e_{min}= 0$}
    \end{subfigure}
      \begin{subfigure}[b]{0.32\textwidth}
        \includegraphics[width=\textwidth]{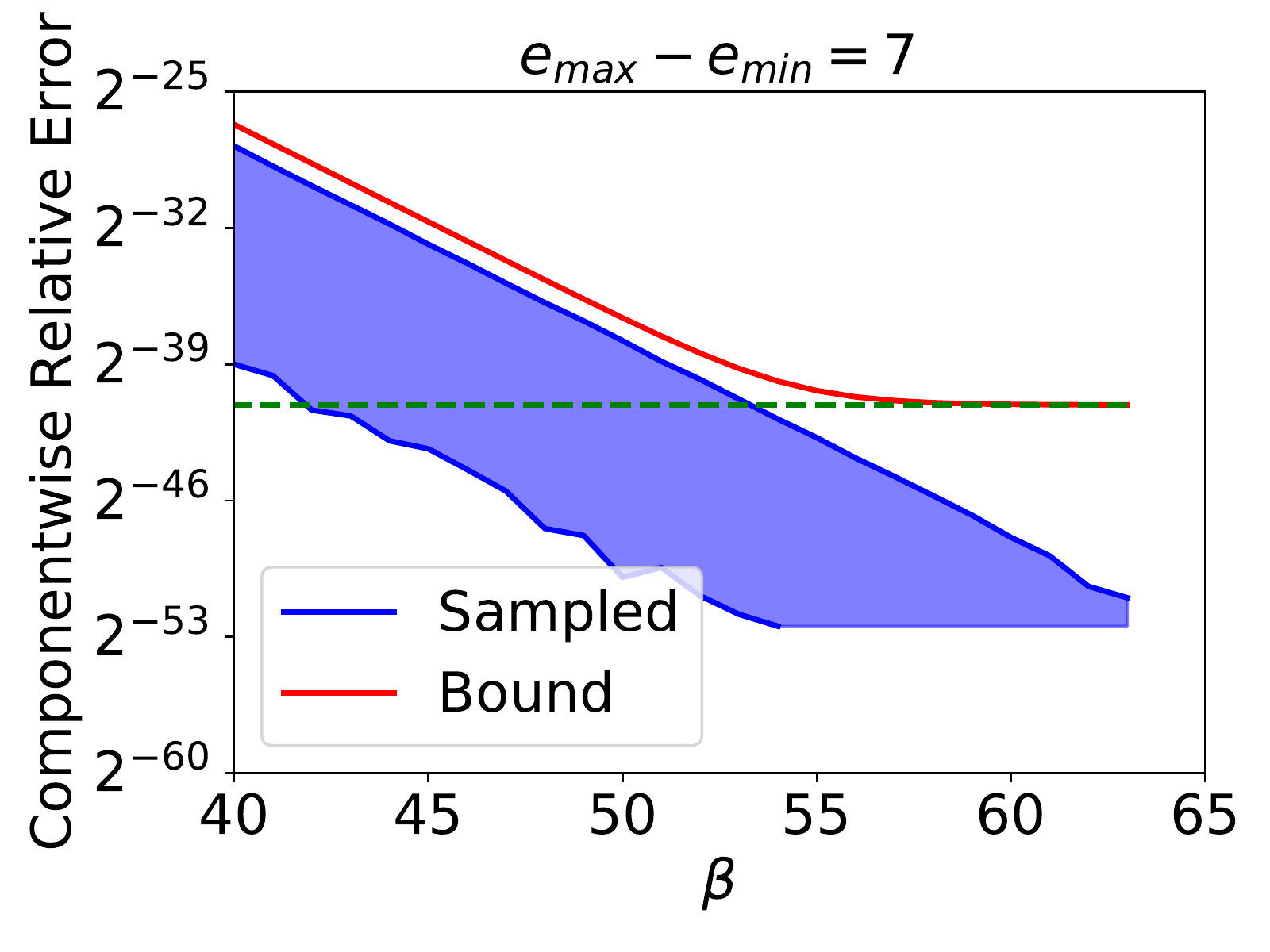}
%\caption{$e_{max}- e_{min}= 7$}
    \end{subfigure}
       \begin{subfigure}[b]{0.32\textwidth}
        \includegraphics[width=\textwidth]{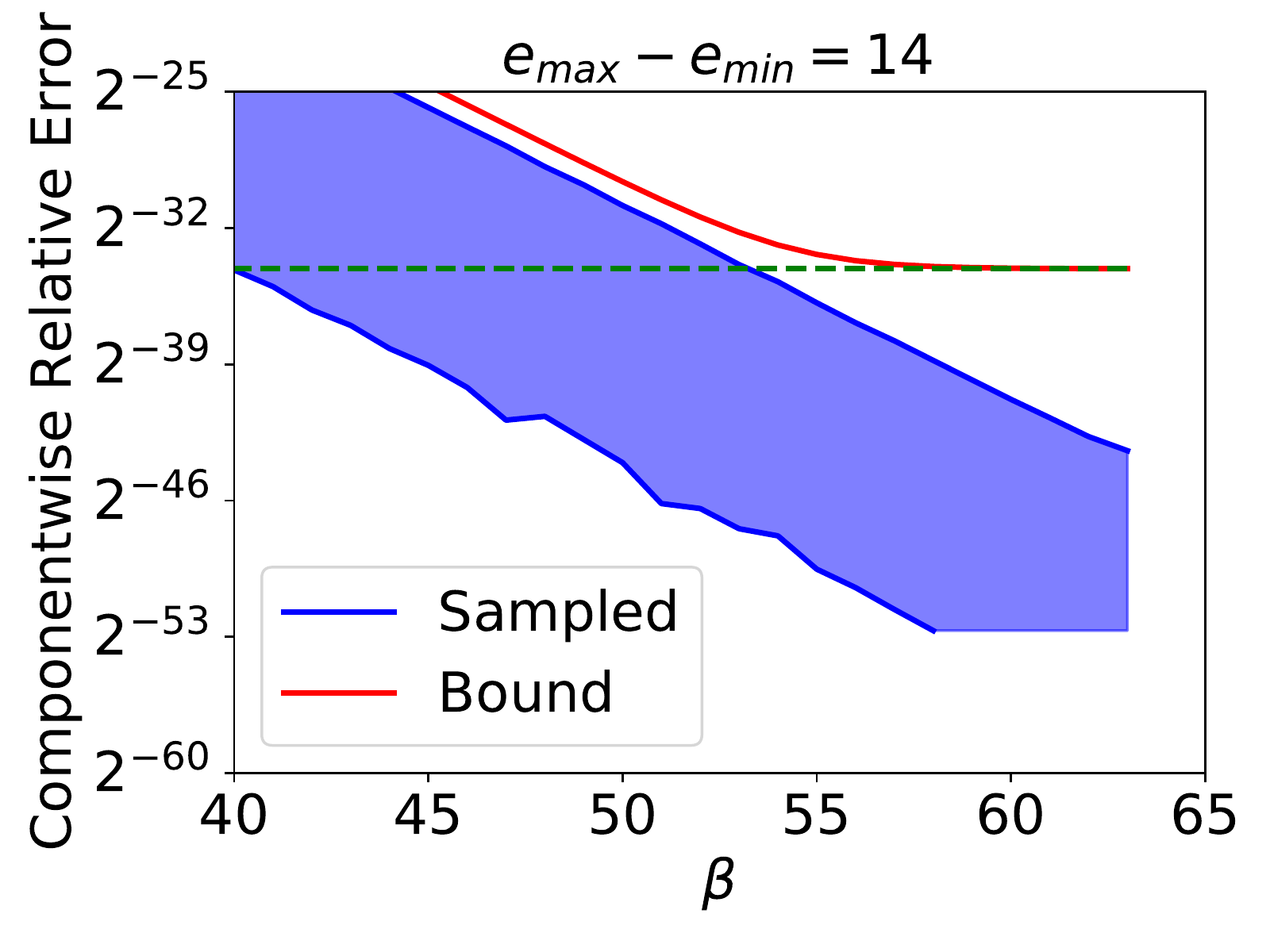}
%\caption{$e_{max}- e_{min}= 14$}
    \end{subfigure}
        \begin{subfigure}[b]{0.32\textwidth}
        \includegraphics[width=\textwidth]{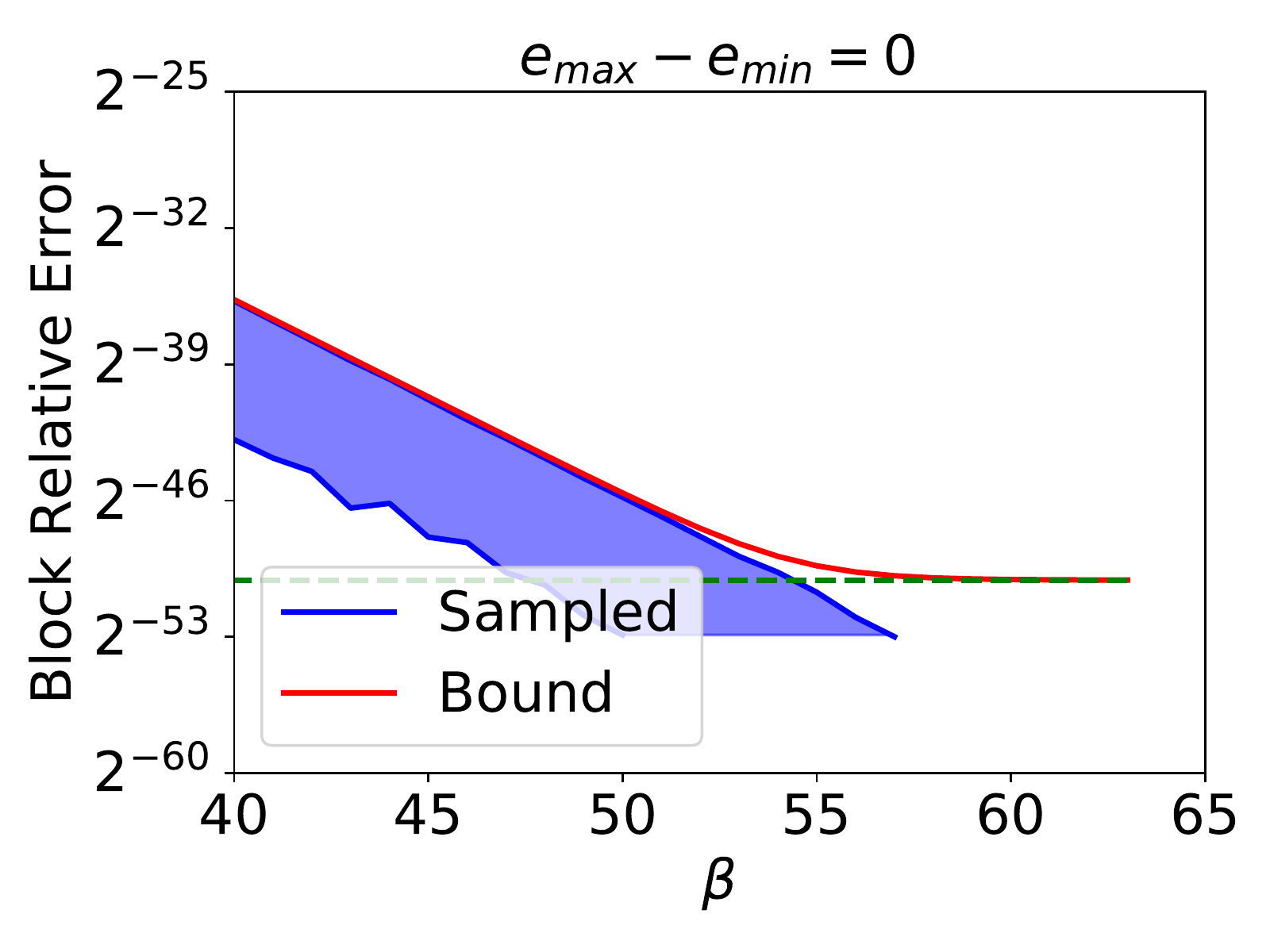}
    %    \caption{$e_{max}- e_{min}= 0$}
    \end{subfigure}
        \begin{subfigure}[b]{0.32\textwidth}
        \includegraphics[width=\textwidth]{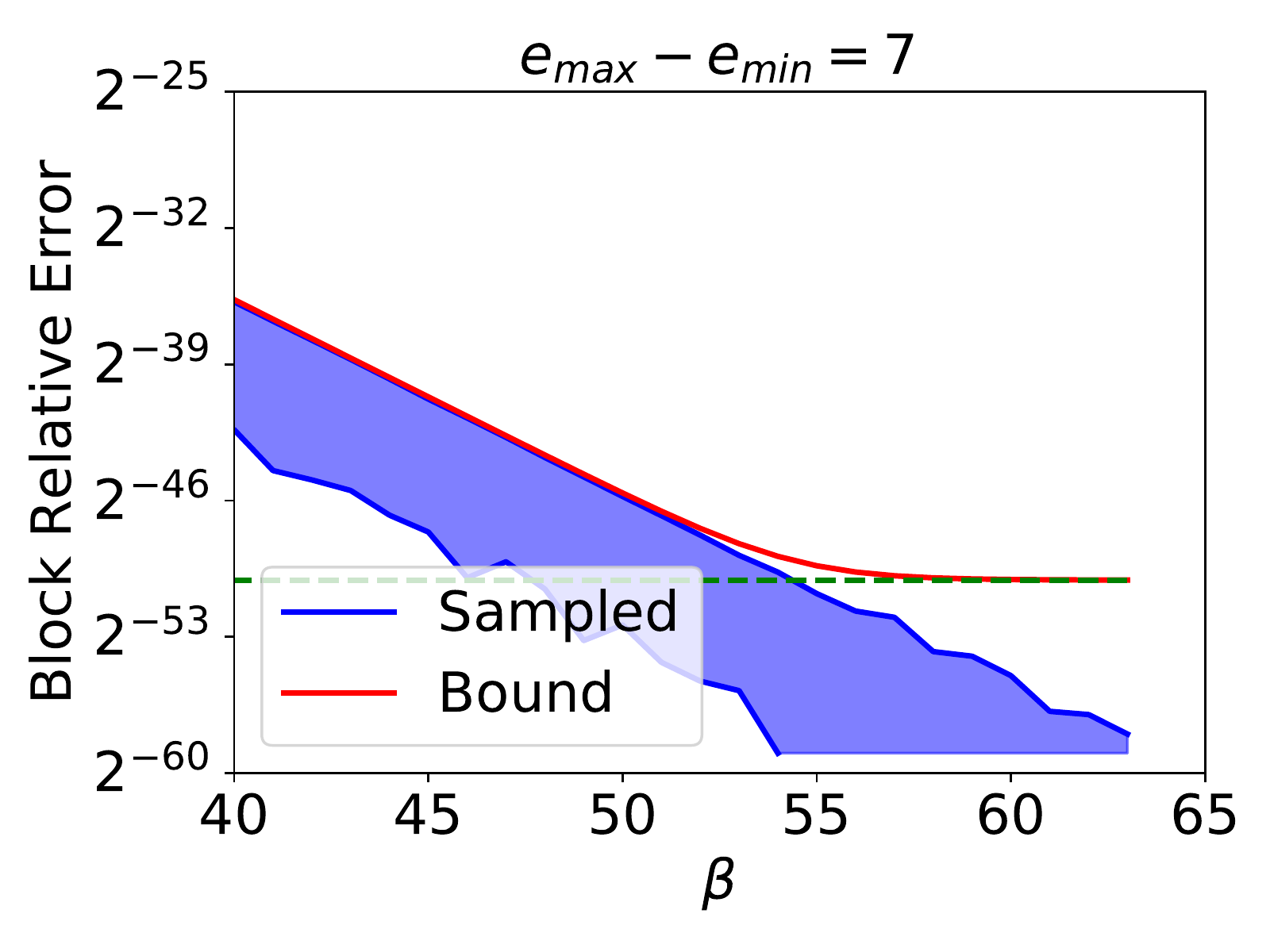}
%\caption{$e_{max}- e_{min}= 7$}
    \end{subfigure}
    \begin{subfigure}[b]{0.32\textwidth}
        \includegraphics[width=\textwidth]{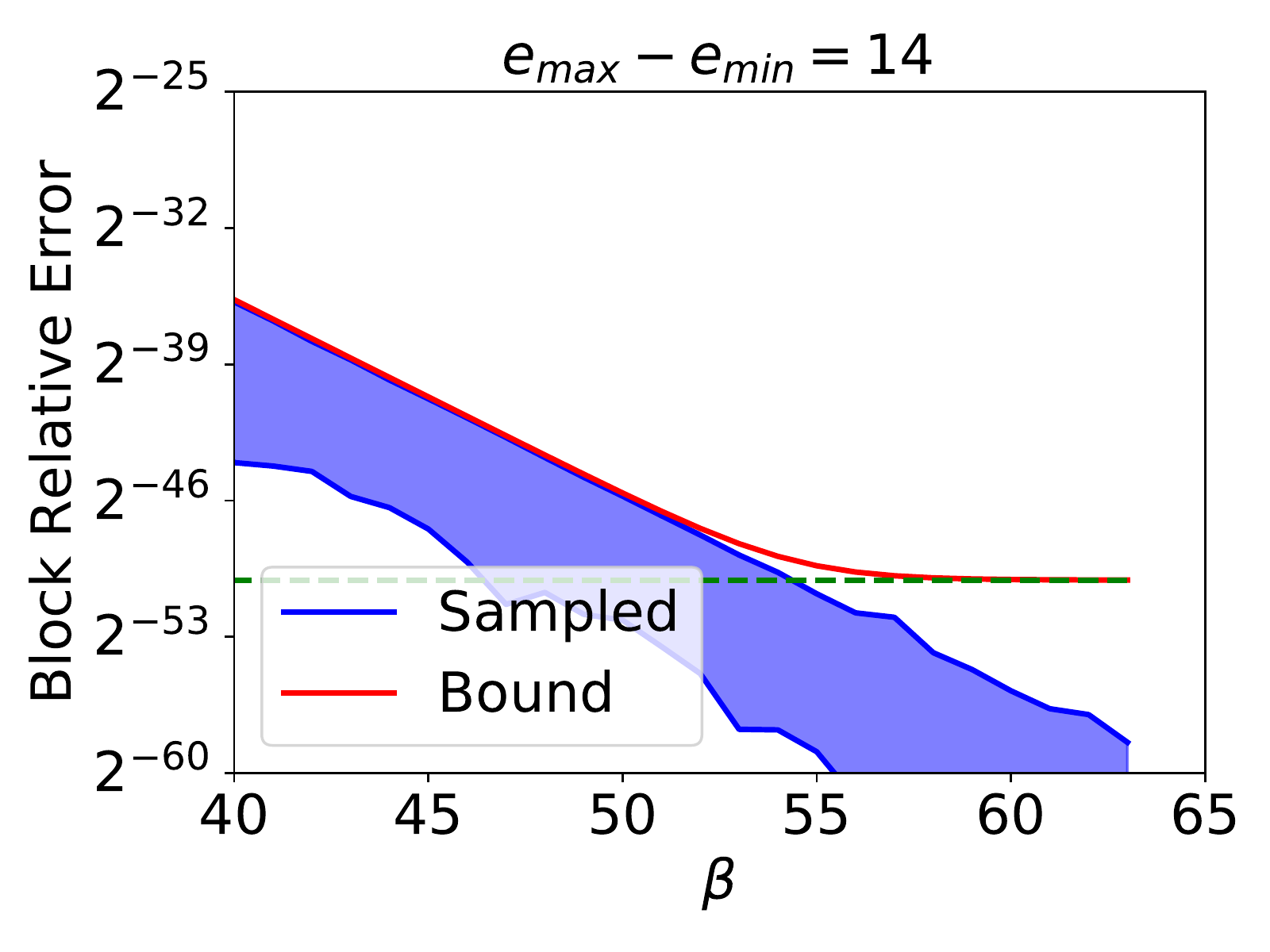}
%\caption{$e_{max}- e_{min}= 14$}
    \end{subfigure}
  \caption{1-d Example with double precision:  componentwise relative error (top) and block relative error (bottom) with respect to the precision parameter ($\beta$) for $ e_{max}-e_{min} \in\{ 0,7,14\}${. The blue band represents the sampled maximum and minimum error, the red line depicts the theoretical bound, and the dashed green line represents the asymptotic behavior of the theoretical bound. }}
      \label{fig:1d_double_1}
\end{figure}

\begin{figure}
    \centering
    \begin{subfigure}[b]{0.32\textwidth}
        \includegraphics[width=\textwidth]{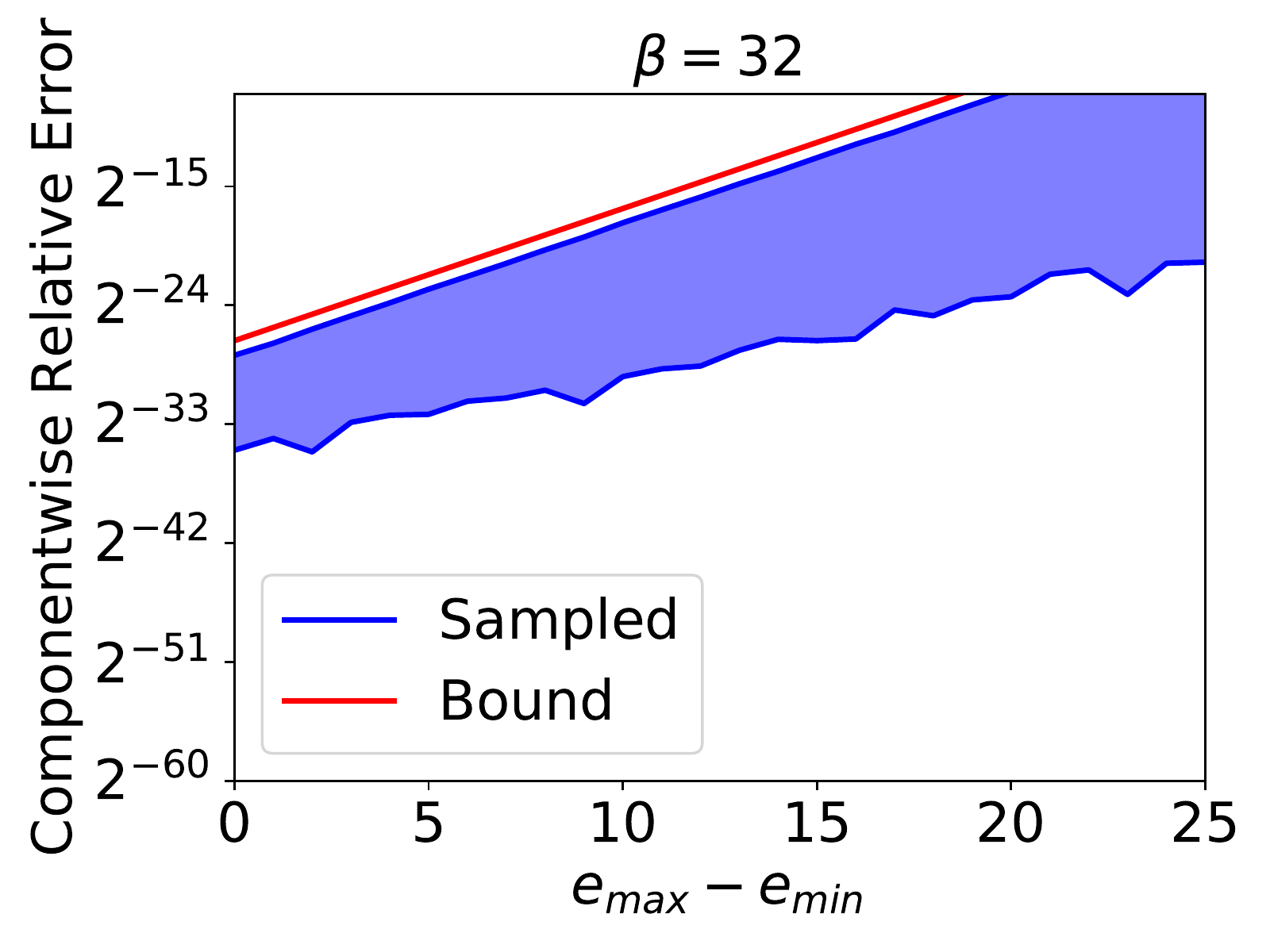}
 %       \caption{$\beta = 64$}
    \end{subfigure}
        \begin{subfigure}[b]{0.32\textwidth}
        \includegraphics[width=\textwidth]{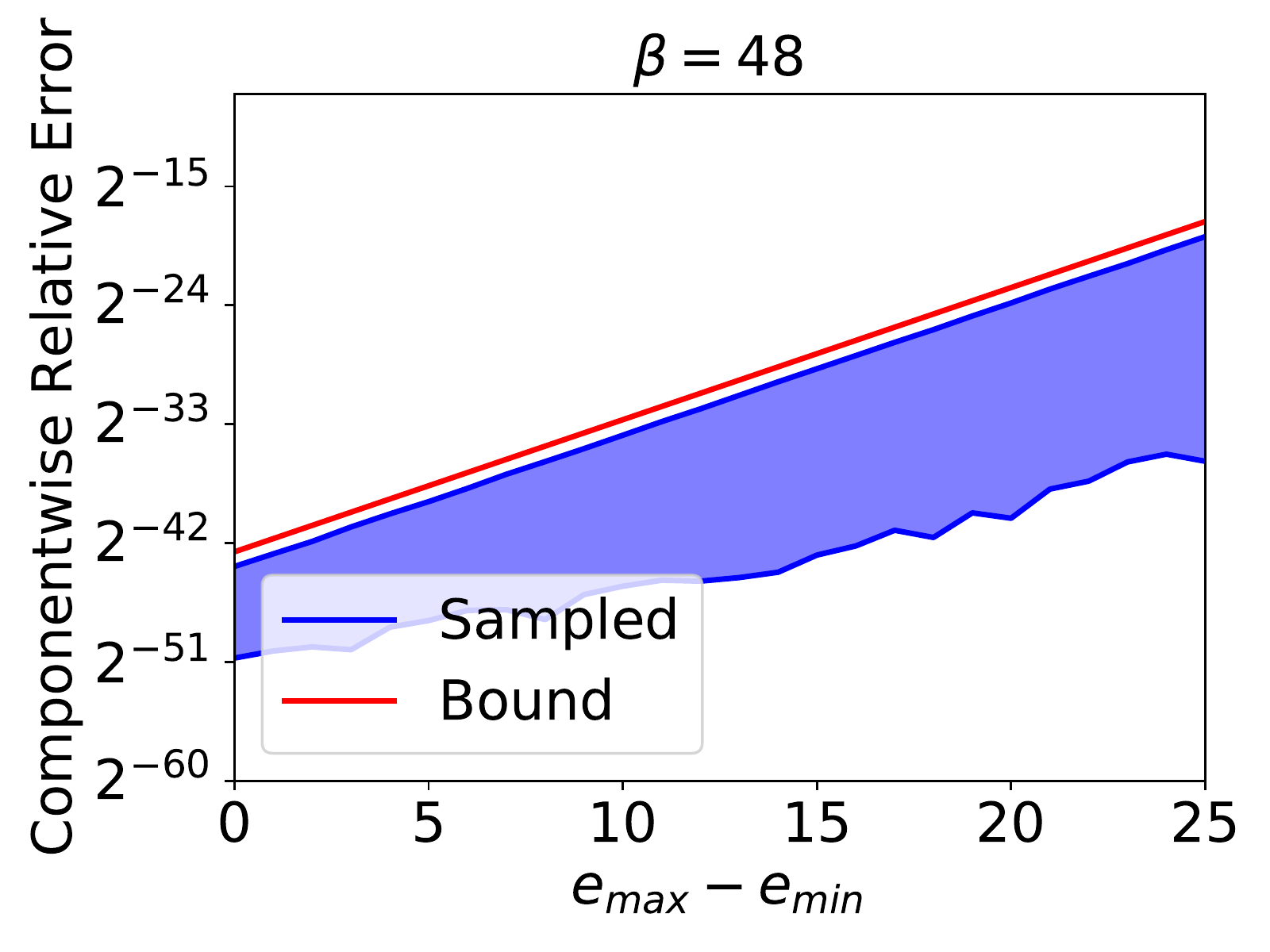}
%\caption{$\beta = 60$}
    \end{subfigure}
        \begin{subfigure}[b]{0.32\textwidth}
        \includegraphics[width=\textwidth]{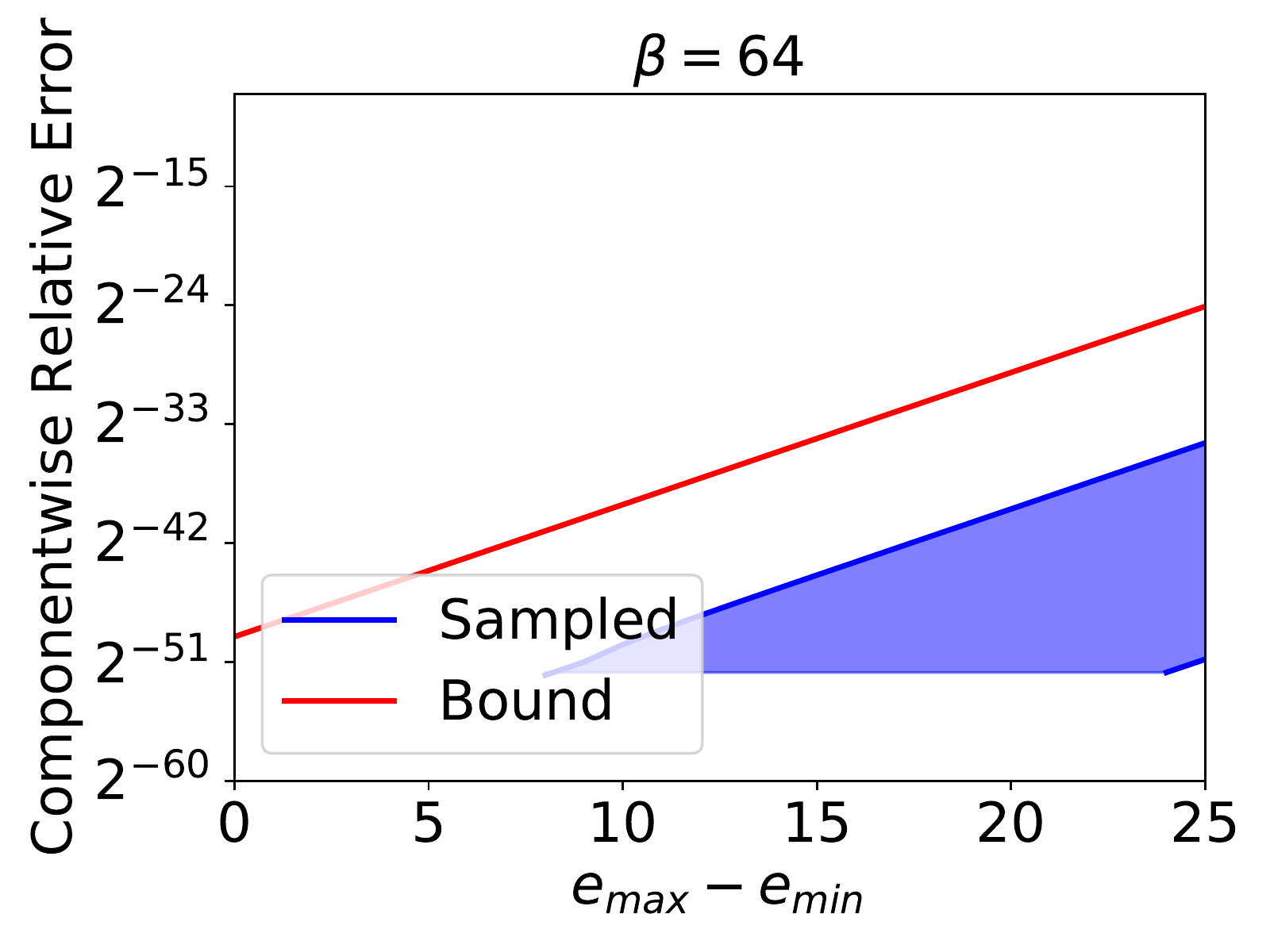}
%\caption{$\beta = 56$}
\end{subfigure}

    \begin{subfigure}[b]{0.32\textwidth}
        \includegraphics[width=\textwidth]{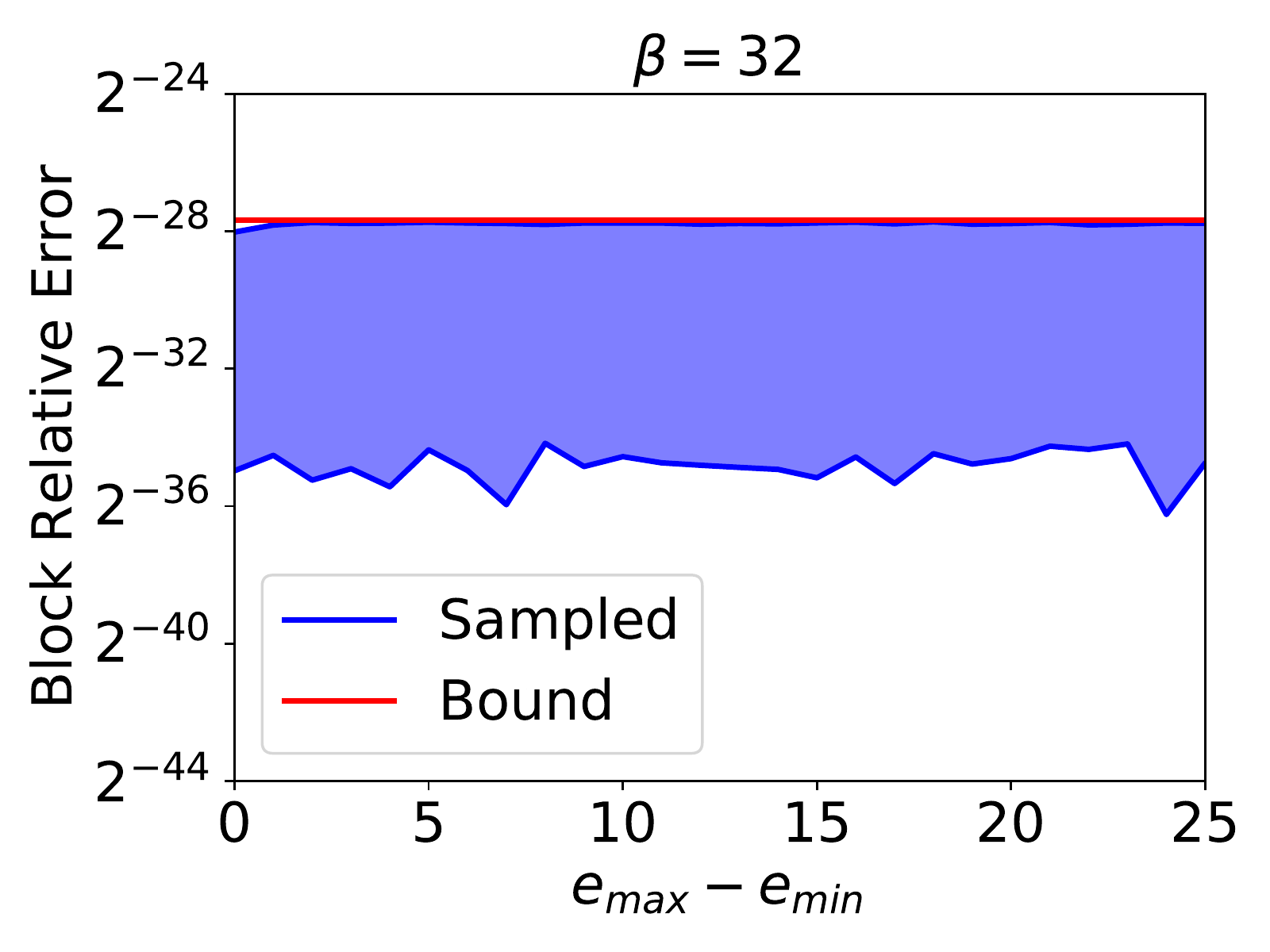}
 %       \caption{$\beta= 64$}
    \end{subfigure}
    \begin{subfigure}[b]{0.32\textwidth}
        \includegraphics[width=\textwidth]{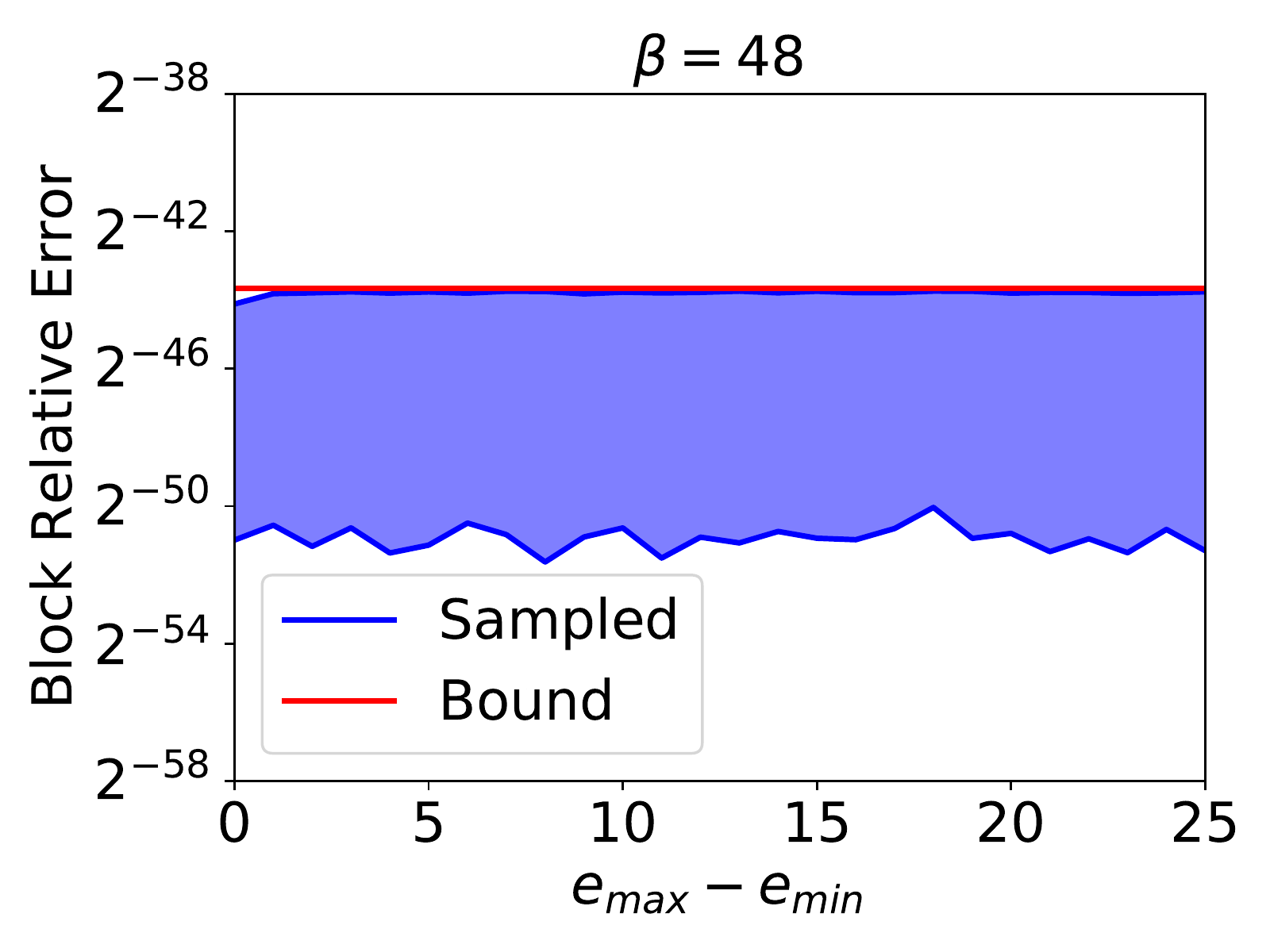}
%\caption{$\beta = 60$}
    \end{subfigure}
    \begin{subfigure}[b]{0.32\textwidth}
        \includegraphics[width=\textwidth]{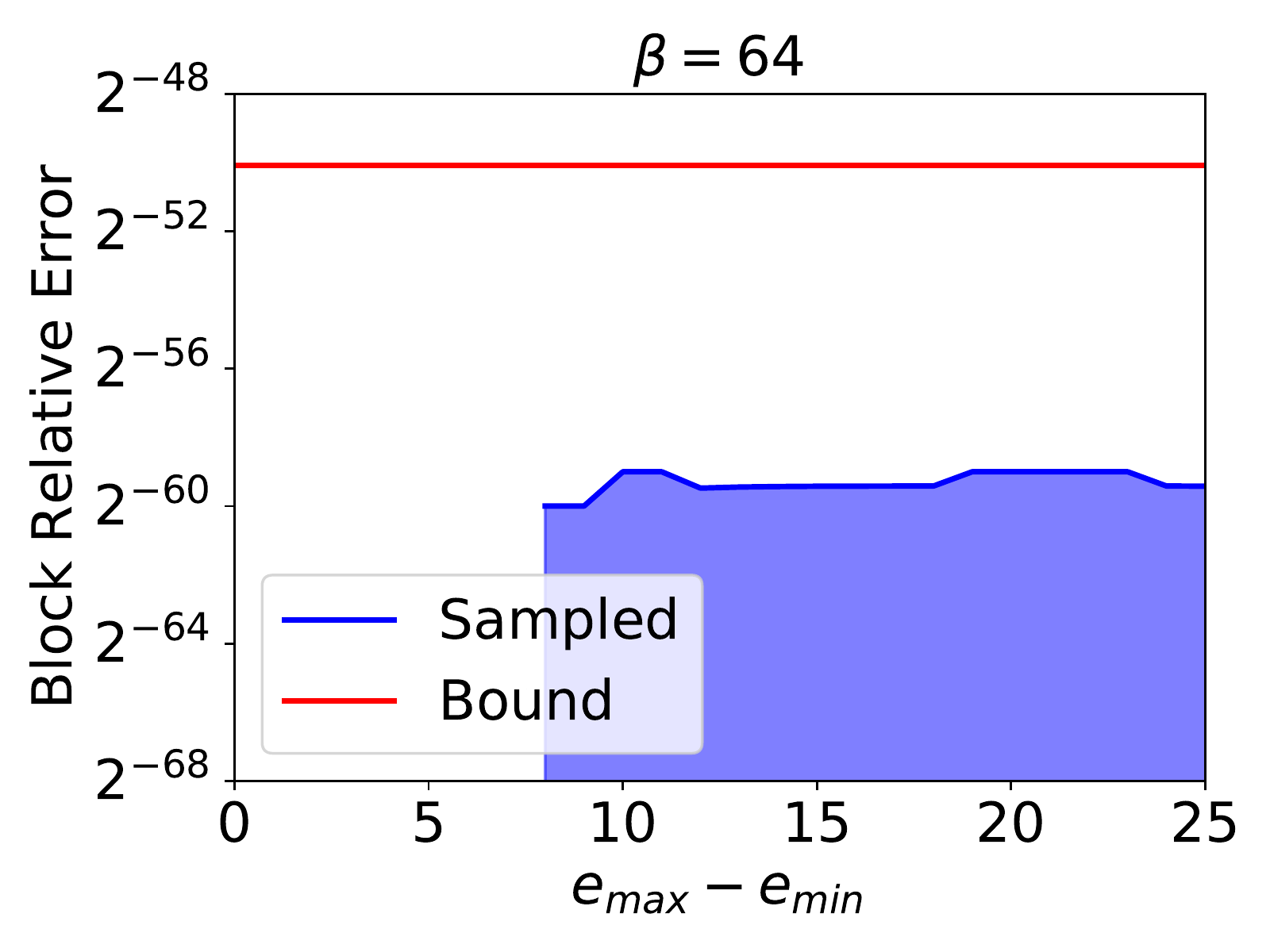}
%\caption{$\beta = 56$}
    \end{subfigure}
 \caption{1-d Example with double precision:  componentwise relative error (top) and block relative error (bottom) with respect to the difference in exponents ($e_{max}-e_{min}$) for $\beta \in \{32,48,64\}${. The blue band represents the sampled maximum and minimum error and the red line depicts the theoretical bound. }}
    \label{fig:1d_double_4}
\end{figure}

%%%%%%%%%%%% %%%%%%%%%%%%%%%%%%%%%%%%%%%%%%%%%%%%%%%
%3D, double precision  examples 
%%%%%%%%%%%% %%%%%%%%%%%%%%%%%%%%%%%%%%%%%%%%%%%%%%%

\begin{figure}
    \centering
    \begin{subfigure}[b]{0.32\textwidth}
        \includegraphics[width=\textwidth]{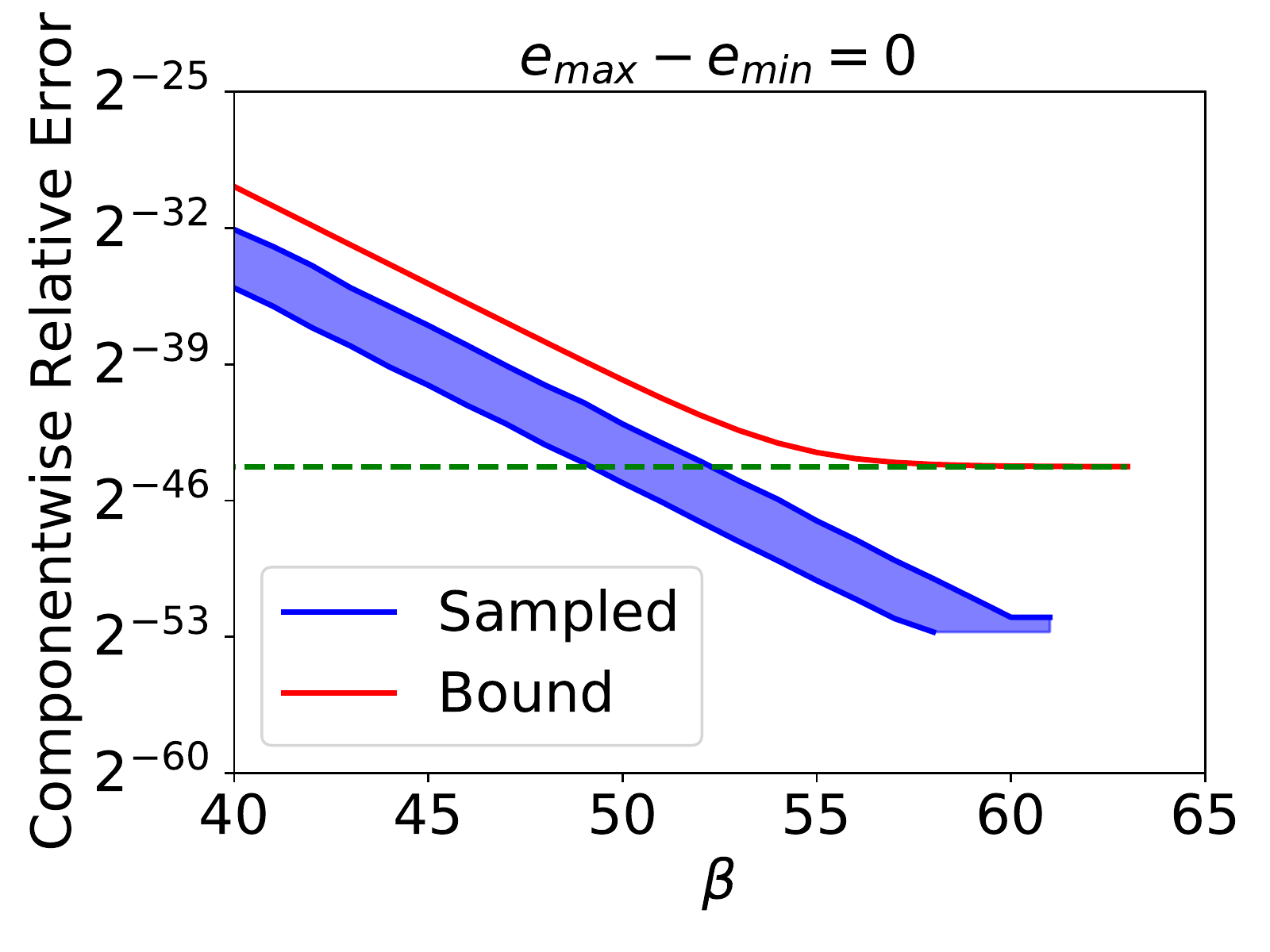}
  %      \caption{$e_{max}- e_{min}= 0$}
    \end{subfigure}
      \begin{subfigure}[b]{0.32\textwidth}
        \includegraphics[width=\textwidth]{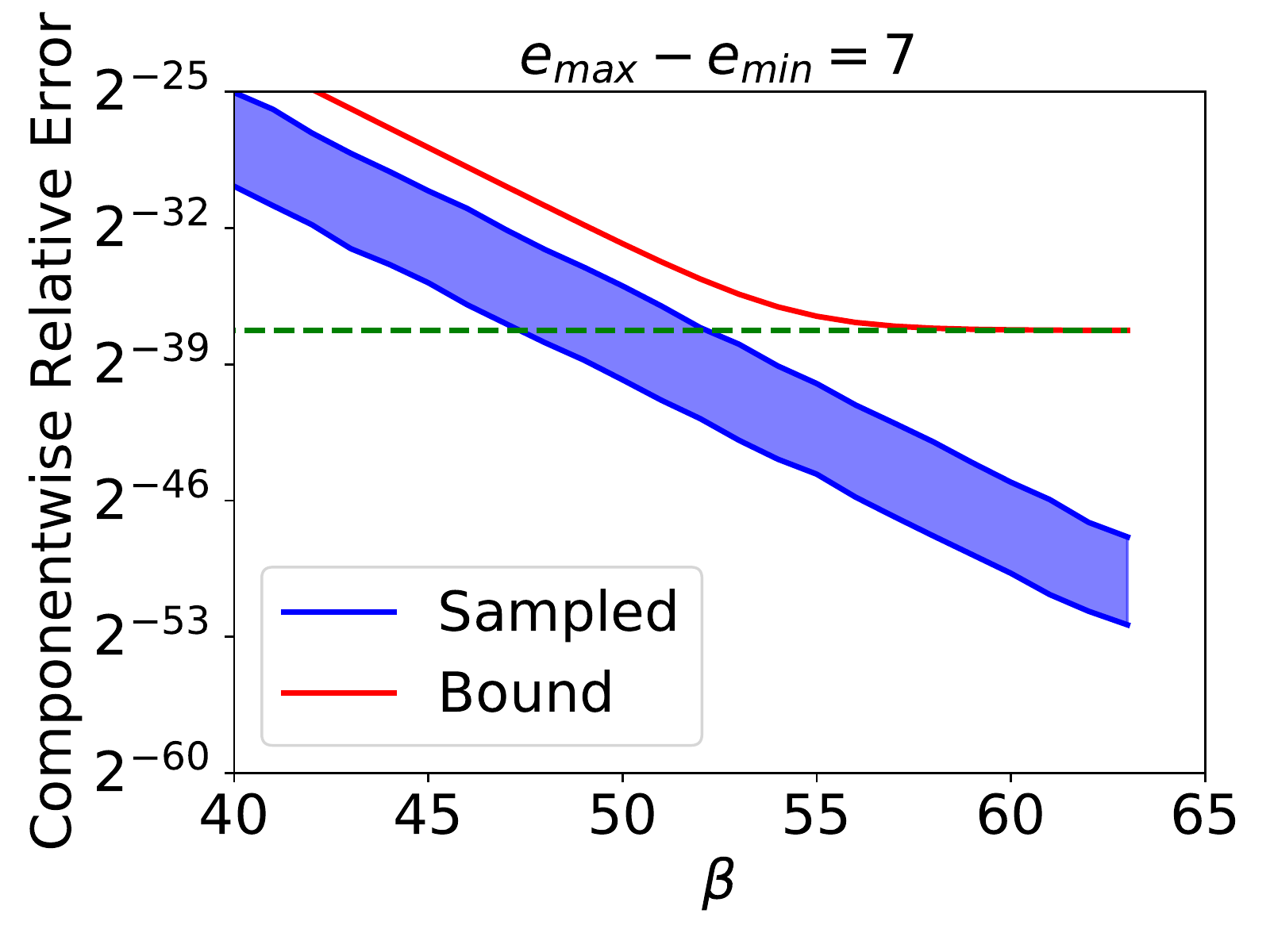}
%\caption{$e_{max}- e_{min}= 7$}
    \end{subfigure}
       \begin{subfigure}[b]{0.32\textwidth}
        \includegraphics[width=\textwidth]{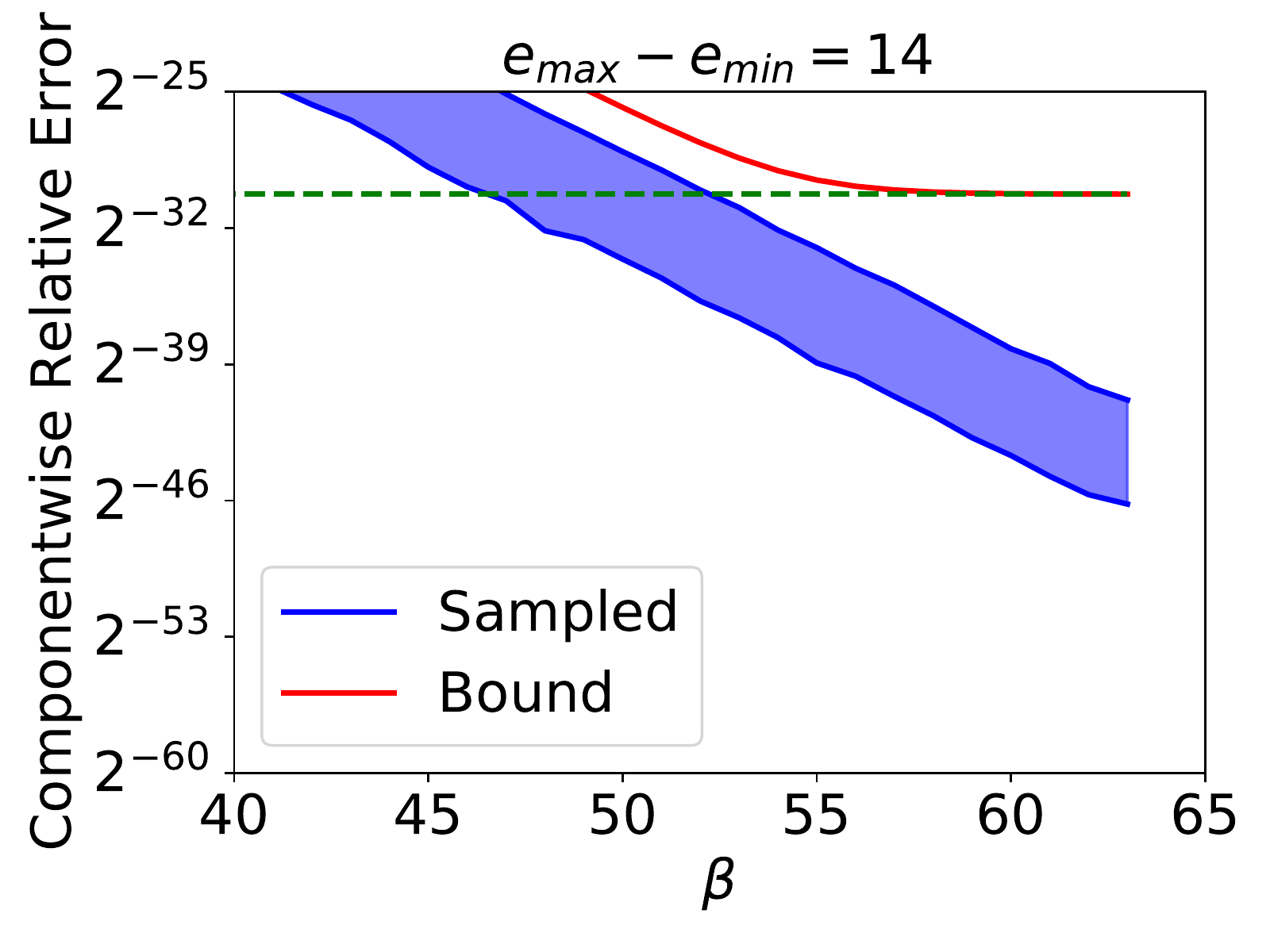}
%\caption{$e_{max}- e_{min}= 14$}

    \end{subfigure}

    \begin{subfigure}[b]{0.32\textwidth}
        \includegraphics[width=\textwidth]{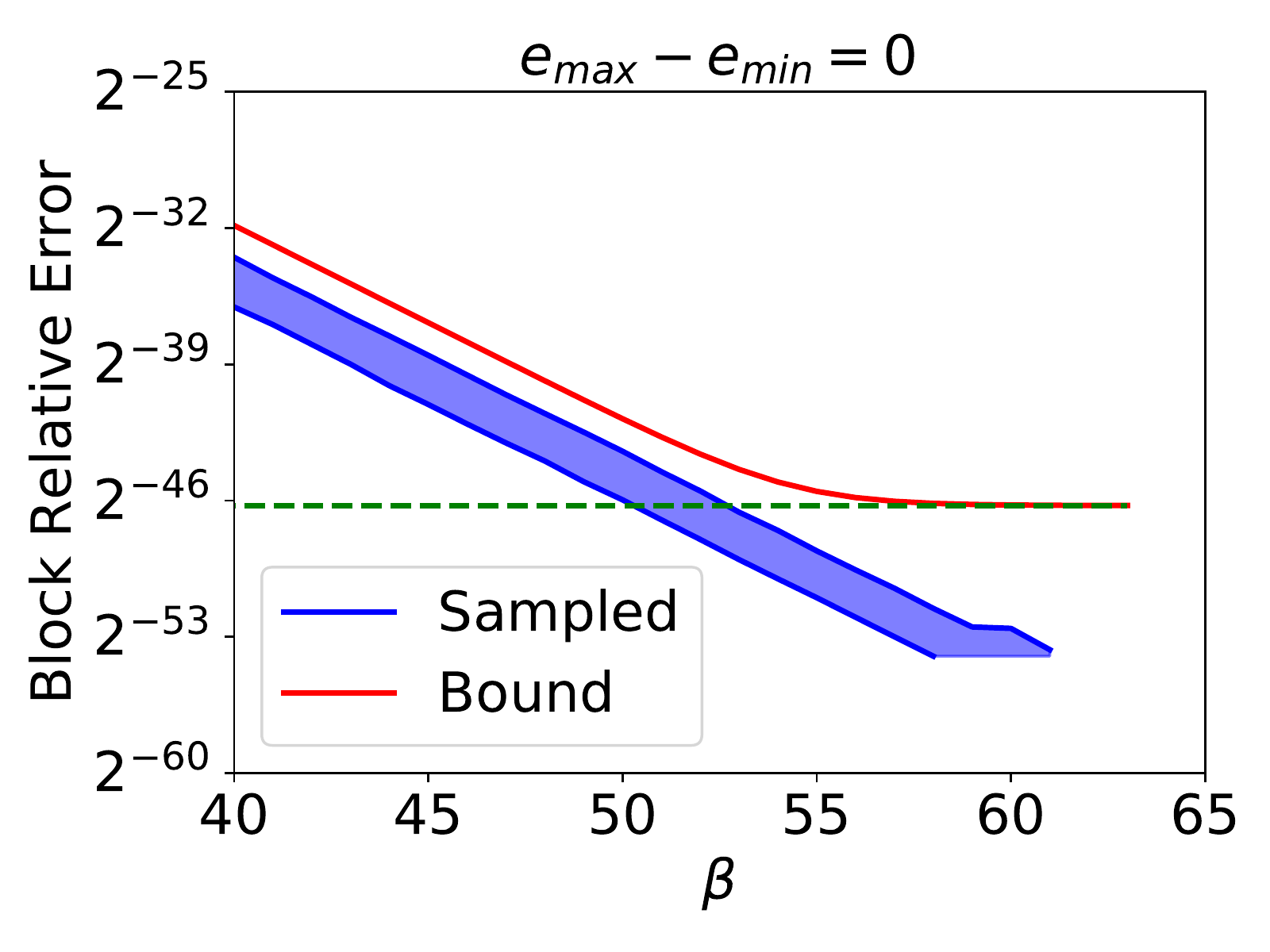}
   %     \caption{$e_{max}- e_{min}= 0$}
    \end{subfigure}
        \begin{subfigure}[b]{0.32\textwidth}
        \includegraphics[width=\textwidth]{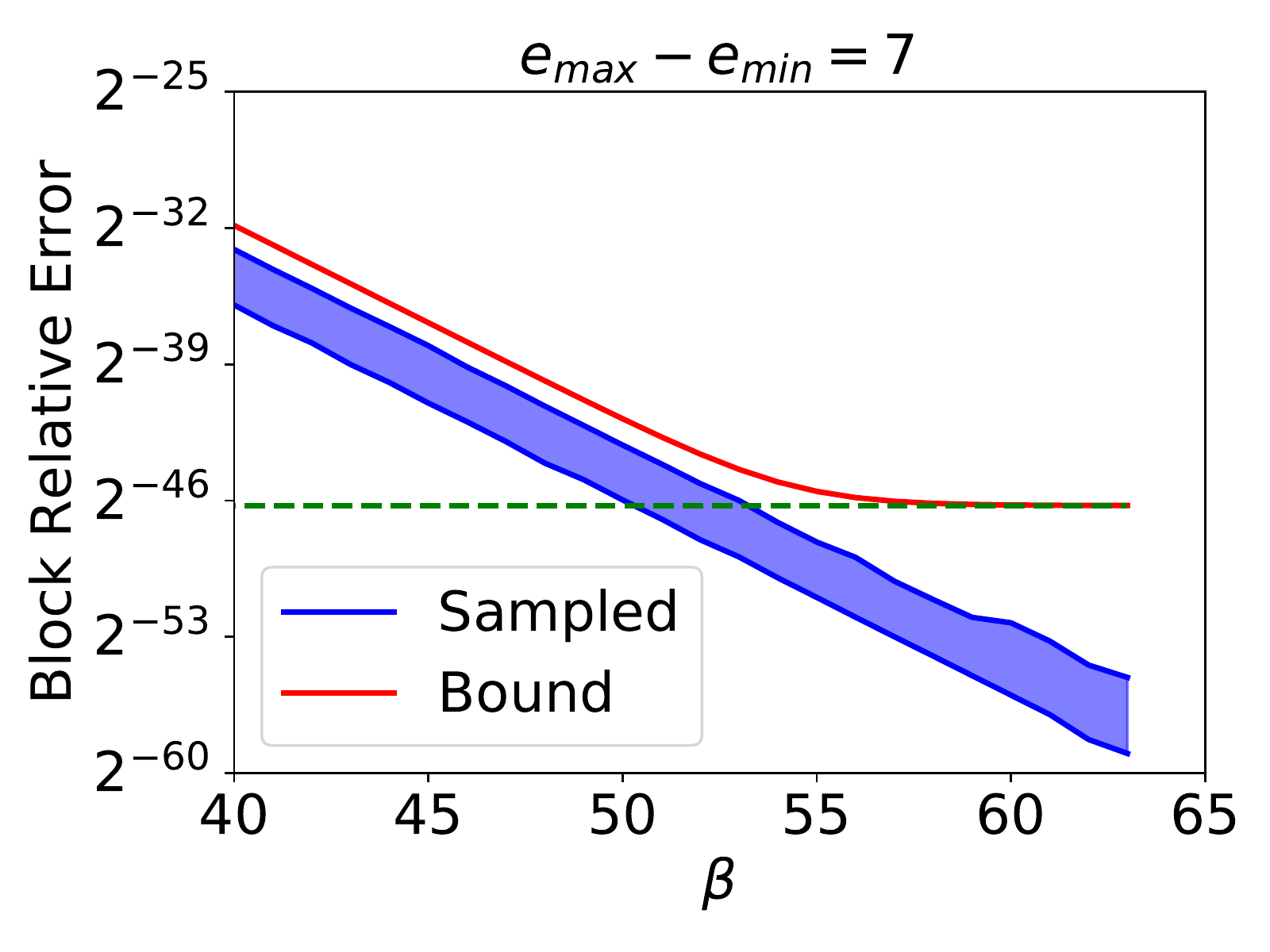}
%\caption{$e_{max}- e_{min}= 7$}
    \end{subfigure}
    \begin{subfigure}[b]{0.32\textwidth}
        \includegraphics[width=\textwidth]{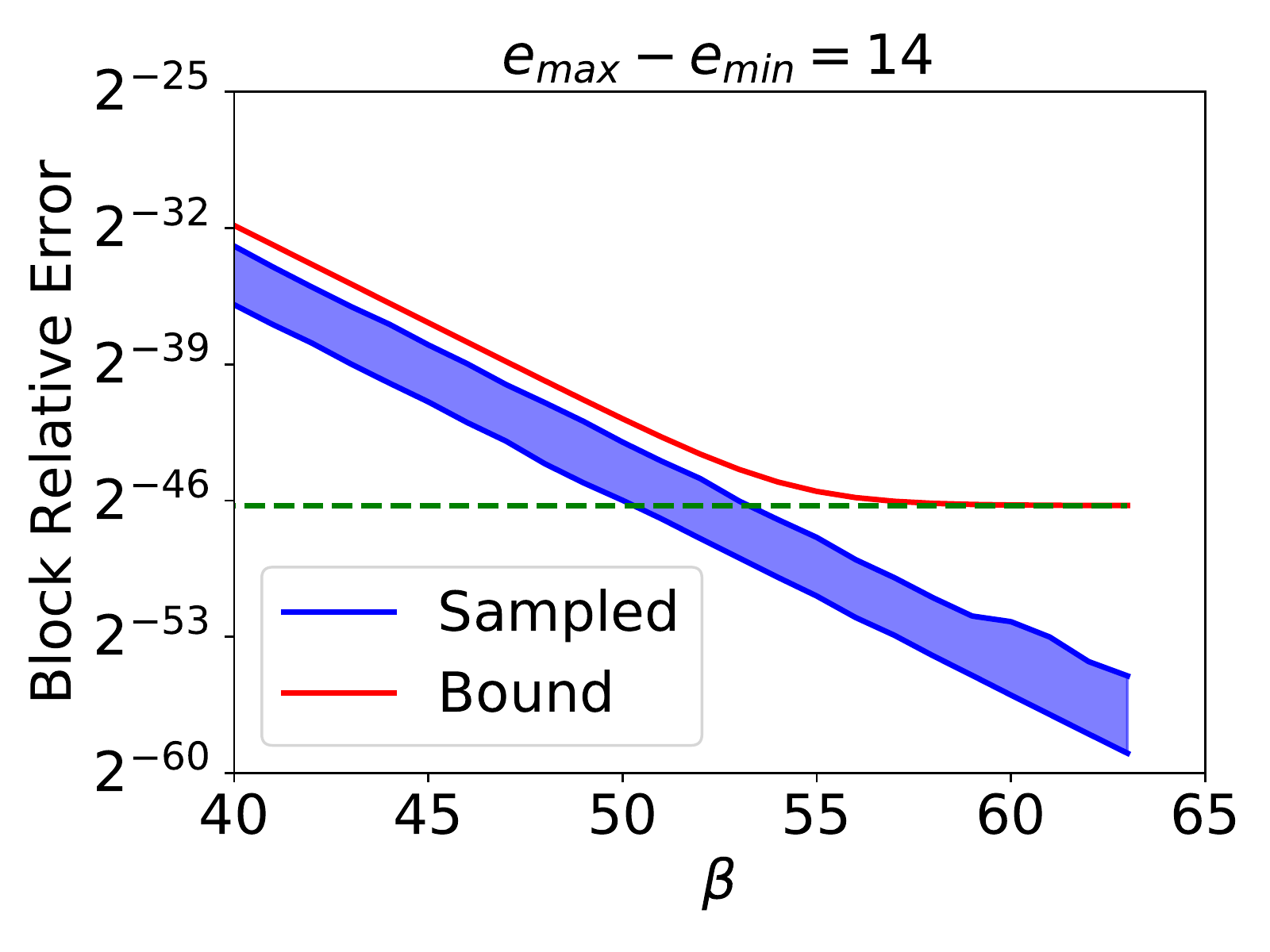}
%\caption{$e_{max}- e_{min}= 14$}
    \end{subfigure}
  \caption{3-d Example with double precision:  componentwise relative error (top) and block relative error (bottom) with respect to the precision parameter ($\beta$) for $ e_{max}-e_{min} \in \{0,7,14\}${. The blue band represents the sampled maximum and minimum of the error, the red line depicts the theoretical bound, and the dashed green line represents the asymptotic behavior of the theoretical bound. }}
    \label{fig:3d_double_1}
\end{figure}

\begin{figure}
    \centering
    \begin{subfigure}[b]{0.32\textwidth}
        \includegraphics[width=\textwidth]{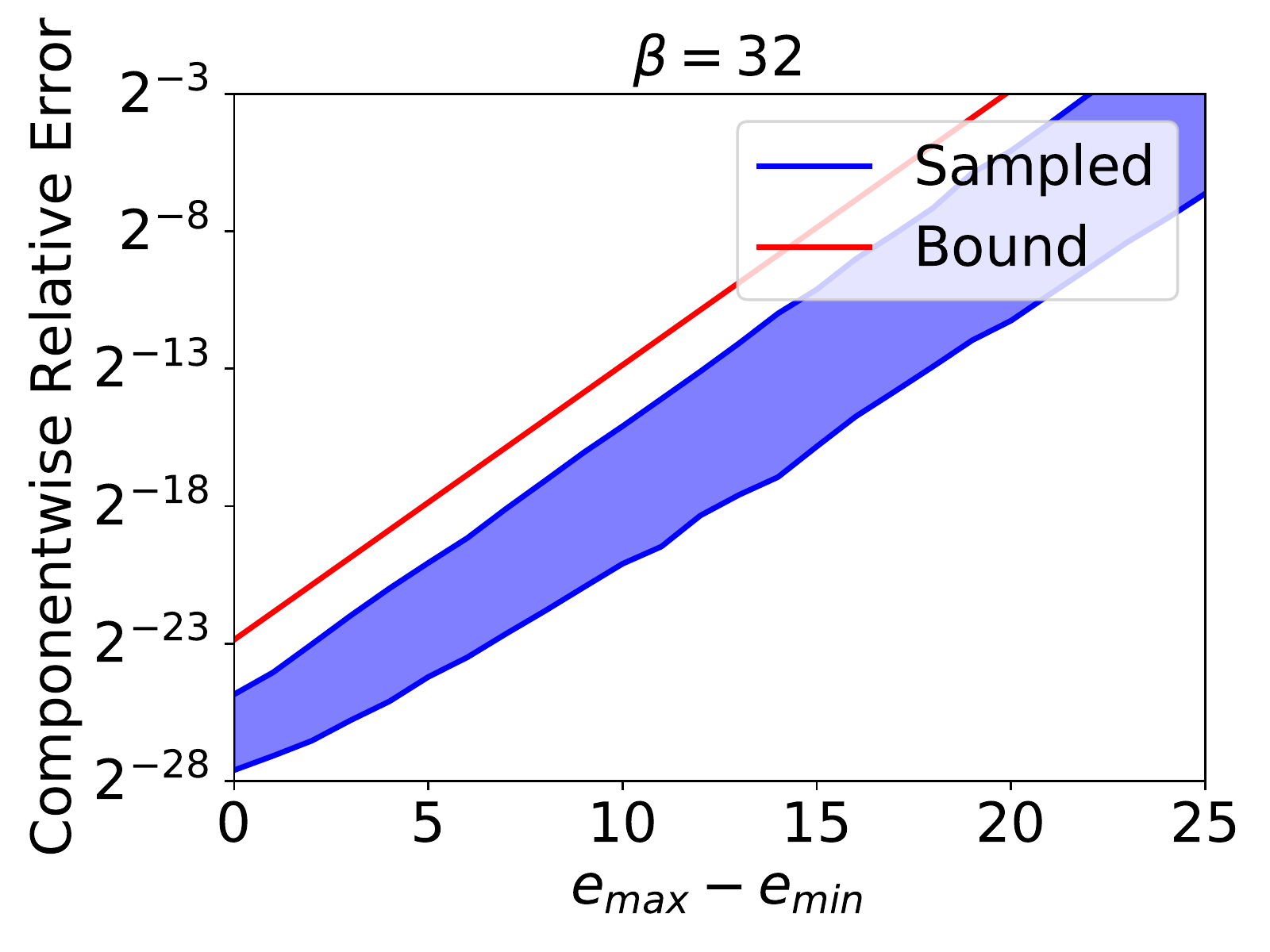}
     %   \caption{$\beta = 64$}
    \end{subfigure}
        \begin{subfigure}[b]{0.32\textwidth}
        \includegraphics[width=\textwidth]{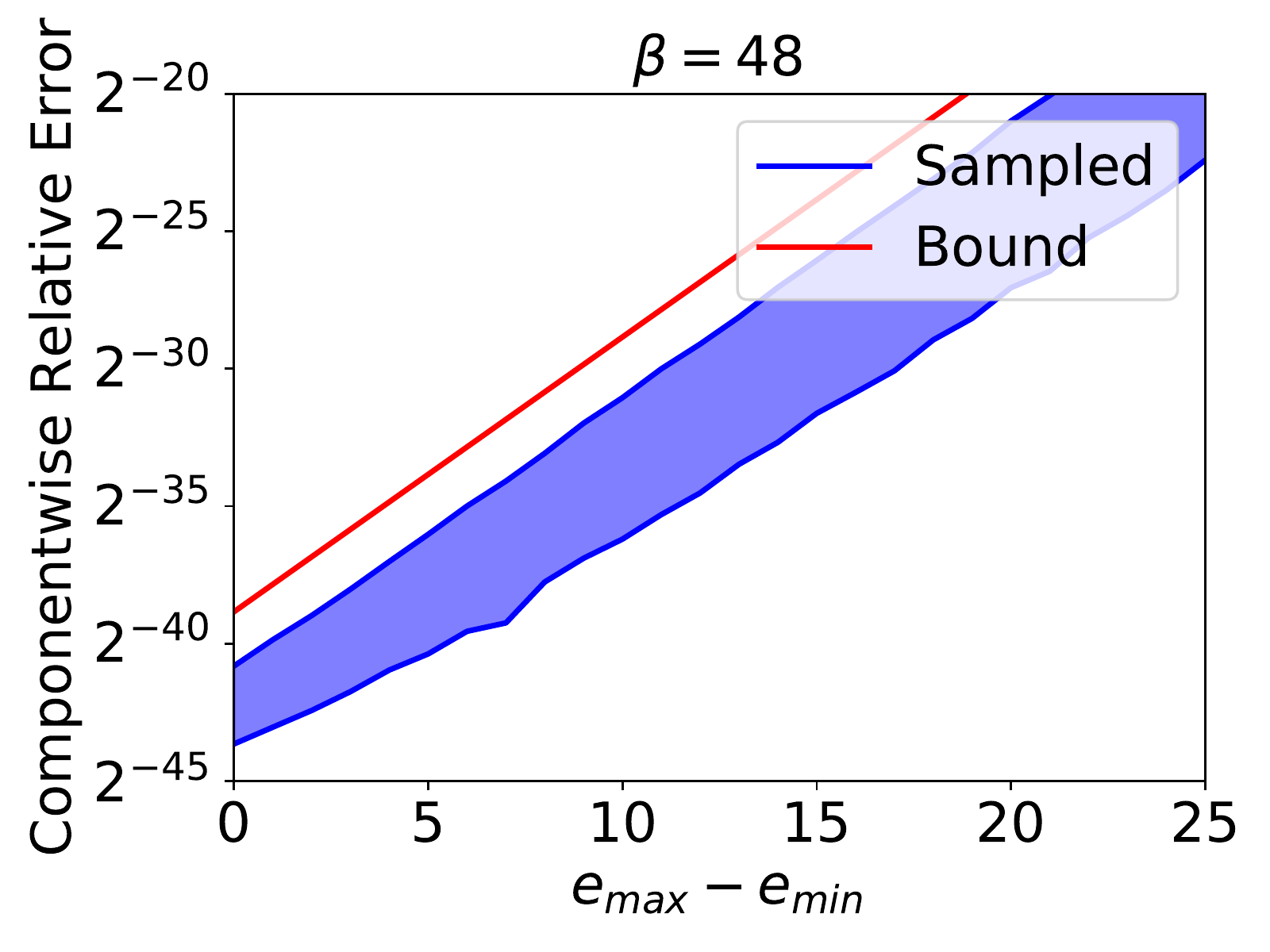}
%\caption{$\beta = 60$}
    \end{subfigure}
        \begin{subfigure}[b]{0.32\textwidth}
        \includegraphics[width=\textwidth]{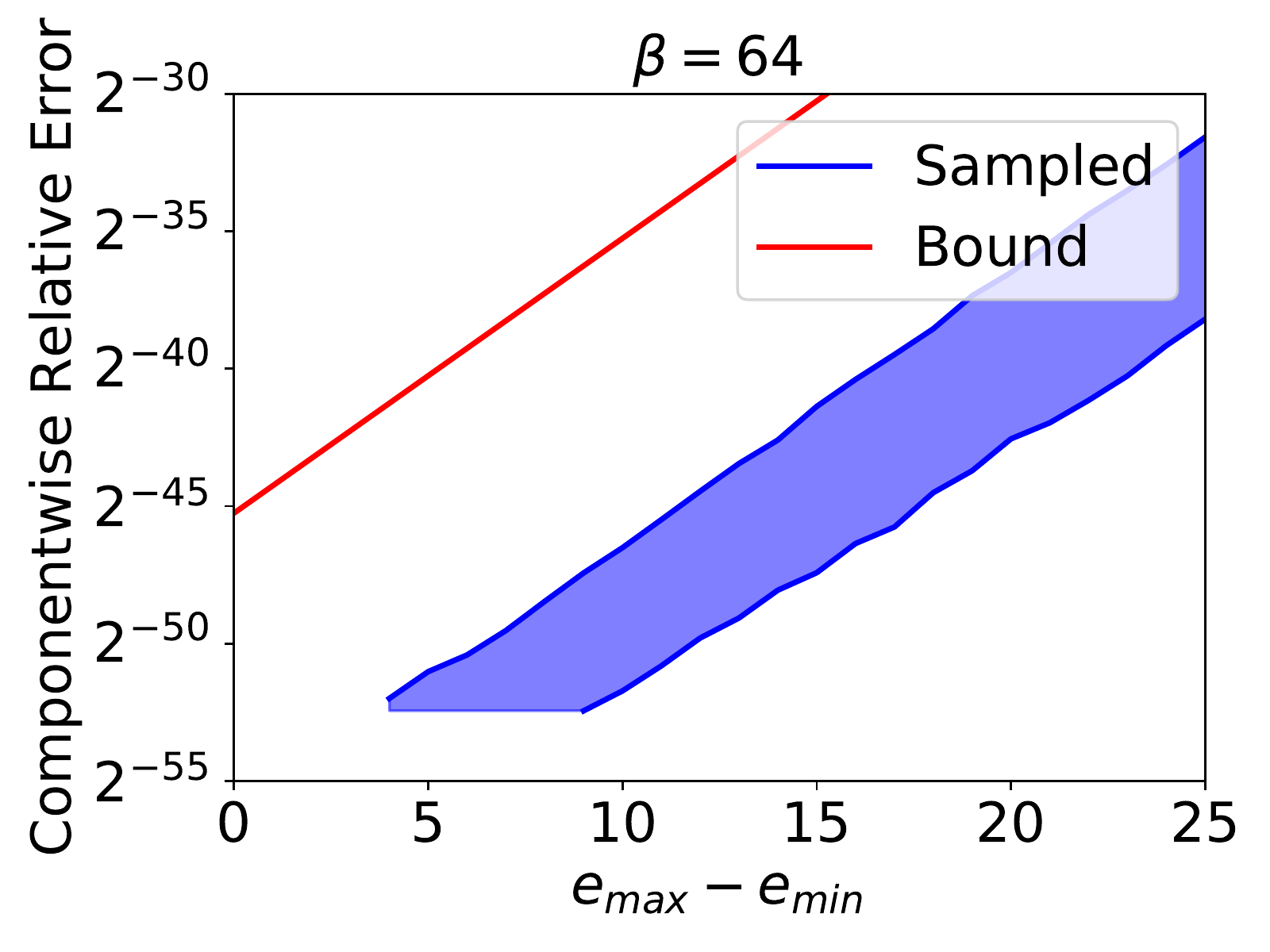}
%\caption{$\beta = 56$}
\end{subfigure}
        \begin{subfigure}[b]{0.32\textwidth}
        \includegraphics[width=\textwidth]{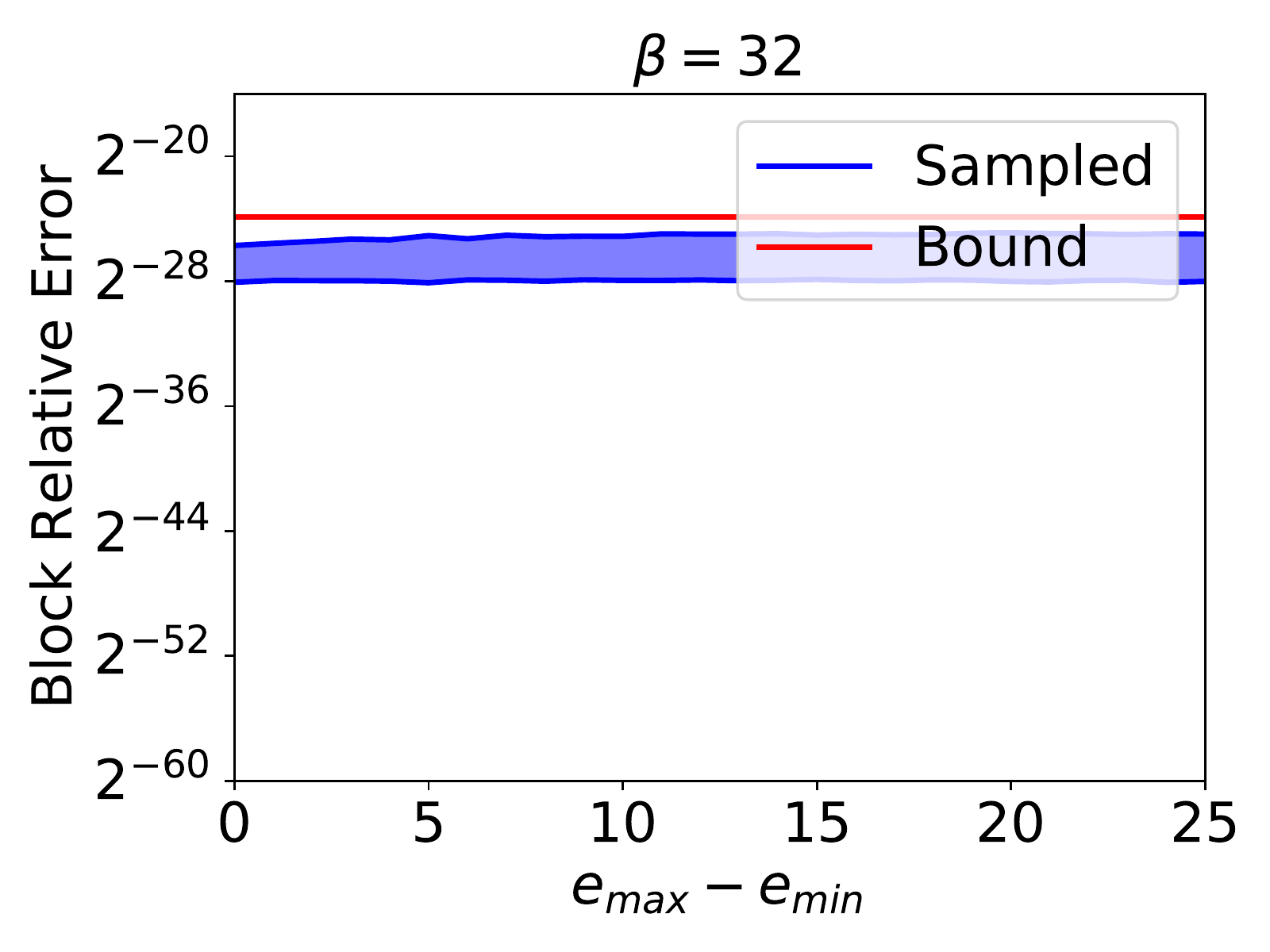}
   %     \caption{$\beta= 64$}
    \end{subfigure}
    \begin{subfigure}[b]{0.32\textwidth}
        \includegraphics[width=\textwidth]{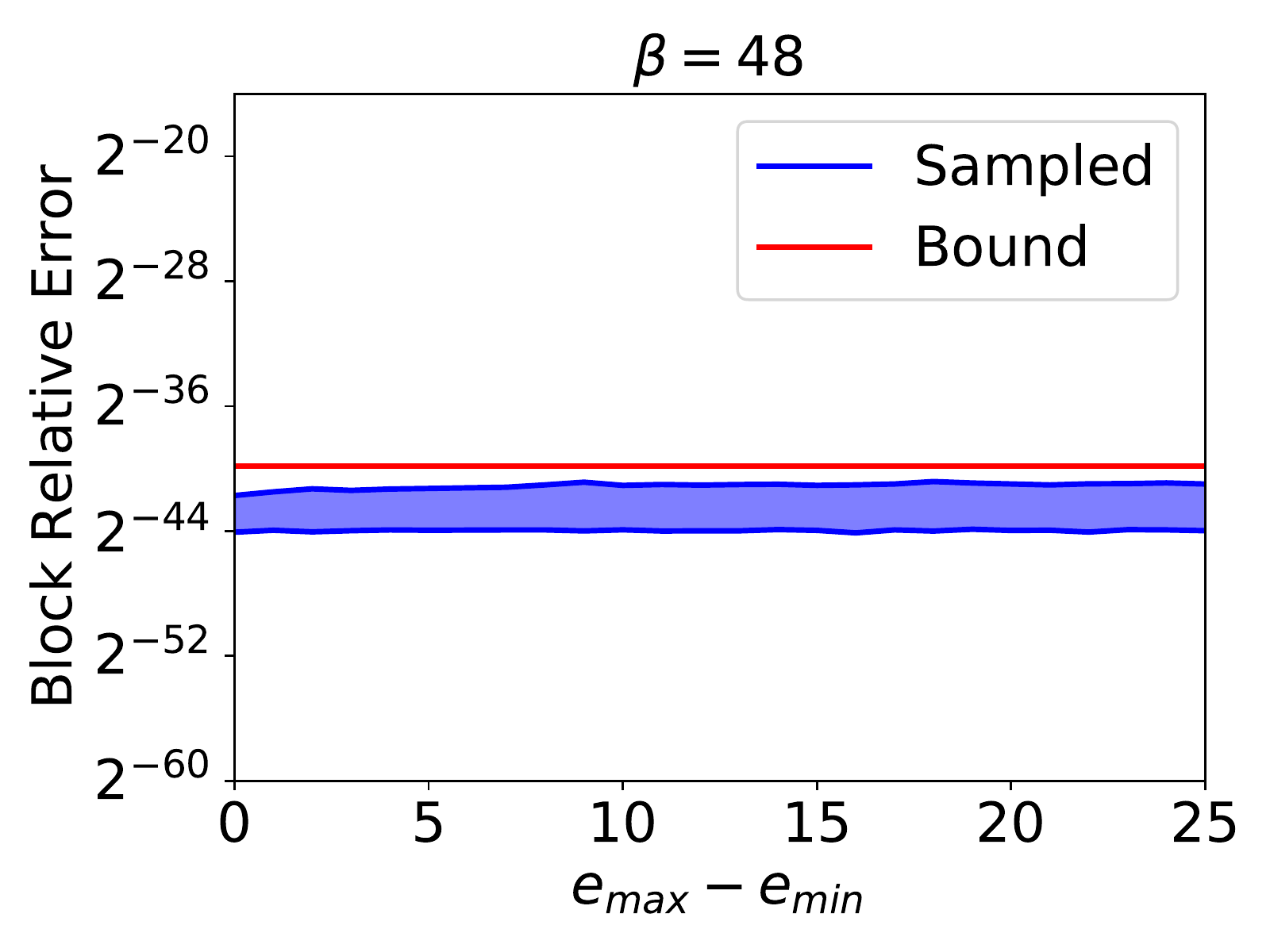}
%\caption{$\beta = 60$}
    \end{subfigure}
    \begin{subfigure}[b]{0.32\textwidth}
        \includegraphics[width=\textwidth]{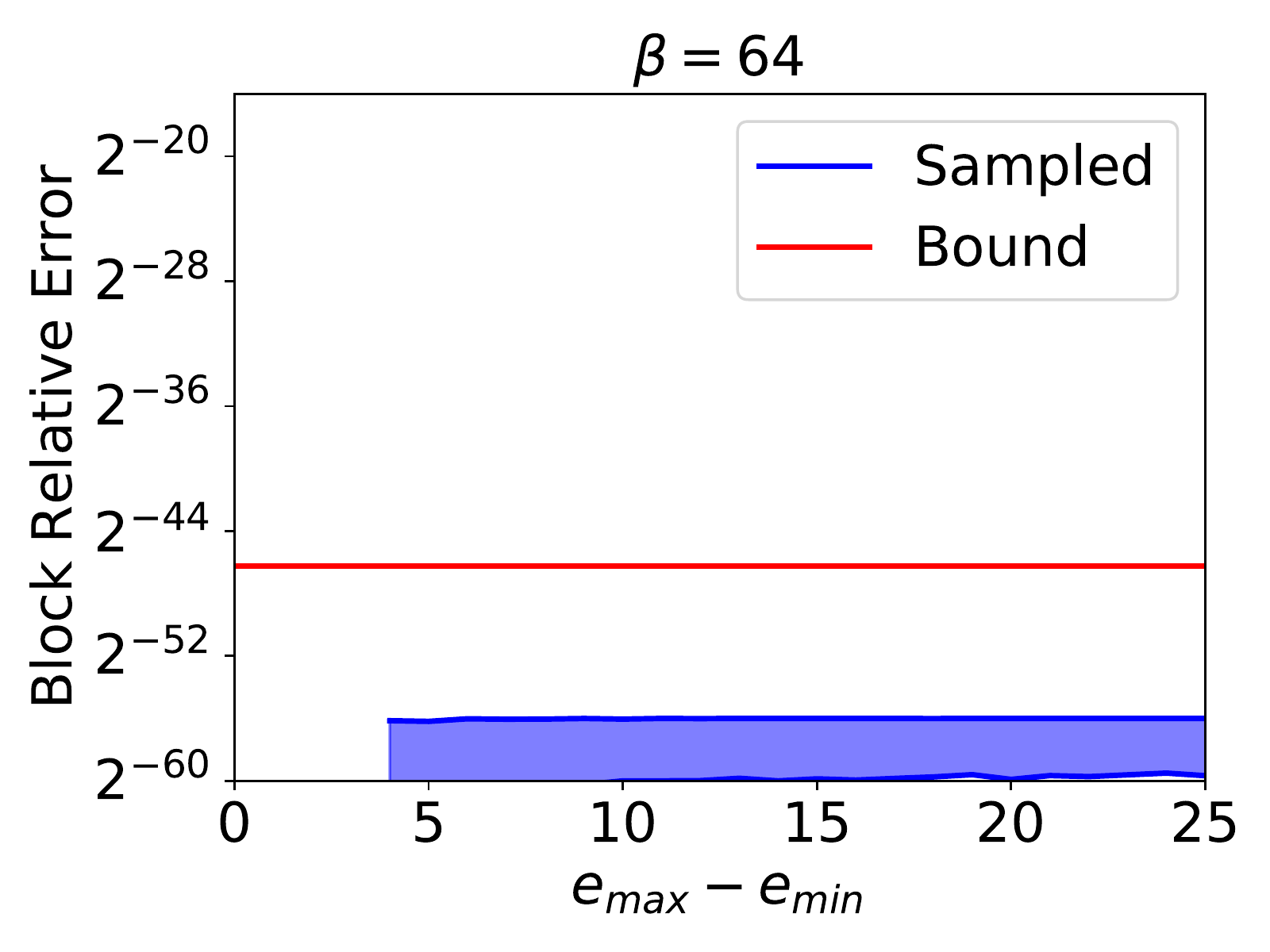}
%\caption{$\beta = 56$}
    \end{subfigure}
 \caption{3-d Example with double precision:  componentwise relative error (top) and block relative error (bottom) with respect to the difference in exponents ($e_{max}-e_{min}$) for $\beta = \{32,48,64\} ${. The blue band represents the sampled maximum and minimum of the error and the red line depicts the theoretical bound. }}
    \label{fig:3d_double_4}
\end{figure}

%%%%%%%%%%%% %%%%%%%%%%%%%%%%%%%%%%%%%%%%%%%%%%%%%%%
%Real-world Viscosity example in double precision 
%%%%%%%%%%%% %%%%%%%%%%%%%%%%%%%%%%%%%%%%%%%%%%%%%%%
\subsection{Real-World Example}
\label{ex:RealWorld}
%So far, the error analysis as been conducted for a single block of
%generated data.
For this example, we compress data from a real-world three dimensional
viscosity and density field from 
a Rayleigh-Taylor instability simulation produced by Miranda
\cite{viscosity}. For the viscosity field, the average exponent range over all blocks is
approximately 7.32. This data set is a highly variable example, as the viscosity
values are signed and have a high dynamic range.  The density field has a much smaller dynamic range than the viscosity field. That means the density field is a more compressible data set for ZFP. For both fields, the same value of
$\beta$ was used across all blocks during 
compression to simplify the visualization of the results. In Figure
\ref{fig:viscosity} and \ref{fig:density}, the block relative error is plotted after the data
has been compressed and decompressed (blue) as a function of
$\beta$. The theoretical bound is plotted in red.  Again, we conclude
that the theoretical bound completely bounds the true error for both examples. As for the compression ratio, since the density field has a smaller dynamic range, there is a substantial increase in the compression ratio for every bit plane removed, especially compared to the viscosity field. It can be concluded, that for some error tolerance, ZFP compresses at a higher ratio for data that is ``smooth," i.e., the exponent range for each block is small. 

\begin{figure}
    \centering
                    \begin{subfigure}[b]{0.32\textwidth}
        \includegraphics[width=\textwidth]{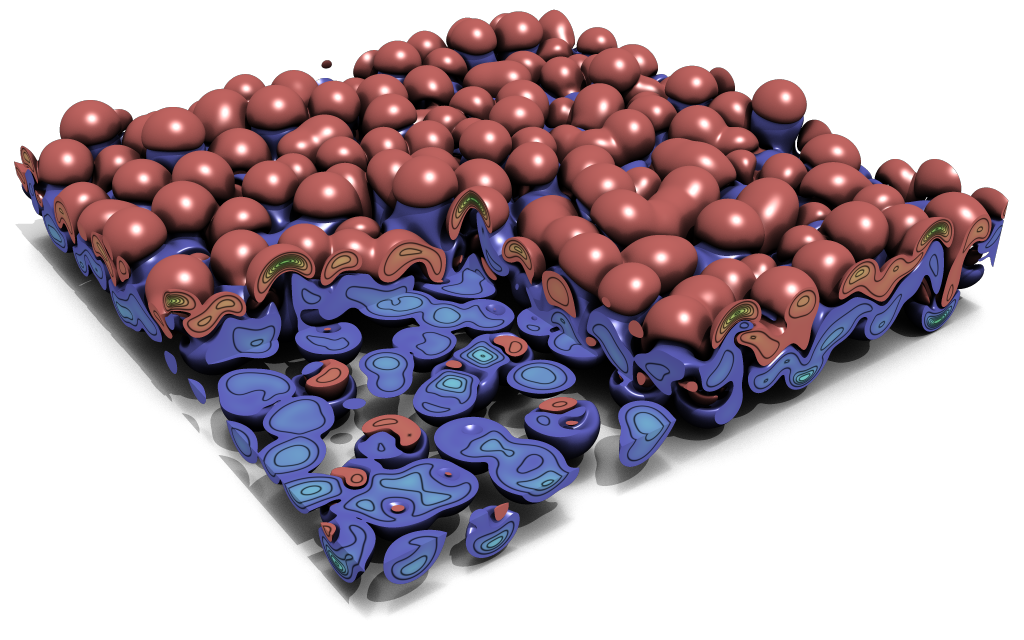}
        \end{subfigure}
        \begin{subfigure}[b]{0.32\textwidth}
        \includegraphics[width=\textwidth]{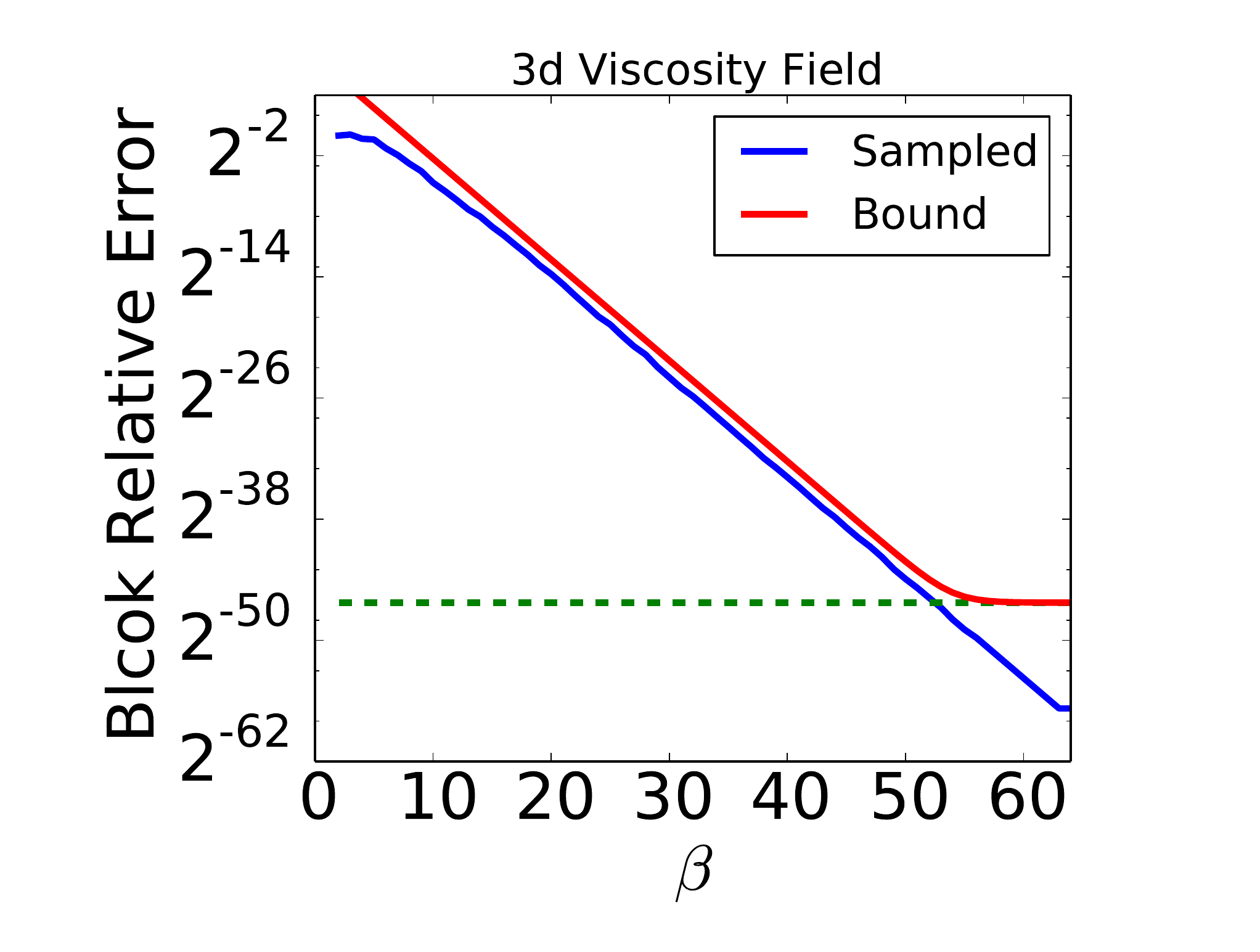}
        \end{subfigure}
             \begin{subfigure}[b]{0.32\textwidth}
        \includegraphics[width=\textwidth]{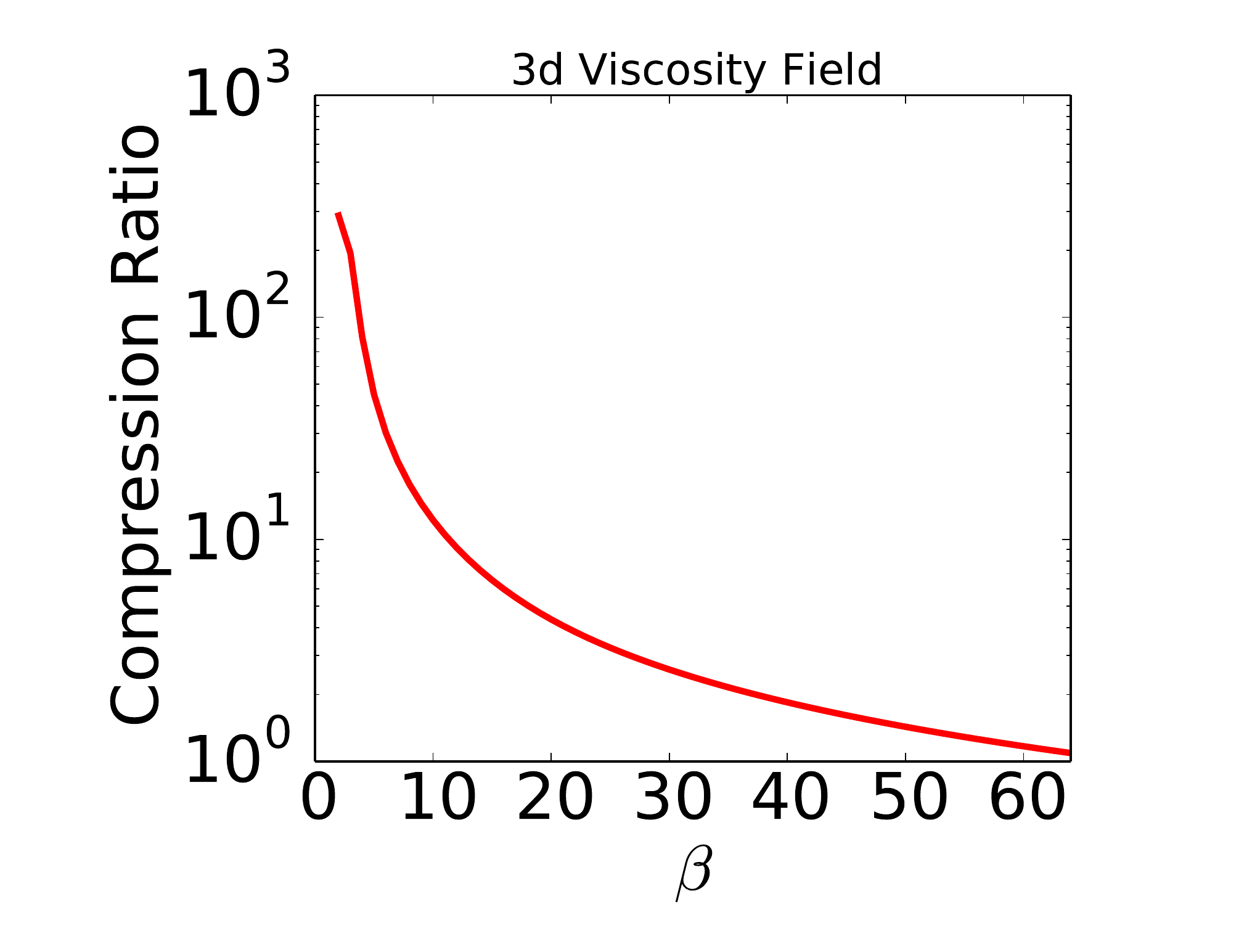}
        \end{subfigure}

           \caption{3-d viscosity field example with double precision. On the left is a 3-d rendering of the viscosity field. In the middle is the maximum block relative error as a function of the precision parameter $\beta$. The blue line is the error from ZFP compression and decompression with fixed $\beta$ and the red line depicts the theoretical bound. On the right is the compression ratio as a function of $\beta$.  }
                              \label{fig:viscosity}

\end{figure}

\begin{figure}
    \centering
     \begin{subfigure}[b]{0.32\textwidth}
        \includegraphics[width=\textwidth]{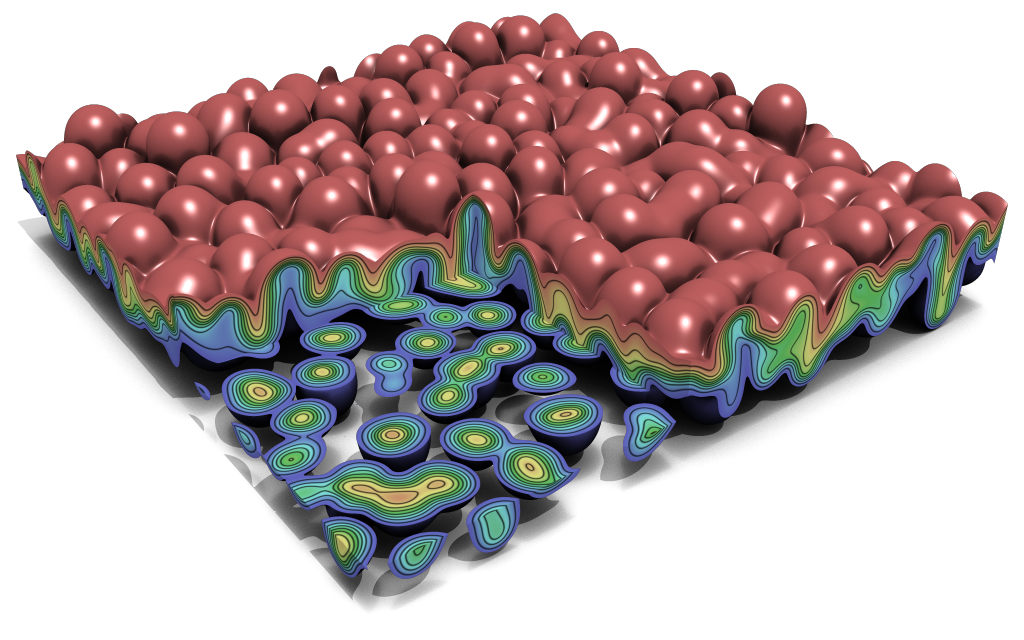}
        \end{subfigure}
        \begin{subfigure}[b]{0.32\textwidth}
        \includegraphics[width=\textwidth]{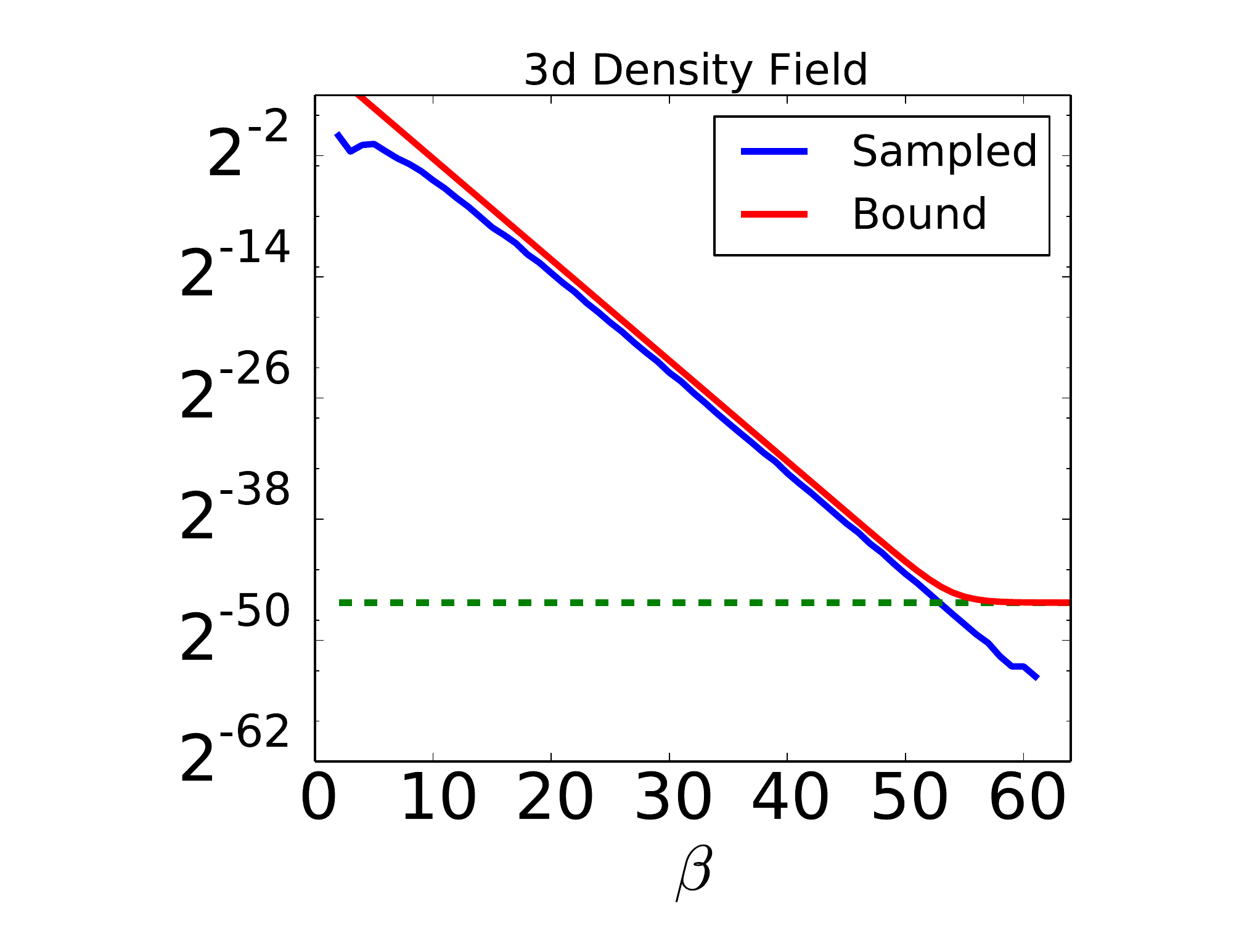}
        \end{subfigure}
             \begin{subfigure}[b]{0.32\textwidth}
        \includegraphics[width=\textwidth]{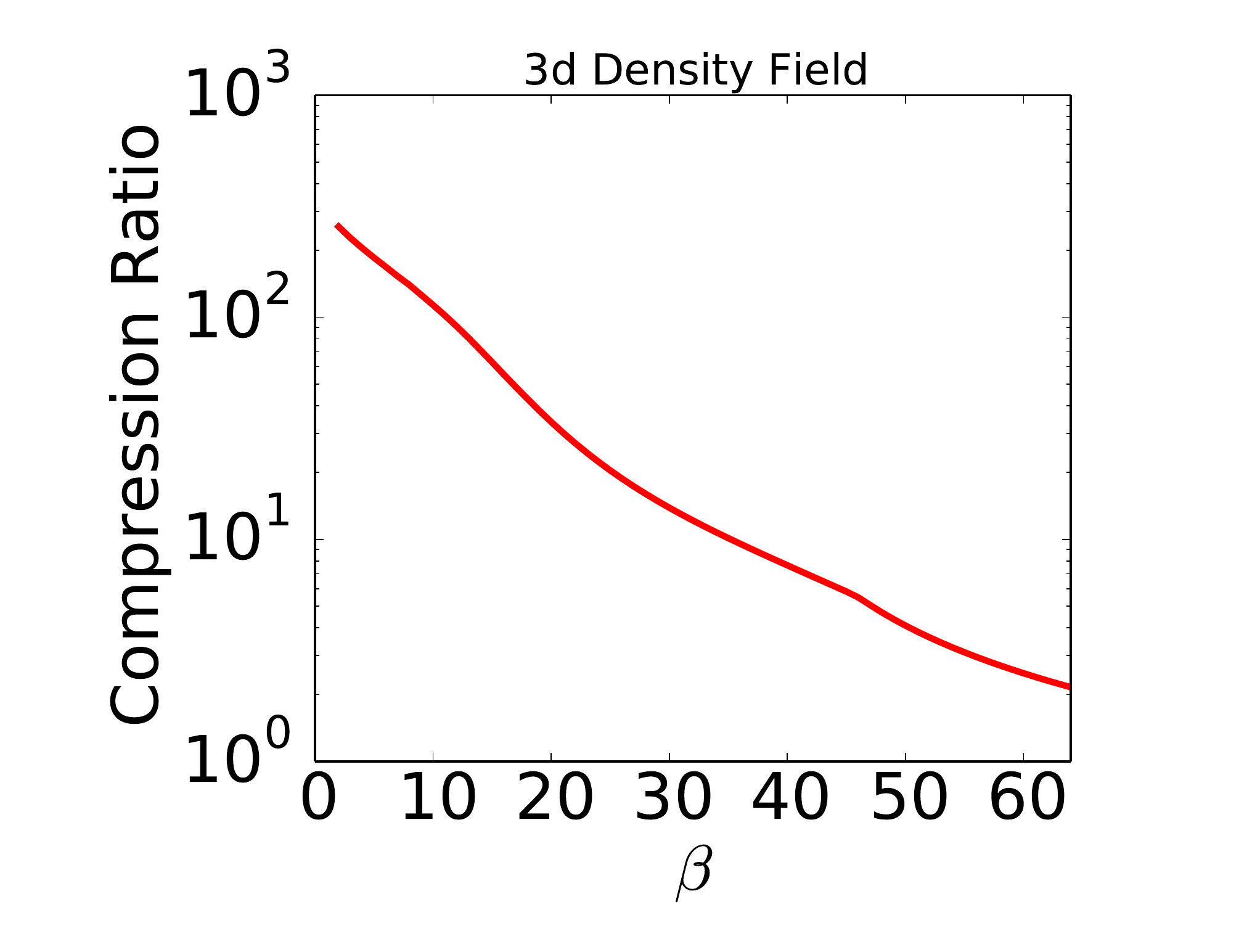}
        \end{subfigure}
            \caption{3-d density field example with double precision. On the left is a 3-d rendering of the density field. In the middle is the maximum block relative error as a function of the precision parameter $\beta$. The blue line is the error from ZFP compression and decompression with fixed $\beta$ and the red line depicts the theoretical bound. On the right is the compression ratio as a function of $\beta$.  }
                              \label{fig:density}

\end{figure}

%% file: conclusions.tex
\section{Conclusion}
%In this paper, we addressed the error introduced in the use of lossy
%compression of floating-point data. Our analysis provides tight theoretical bounds on the
%errors introduced by ZFP for the storage of
%solution results and look-up tables.  Given the trends in computing
%hardware, methods are needed to reduce the memory capacity and bandwidth
%demands in simulation codes.  One technique with promise, particularly
%for grid-based PDE methods, is the use of lossy floating-point
%compression. A extension of our work would include an error analysis of the propagation of the errors of storing the solution
%state in compressed format for an iterative method, i.e., repeatedly decompressing and
%recompressing the solution data at each time step or iteration of the
%numerical algorithm. 

In this paper, we addressed the error introduced in the use of lossy
compression of floating-point data. An important contribution of this paper is the formulation of the problem in
a way that simplifies analysis.
The vector space, $\cB^n$, introduced in
Section~\ref{sec:bitvectorSpace}, proved to be useful in developing
operators that accurately represent each step of the ZFP compression
algorithm. Section~\ref{sec:bounds} presented the error analysis of  the
current implementation of the fixed precision mode of ZFP and the
numerical tests presented in Section~\ref{sec:results} provided a
demonstration that the theoretical bounds established in this paper
capture the error introduced by ZFP. The techniques presented in
Section~\ref{IntroSec:seqspace} and methodology from Sections~\ref{sec:analysis}
and \ref{sec:bounds} could be applied to any compression algorithm or
numerical method involving direct manipulation of bits.  

In the majority of mesh-based PDE simulations, it is reasonable to
assume that most of the blocks 
provided as an input to the ZFP compression algorithm will be ``smooth"
in the sense that the exponent range, $e_{max}-e_{min}$, will be small
and there will be some natural ordering correlation between the values
within the block. As Theorem \ref{thm:diffDCandDC} represents the worst
possible error achieved, the tests presented in Section
\ref{sec:generatedBlock} were constructed to provide an exposition of
the worst case scenario (i.e., not ``smooth'') of the error introduced
by ZFP compression and decompression. Even so, the error bound
established in Theorem \ref{thm:diffDCandDC} accurately, and narrowly,
bounded the error in each example. 
%The same can be said for the
%experiment using the time-evolving method in Section~\ref{sec:laxwend}.  

We limited our detailed analysis to the fixed precision implementation of ZFP,
which is one of three possible compression modes implemented by ZFP. Using the bound in Theorem~\ref{thm:diffDCandDC}, we were able to provide a similar bound for the fixed accuracy and fixed rate modes. Further research is needed to determine a tighter error bound for the fixed
rate mode for real-world data as most real-world problems will produce ``smooth" values and the bound in Theorem~~\ref{thm:diffDCandDCFixedRate} is for the worst case scenario.

Additionally, given the trends in computing
hardware, methods are needed to reduce the memory capacity and bandwidth
demands in simulation codes.  \black{One common technique is to use mixed precision algorithms, which typically require changing the underlying algorithms to achieve the same purpose. However, one technique with promise, particularly for grid-based PDE methods, is the use of lossy floating-point compression. In \cite{zfp-doc}, the C++ compressed array primitives handle the complexity of decompression, caching, and compression transparently. By using ZFP instead, we can achieve bandwidth reduction without changing the underlying structure of the algorithm. }
An extension of our work would include an error analysis of the propagation of the errors of storing the solution
state in compressed format for an iterative method, i.e., repeatedly decompressing and
recompressing the solution data at each time step or iteration of the
numerical algorithm. 

%Additionally, it would be of interest to determine
%convergence results for various numerical methods when ZFP compression
%and decompression are implemented in order to reduce bandwidth
%usage. \red{Depending on the application or numerical method implemented
%  (e.g., diffusion or Jacobi relaxation)determining if the evolving
%  solution using ZFP converges to the true solution and, if so, at what
%  rate would be of significance, as this is a desired application for
%  ZFP compression.}
Lastly, it should be noted that this paper has
analyzed the error of converting an IEEE representation to ZFP. However,
ZFP can be seen as a number representation itself, just like
IEEE. Another interesting direction for future work would be to consider
the behavior of round-off error of floating-point arithmetic conducted directly on
the ZFP format.

%Understanding the direct
%manipulation of bits by ZFP and how, when bits are lost, it affects the
%accuracy of the solution is the novelty of this paper. 

% It should be noted, however, that the bound established here for the fixed-precision mode of ZFP does provide a bound for the fixed-accuracy and fixed-rate mode. Although the bound is not as tight as it could possibly be since these modes are not required to retain an entire bit-plane during the final stage of compression.

%% file: appendix.tex
\clearpage
\appendix
\section{ZFP Toy Example}
\label{sec:appendixa}
Here we include a toy example of ZFP compression and decompression as implemeted by the operators defined in Section~\ref{sec:analysis}. As the embedded coding implemented by ZFP in Step 7 is nontrivial and lossless we exclude it from the following example. More details on Step 7 can be found in the \emph{Algorithm} section of \cite{zfp-doc}. As our analysis is focused on the fixed precision mode of ZFP and it is the simplest mode to illustrate, we only present the output for fixed precision mode at Step 8. Note that we will write $x^{(i)}$ to denote the output from step $i$ of ZFP.

\noindent \textit{Compression}: We first outline the steps for compression on $x = [ 5632, \ 3072, \ 400, \ 68 ]^T \in \mathbb{R}^4$. \\
\noindent \textit{Step 1}: As $d = 1$ and the vector is already in $\mathbb{R}^4$ there is no partitioning to be done. \\
\noindent \textit{Step 2}: For simplicity, we will use $k = 13$ and $q = 9$. First, as outlined in Section \ref{Step2Sec}, the components are converted to a bit representation in $\cB$. Next, we apply the shift operator $S_{\ell}$ with $\ell = e_{\max} (x) - q + 1= 12 - 9 + 1 = 4$. This operation amounts to shifting each bit four positions to the right. Finally, we apply the truncation operator $T_{\mathcal{S}}$, where $\mathcal{S} = \{ i \in \mathbb{Z} : i \geq 0 \}$. Hence, any bits after the decimal are dropped which will result in the loss of some information in this example. This procedure is illustrated below in (A.1). Note that the representations in (A.1) include the guard bit so they will have $q + 1 = 10$ bits in their representation after the truncation phase. Additionally, there is one bit allotted for the sign bit which will not be represented below.
\begin{align}
\begin{array}{c<{\hspace{-1.3mm}} c<{\hspace{-0.5mm}} c<{\hspace{-0.5mm}} c<{\hspace{-0.5mm}} c<{\hspace{-0.5mm}} c<{\hspace{-0.5mm}} c}
\begin{bmatrix} 
5632 \\ 
3072 \\ 
400 \\ 
68 \end{bmatrix} &\longrightarrow &\begin{bmatrix} 
01011000000000 \\ 
00110000000000 \\ 
00000110010000 \\ 
00000001000100 \end{bmatrix} &\longrightarrow &\begin{bmatrix} 
0101100000.0000 \\ 
0011000000.0000 \\ 
0000011001.0000 \\ 
0000000100.0100 \end{bmatrix} &\longrightarrow &\begin{bmatrix} 
0101100000 \\ 
0011000000 \\ 
0000011001 \\ 
0000000100 \end{bmatrix} \vspace{2mm} \\
\text{Decimal} & & \text{Signed Binary} & & \text{Bit Shift} & & x^{(2)}
\end{array} \tag{A.1} 
\end{align} 
In addition to the bits used to store $x^{(2)}$, note that ZFP also encodes the value $e_{\max} (x) = 12$. \\
\noindent \textit{Step 3}: Using Table~\ref{table:actionTlossy} from Section \ref{Step3Sec}, with $\ba_1 = 0101100000$, $\ba_2 = 0011000000$, $\ba_3 = 0000011001$, and $\ba_4 = 0000000100$, we can compute $x^{(3)}$. This process yields the vector %$[ 143, \ 120, \ -35, \ 19 ]^T$ 
\begin{align*}
x^{(3)} = \begin{bmatrix*}[r]
0010001111 \\
0001111000 \\
-0000100011 \\
0000010011
\end{bmatrix*}
\end{align*}
and is outlined step by step below (steps work from left to right starting with upper left entry): 
\begin{center}
	\begin{adjustbox}{width=0.9\textwidth} 
		\begin{tabular}{ | l<{\hspace{-3mm}} l<{\hspace{-2mm}} r<{\hspace{2mm}} l<{\hspace{-3mm}} l<{\hspace{-2mm}} r<{\hspace{2mm}} l<{\hspace{-3mm}} l<{\hspace{-2mm}} r | } 
			\hline
			\multicolumn{9}{| c |}{\vspace*{-4mm}} \\
			\multicolumn{9}{| c |}{$\tilde{L}$} \\ \hline \hline
			\vspace*{-4mm} & & & & & & & & \\
			$\ba_1 \leftarrow \ba_1 + \ba_4 $ &$=$  &$0101100100$,     & $\ba_1 \leftarrow r(\ba_1) $ &$=$  &$0010110010$,      & $\ba_4 \leftarrow \ba_4 - \ba_1 $ &$=$ &$-0010101110$, \\ \hline
			$\ba_3 \leftarrow \ba_3 + \ba_2 $ &$=$  &$0011011001$,      & $\ba_3 \leftarrow r(\ba_3) $ &$=$  &$0001101100$,      & $\ba_2 \leftarrow \ba_2 - \ba_3 $ &$=$  &$0001010100$, \\ \hline
			$\ba_1 \leftarrow \ba_1 + \ba_3 $ &$=$  &$0100011110$,      & $\ba_1 \leftarrow r(\ba_1) $ &$=$  &$0010001111$,      & $\ba_3 \leftarrow \ba_3 - \ba_1 $ &$=$  &$-0000100011$, \\ \hline
			$\ba_4 \leftarrow \ba_4 + \ba_2 $ &$=$  &$-0001011010$,     & $\ba_4 \leftarrow r(\ba_4) $ &$=$  &$-0000101101$,      & $\ba_2 \leftarrow \ba_2 - \ba_4 $ &$=$  &$0010000001$, \\ \hline
			$\ba_4 \leftarrow \ba_4 + r(\ba_2) $ &$=$  &$0000010011$,    & $\ba_2  \leftarrow \ba_2 - r(\ba_4) $ &$=$  &$0001111000$. & & & \\
			\hline
		\end{tabular}
	\end{adjustbox}
\end{center}
\vspace{3mm}
\noindent \textit{Step 4}: Note that the components of a vector in $\mathbb{R}^4$ are already in total sequency order. \\
\noindent \textit{Step 5}: The components are converted to a negabinary representation. Note that $q+2 = 11$ bits are available for the negabinary representation as it does not require a dedicated sign bit. Hence, 
\begin{align*}
\begin{array}{c<{\hspace{-2mm}} c<{\hspace{-4mm}} c}
x^{(5)} = &\begin{bmatrix}
00110010011 \\
00110001000 \\
00000101101 \\
00000010111
\end{bmatrix} &. \vspace{1.2mm} \\
&\text{Negabinary} &
\end{array}
\end{align*}
\noindent \textit{Step 6}: The transposition is performed so that the first row corresponds to the most significant bit while the last row corresponds to the least significant bit. Hence $x^{(6)} = (x^{(5)})^T$. %which is given by 
\iffalse
\begin{align*}
x^{(6)} = \begin{bmatrix}
0000 \\
0000 \\
1100 \\
1100 \\
0000 \\
0010 \\
1001 \\
0110 \\
0011 \\
1001 \\
1011
\end{bmatrix}
\end{align*}
\fi
Note that the terminology \emph{bit plane} used in the discussion of ZFP can be realized as the rows of bits in $x^{(6)}$ or the columns of bits of $x^{(5)}$. For example, the third bit plane of $x^{(6)}$ is $1100$ and the sixth bit plane of $x^{(6)}$ is $0010$. \\
\noindent \textit{Step 8}: Recall that we have decided to exclude Step 7 from this example for simplicity. In Step 8, the user provides a fixed number of bit planes, $\beta$, to keep from $x^{(6)}$ starting with the most significant bit plane. The remaining bit planes are then ordered in a sequence as the rows of $x^{(6)}$. Supposing the number of bit planes to keep is $\beta = 7$, the output would be $x^{(8)} = 0000 \ 0000 \ 1100 \ 1100 \ 0000 \ 0010 \ 1001$.
\iffalse
\begin{align*}
\begin{array}{rcl}
x^{(8)} = \hspace*{-2.3mm} & 0000 \ 0000 \ 1100 \ 1100 \ 0000 \ 0010 \ 1011 & \hspace*{-3mm} \text{.}
\vspace{2mm} \\
& \text{Bit Sequence} &
\end{array}
\end{align*}
\fi
\iffalse
\begin{align*}
\begin{array}{rcl}
x^{(8)} = \hspace*{-4mm} &\begin{bmatrix}
0011001 \\
0011000 \\
0000010 \\
0000001
\end{bmatrix} & \hspace*{-5mm} \text{.}
\vspace{2mm} \\
& \text{Bit Sequence} &
\end{array}
\end{align*}
\fi
It should be noted that $x^{(8)}$ is close to the compressed bit stream format of ZFP but is not in the exact compressed format since we did not perform Step 7. Note that if Step 7 had been performed then the final compressed bit sequence would have been $0 \ 0 \ 11110 \ 110 \ 000 \ 00110 \ 1011$, requiring 6 fewer bits. This completes the steps for compression. \\

\noindent \textit{Decompression}: We now highlight the steps for decompression. To decompress, we first convert the bit sequence back into a vector format and transpose. We then place zeros to the end of each row until we have the same number of bits before we dropped bit planes (in this case four zeros per row). This vector will be very similar to $x^{(5)}$, however, due to the removal of bit planes in Step 8 some information was lost that cannot be restored. Next, we convert each negabinary representation to a signed binary representation. These steps are illustrated in (A.2).
\begin{align}
\begin{array}{ccccc}
\begin{bmatrix}
0011001 \\
0011000 \\
0000010 \\
0000001
\end{bmatrix} &\longrightarrow &\begin{bmatrix}
00110010000 \\
00110000000 \\
00000100000 \\
00000010000
\end{bmatrix} &\longrightarrow &\begin{bmatrix*}[r]
0010010000 \\  %144
0010000000 \\  %128
-0000100000 \\  %-32
0000010000      %16
\end{bmatrix*} %&\longrightarrow &\begin{bmatrix*}[r]
%144 \\
%128 \\
%-32 \\
%16
%\end{bmatrix*} 
\vspace{2mm} \\
%\text{Bit Sequence} 
& & \text{Negabinary} & & \text{Signed Binary} %& & \text{Decimal}
\end{array}
\tag{A.2}
\end{align}
We now use the final vector from (A.2) as the input for the routine outlined in Table~\ref{table:actionTlossy} in Section \ref{Step3Sec}. 
\begin{center}
	\begin{adjustbox}{width=0.9\textwidth} 
		\begin{tabular}{ | l<{\hspace{-3mm}} l<{\hspace{-2mm}} r<{\hspace{2mm}} l<{\hspace{-3mm}} l<{\hspace{-2mm}} r<{\hspace{2mm}} l<{\hspace{-3mm}} l<{\hspace{-2mm}} r | } 
			\hline
			\multicolumn{9}{| c |}{\vspace*{-4mm}} \\
			\multicolumn{9}{| c |}{$\tilde{L}^{-1}$} \\ \hline \hline
			\vspace*{-4mm} & & & & & & & & \\
			$\ba_2 \leftarrow \ba_2 + r(\ba_4) $ &$=$  &$0010001000$,  & $\ba_4 \leftarrow \ba_4 - r(\ba_2) $ &$=$  &$-0000110100$, & & & \\ \hline
			$\ba_2 \leftarrow \ba_2 + \ba_4 $ &$=$  &$0001010100$,   & $\ba_4 \leftarrow s_{-1}(\ba_4) $ &$=$  &$-0001101000$,  & $\ba_4 \leftarrow \ba_4 - \ba_2 $ &$=$  &$-0010111100$, \\ \hline
			$\ba_3 \leftarrow \ba_3 + \ba_1 $ &$=$  &$0001110000$,   & $\ba_1 \leftarrow s_{-1}(\ba_1) $ &$=$  &$0100100000$,  & $\ba_1 \leftarrow \ba_1 - \ba_3 $ &$=$  &$0010110000$, \\ \hline
			$\ba_2 \leftarrow \ba_2 + \ba_3 $ &$=$  &$0011000100$,   & $\ba_3 \leftarrow s_{-1}(\ba_3) $ &$=$  &$0011100000$,  & $\ba_3 \leftarrow \ba_3 - \ba_2 $ &$=$  &$0000011100$, \\ \hline
			$\ba_4 \leftarrow \ba_4 + \ba_1 $ &$=$  &$-0000001100$,   & $\ba_1 \leftarrow s_{-1}(\ba_1) $ &$=$  &$0101100000$,  & $\ba_1 \leftarrow \ba_1 - \ba_4 $ &$=$  &$0101101100$. \\
			\hline
		\end{tabular}
	\end{adjustbox}
\end{center}
\vspace{4mm}
Lastly, we perform a bit shift of four bits to the left to undo the shift performed during Step 2 and convert to decimal to yield the decompressed vector. This procedure is illustrated below in (A.3).
\begin{align}
\begin{array}{ccccc}
\begin{bmatrix*}[r]
0101101100.0000 \\
0011000100.0000 \\
0000011100.0000 \\
-0000001100.0000
\end{bmatrix*} &\longrightarrow &\begin{bmatrix*}[r]
01011011000000 \\
00110001000000 \\
00000111000000 \\
-00000011000000
\end{bmatrix*} &\longrightarrow &\begin{bmatrix*}[r]
5824 \\
3136 \\
448 \\
-192
\end{bmatrix*}
\vspace{2mm} \\
\text{Signed Binary} & & \text{Bit Shift} & & \text{Decimal}
\end{array}
\tag{A.3}
\end{align}
We conclude by comparing the error with the bound in Theorem~\ref{thm:diffDCandDC}. Since $d = 1$, $k = 13$, $q = 9$, and $\beta = 7$ for this example, we have that $K_{\beta} \approx 0.19831$. Since $K_{\beta} \| x \|_\infty \leq (0.19832) (5632) \leq 1117$ and $\| \tilde{D}\tilde{C} x - x \|_\infty = 260$ we observe that the bound established in Theorem~\ref{thm:diffDCandDC} holds. %Note that since $e_{\max}(x) - e_{\min}(x) = 12 - 2 = 10$ and $\beta = 7 < 10$ it should be expected that the relative error is on the order of $2^{10} \epsilon_{\beta} = 2^4$.

\section{Round-off Error of the Lossy Decorrelating Backwards Linear Transform}
	\label{sec:appendixb}
	There are three cases that must be considered: $(i)$ $\beta = q+2,$ $(ii)$  $q- 2d +2<\beta < q+2$, and $(iii)$ $\beta \leq  q- 2d+2$. We will investigate each separately in the following sections. 
	\begin{itemize}
		\item[$(i)$]
		When $\beta=q+2$, no information is lost at Step 8. Thus, from Table~\ref{table:actionTlossy} we observe that the last two steps for the forward transform are exactly reversed by the first two steps of the backwards transform, 
		$$\begin{array}{l<{\hspace{-4mm}} r l<{\hspace{-2mm}} l<{\hspace{-2mm}} l}
		\tilde{L} &\text{Step 13}: & \ba_4 \leftarrow \ba_4 + r(\ba_2) & \Rightarrow \ba_4 = \ba_4 + r(\ba_2),&\\
		\tilde{L} &\text{Step 14}: & \ba_2  \leftarrow \ba_2 - r(\ba_4)  & \Rightarrow \ba_2 =\ba_2-r(\ba_4 + r(\ba_2)), &\\ 
		\tilde{L}^{-1} &\text{Step 1}: & \ba_2 \leftarrow \ba_2 + r(\ba_4) & \Rightarrow \ba_2 =\ba_2-r(\ba_4 + r(\ba_2)) +r(\ba_4 + r(\ba_2) )&= \ba_2,\\ 
		\tilde{L}^{-1} &\text{Step 2}: & \ba_4 \leftarrow \ba_4- r(\ba_2) & \Rightarrow \ba_4 = (\ba_4 + r(\ba_2)) - r(\ba_2) &= \ba_4. \\
		\end{array}$$
		Thus, no error occurs by applying $\tilde{L}_d^{-1}$. 
	
		\item[$(ii)$]
		If $q- 2d +2<\beta < q+2 $, then using similar techniques as for the forward linear transform operator, a bound can be found for the lossy backwards linear transform operator.
		\begin{lemma}
			\label{LemmaLingBound}
			Suppose $\bx \in \mathbb{Z}^4$ such that $e_{max}(\bx) = q - 1$, $\bx \neq \bfz$ and $q- 2d +2<\beta < q+2 $. Given the bit arithmetic implementation in Table \ref{table:actionT} and Table \ref{table:actionTlossy} for ZFP's backwards linear transforms, we have
			\begin{align*}
			\|\cL^{-1} \bx - \tilde{\cL}^{-1} \bx\|_\infty \leq \frac{5}{2} \epsilon_q \|\bx\|_\infty \quad and \quad \|\cL^{-1}_d\bx - \tilde{\cL}^{-1}_d \bx \|_\infty & \leq k_{\cL^{-1}}\epsilon_q \|\bx\|_\infty, 
			\end{align*}
			where $k_{\cL^{-1}} = \frac{5}{2} \left( 2^d - 1 \right)$. 
		\end{lemma}
		\begin{proof}
			Use outline of proof from Lemma \ref{LemmaL1DBound} and \ref{lemma:boundT}. 
		\end{proof}
		\item[$(iii)$]
	If $\beta \leq q - 2d+2$, then the rightmost $2d$ least significant bits of $\ba_i$ are zero, for all $i$, resulting in the following equivalences for the first two steps from Table~\ref{table:actionTlossy}: $r(\ba_4) = s_{-1}(\ba_4)$, $ r(\ba_2) = s_{-1}(\ba_2)$, and $r(\ba_4 + r(\ba_2))  = s_{-1}(\ba_4 + s_{-1}(\ba_2)) $. Thus, the lossy backwards transform operator is exactly the lossless version, resulting in no additional error. 
	\end{itemize}
	If $q- 2d +2<\beta < q+2 $, only a modest reduction of the data will be achieved over using $\beta =  q- 2d+2$. Thus, for the analysis of this paper, we chose to assume $\beta \leq q - 2d+2$ so that
	\[ \tilde{D}_3(\ba) = {D}_3(\ba)= F_\cB^{-1}{\cL}_d^{-1}F_\cB(\ba), \text{ for all } \bm{a} \in \mathcal{B}^{4^d}. \]
	Note, Theorem \ref{thm:diffDCandDC} can be modified to accommodate the error that occurs from the lossy backwards decorrelating operator. 
	\begin{theorem}\label{thm:diffDCandDCappendix} 
		Assume $\bx \in \mathbb{R}^{4^d}$ with $\bm{x} \neq \bm{0}$ such that $F_{\mathcal{B}} (\bx) \in \cB_k^{4^d}$, for some precision $k$. Let $q- 2d +2<\beta < q+2 $ be the fixed precision parameter. Then 
		\begin{align}
		\| \tilde{D}\tilde{C} \bx - \bx \|_\infty &\leq B_\beta \|\bx\|_\infty
		\end{align}
		where $q \in \mathbb{N}$ is the precision for the block-floating point representation in Step 2,
		\begin{align}
		B_\beta :=k_{\cL^{-1}} \epsilon_q(1+\epsilon_{k})\left( \frac{8}{3} \epsilon_\beta+ \epsilon_q \left(1+ \frac{8}{3} \epsilon_\beta \right) \left(k_\cL(1+\epsilon_q)+1 \right )\right) +K_\beta,
		\end{align}
		$k_\cL= \frac{7}{4} (2^d-1)$, and $k_{\cL^{-1}} = \frac{5}{2} \left( 2^d - 1 \right)$.
	\end{theorem}\begin{proof}
	Use outline of proof from Theorem \ref{thm:diffDCandDC} along with Lemma \ref{LemmaLingBound}. 
\end{proof}
\section{Discussion of error bound constant $K_{\beta}$}
\label{sec:appendixc}
First, note that $K_{\beta}$ is a function of $q$, $k$, $d$, and $\beta$. Since $k$ depends on the precision of the data provided to ZFP and $d$ is the dimension of the input data it is important to note that two of the variables used in computing $K_{\beta}$ are dependent on the input data and cannot be changed by the user in ZFP. The value of $q$ depends on the precision of data and is set to a value larger than $k$. For example,  if the input values are IEEE single or double precision, $q \in \{30, 62\}$, since one bit is used to represent the sign bit and another to represent the overflow guard bit, as discussed in Section \ref{Step2Sec}. The remaining variable, $\beta$, can be set to any positive integer when using the fixed precision mode of ZFP, as noted in Section \ref{Step8Sec}. Figure~\ref{fig:KBetaPlot} helps illustrate how $K_{\beta}$ varies with respect to $\beta$ and the dimensionality of the data, $d$. The lines on the contour plot in Figure~\ref{fig:KBetaPlot} represent the log base 10 value of $K_{\beta}$, i.e. $\log_{10} (K_\beta)$. As suspected from the formula for $K_{\beta}$, we observe that a larger value of $d$ has a greater effect on the value of $K_{\beta}$ for small values of $\beta$.
%\begin{figure}[h!]
%	\centering
%%	\includegraphics[width=\linewidth]{figures/KBPlotMedResNoTitle.png}
%	\includegraphics[width=.45\linewidth]{figures/contour_kbeta.pdf}
%	\caption{\Visualization of $K_{\beta}$ for $\beta \in [1, 32]$ and dimension $d \in [2, 5]$ with \blue{$k = 52$} and $q = 62$.}
%	\label{fig:KBetaPlot}
%\end{figure}}

\begin{figure}[h!]
	\centering
	\includegraphics[width=.45\linewidth]{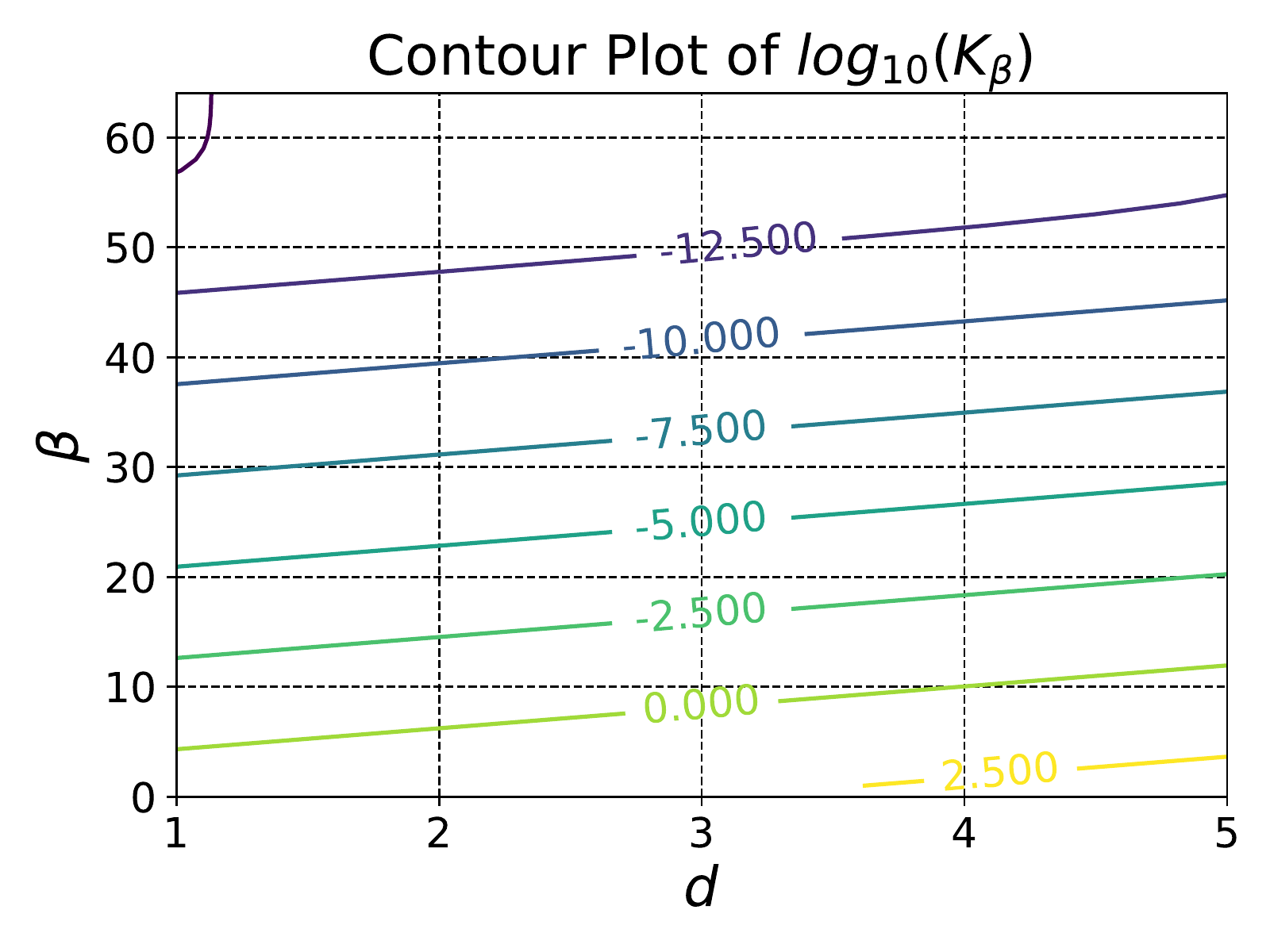}
	\caption{Contour plot of $\log_{10} (K_\beta)$ for $\beta \in [1, 64]$ and dimension $d \in [1, 5]$ with $k = 53$ and $q = 62$.}
	\label{fig:KBetaPlot}
\end{figure}

%% file: zfp_stability_paper.bbl
\begin{thebibliography}{10}

\bibitem{XDMAV2011}
{\sc S.~Ahern, A.~Shoshani, K.-L. Ma, A.~Choudhary, T.~Critchlow, S.~Klasky,
  V.~Pascucci, J.~Ahrens, E.~W. Bethel, H.~Childs, J.~Huang, K.~Joy, Q.~Koziol,
  G.~Lofstead, J.~Merifith, K.~Moreland, G.~Ostrouchov, M.~Papka,
  V.~Vishwanath, M.~Wolf, N.~Wright, and K.~Wu}, {\em {S}cientific {D}iscovery
  at the {E}xascale: {R}eport from the {DOE ASCR} 2011 {W}orkshop on {E}xascale
  {D}ata {M}anagement, {A}nalysis, and {V}isualization}, tech. report, U.S.
  Department of Energy, Feb. 2011.

\bibitem{gmd-9-4381-2016}
{\sc A.~H. Baker, D.~M. Hammerling, S.~A. Mickelson, H.~Xu, M.~B. Stolpe,
  P.~Naveau, B.~Sanderson, I.~Ebert-Uphoff, S.~Samarasinghe, F.~De~Simone,
  F.~Carbone, C.~N. Gencarelli, J.~M. Dennis, J.~E. Kay, and P.~Lindstrom},
  {\em Evaluating lossy data compression on climate simulation data within a
  large ensemble}, Geoscientific Model Development, 9 (2016), pp.~4381--4403,
  \url{https://doi.org/10.5194/gmd-9-4381-2016},
  \url{https://www.geosci-model-dev.net/9/4381/2016/}.

\bibitem{BorkarChien2011}
{\sc S.~Borkar and A.~A. Chien}, {\em The future of microporcessors}, CACM, 54
  (2011), pp.~67--77, \url{https://doi.org/10.1145/1941487.1941507}.

\bibitem{Xcutting2010}
{\sc D.~L. Brown, P.~Messina, D.~Keyes, J.~Morrison, R.~Lucas, J.~Shalf,
  P.~Beckman, R.~Brightwell, A.~Geist, J.~Vetter, B.~L. Chamberlain, E.~Lusk,
  J.~Bell, M.~S. shephard, M.~Anitescu, D.~Estep, B.~Hendrickson, A.~Pinar, and
  M.~A. Heroux}, {\em Scientific grand challenges: Crosscutting technologies
  for computing at the exascale}, tech. report, U.S. Department of Energy, Feb.
  2010.

\bibitem{viscosity}
{\sc W.~H. Cabot and A.~W. Cook}, {\em {R}eynolds number effects on
  {R}ayleigh-{T}aylor instability with possible implications for {T}ype {I}a
  supernovae}, Nature Physics, 2 (2006), pp.~562 EP --,
  \url{http://dx.doi.org/10.1038/nphys361}.

\bibitem{752522}
{\sc R.~H. Dennard, F.~H. Gaensslen, H.~{N}ien Yu, V.~L. Rideout, E.~Bassous,
  and A.~R. Leblanc}, {\em Design of ion-implanted {MOSFET}'s with very small
  physical dimensions}, Proceedings of the {IEEE}, 87 (1999), pp.~668--678,
  \url{https://doi.org/10.1109/JPROC.1999.752522}.

\bibitem{deflate}
{\sc L.~P. Deutsch}, {\em Deflate compressed data format specification version
  1.3}, May 1996, \url{https://tools.ietf.org/html/rfc1951#section-Abstract}
  (accessed 2017-10-25).

\bibitem{sz}
{\sc S.~Di and F.~Cappello}, {\em Fast error-bounded lossy {HPC} data
  compression with {SZ}}, in 2016 IEEE International Parallel and Distributed
  Processing Symposium (IPDPS), May 2016, pp.~730--739,
  \url{https://doi.org/10.1109/IPDPS.2016.11}.

\bibitem{higham2002accuracy}
{\sc N.~Higham}, {\em Accuracy and Stability of Numerical Algorithms: Second
  Edition}, EngineeringPro collection, Society for Industrial and Applied
  Mathematics (SIAM, 3600 Market Street, Floor 6, Philadelphia, PA 19104),
  2002, \url{https://books.google.com/books?id=7J52J4GrsJkC}.

\bibitem{katz91}
{\sc P.~W. Katz}, {\em String searcher, and compressor using same}, Sept. 1991,
  \url{https://www.lens.org/lens/patent/US_5051745_A}.

\bibitem{Knuth}
{\sc D.~E. Knuth}, {\em The Art of Computer Programming, Volume 2 (3rd Ed.):
  Seminumerical Algorithms}, Addison-Wesley Longman Publishing Co., Inc.,
  Boston, MA, USA, 1997.

\bibitem{tensor}
{\sc P.~Lancaster and H.~K. Farahat}, {\em Norms on direct sums and tensor
  products}, Mathematics of Computation, 26 (1972), pp.~401--414,
  \url{http://www.jstor.org/stable/2005167}.

\bibitem{Laney:2013:AED:2503210.2503283}
{\sc D.~Laney, S.~Langer, C.~Weber, P.~Lindstrom, and A.~Wegener}, {\em
  Assessing the effects of data compression in simulations using physically
  motivated metrics}, in Proceedings of the International Conference on High
  Performance Computing, Networking, Storage and Analysis, SC '13, New York,
  NY, USA, 2013, ACM, pp.~76:1--76:12,
  \url{https://doi.org/10.1145/2503210.2503283},
  \url{http://doi.acm.org/10.1145/2503210.2503283}.

\bibitem{zfp}
{\sc P.~Lindstrom}, {\em Fixed-rate compressed floating-point arrays}, IEEE
  Transactions on Visualization and Computer Graphics, 20 (2014),
  pp.~2674--2683, \url{https://doi.org/10.1109/TVCG.2014.2346458}.

\bibitem{Lindstrom2017}
{\sc P.~Lindstrom}, {\em Error distributions of lossy floating-point
  compressors}, JSM Proceedings,  (2017), pp.~2574--2589.

\bibitem{zfp-doc}
{\sc P.~Lindstrom}, {\em {ZFP} version 0.5.3}, April 2018.
\newblock https://zfp.readthedocs.io/en/release0.5.3/index.html.

\bibitem{fpzip}
{\sc P.~Lindstrom and M.~Isenburg}, {\em Fast and efficient compression of
  floating-point data}, {IEEE} Transactions on Visualization and Computer
  Graphics, 12 (2006), pp.~1245--1250,
  \url{https://doi.org/10.1109/TVCG.2006.143}.

\bibitem{MitrablockRoundingError}
{\sc A.~Mitra}, {\em On finite wordlength properties of block-floating-point
  arithmetic}, International Journal of Electrical, Computer, Energetic,
  Electronic and Communication Engineering, 2 (2008).

\bibitem{Rao1990}
{\sc K.~R. Rao and P.~Yip}, {\em Discrete Cosine Transform: Algorithms,
  Advantages, Applications}, Academic Press Professional, Inc., San Diego, CA,
  USA, 1990.

\bibitem{Ratanaworabhan:2006:FLC:1126009.1126035}
{\sc P.~Ratanaworabhan, J.~Ke, and M.~Burtscher}, {\em Fast lossless
  compression of scientific floating-point data}, in Proceedings of the Data
  Compression Conference, DCC '06, Washington, DC, USA, 2006, IEEE Computer
  Society, pp.~133--142, \url{https://doi.org/10.1109/DCC.2006.35},
  \url{https://doi.org/10.1109/DCC.2006.35}.

\bibitem{1659158}
{\sc T.~A. Welch}, {\em A technique for high-performance data compression},
  Computer, 17 (1984), pp.~8--19,
  \url{https://doi.org/10.1109/MC.1984.1659158}.

\bibitem{Williams:2009:RIV:1498765.1498785}
{\sc S.~Williams, A.~Waterman, and D.~Patterson}, {\em Roofline: An insightful
  visual performance model for multicore architectures}, Commun. ACM, 52
  (2009), pp.~65--76, \url{https://doi.org/10.1145/1498765.1498785},
  \url{http://doi.acm.org/10.1145/1498765.1498785}.

\bibitem{1055714}
{\sc J.~Ziv and A.~Lempel}, {\em A universal algorithm for sequential data
  compression}, IEEE Transactions on Information Theory, 23 (1977),
  pp.~337--343, \url{https://doi.org/10.1109/TIT.1977.1055714}.

\bibitem{1055934}
{\sc J.~Ziv and A.~Lempel}, {\em Compression of individual sequences via
  variable-rate coding}, IEEE Transactions on Information Theory, 24 (1978),
  pp.~530--536, \url{https://doi.org/10.1109/TIT.1978.1055934}.

\end{thebibliography}
